\documentclass[12pt]{amsart}
\usepackage[letterpaper, margin=1.25in]{geometry}
\usepackage{amsmath,amsthm,amscd,color}
\usepackage{amssymb}
\usepackage{amsfonts}
\usepackage{latexsym}
\usepackage{mathrsfs}
\usepackage{hyperref}
\usepackage{geometry}
\usepackage{fancyvrb}
\usepackage{dsfont}
\usepackage[dvipsnames]{xcolor}
\usepackage{tikz-cd}
\usepackage{rotating}
\usepackage{graphicx,diagbox}
\usepackage{caption}
\usepackage{subcaption}
\usepackage{MnSymbol}
\usepackage{stmaryrd}

\newtheorem{theorem}[equation]{Theorem}
\newtheorem{lemma}[equation]{Lemma}
\newtheorem{proposition}[equation]{Proposition}
\newtheorem{corollary}[equation]{Corollary}

\theoremstyle{remark}
\newtheorem{definition}[equation]{Definition}

\newtheorem*{remark}{\bf Remark}

\numberwithin{equation}{section}
\numberwithin{table}{section}

\newcommand{\wh}{\widehat}

\DeclareMathOperator{\CKh}{CKh}

\DeclareMathOperator{\Br}{Br}
\DeclareMathOperator{\APS}{APS}
\DeclareMathOperator{\DEC}{DEC}
\DeclareMathOperator{\ACT}{ACT}
\DeclareMathOperator{\CIR}{CIR}
\DeclareMathOperator{\INT}{INT}
\DeclareMathOperator{\FREE}{FREE}
\DeclareMathOperator{\MATCH}{MATCH}
\DeclareMathOperator{\BRIDGE}{BRIDGE}
\DeclareMathOperator{\cl}{cl}
\DeclareMathOperator{\CR}{CR}
\DeclareMathOperator{\RES}{RES}
\DeclareMathOperator{\STATE}{STATE}
\DeclareMathOperator{\Kom}{Kom}
\DeclareMathOperator{\Mat}{Mat}
\DeclareMathOperator{\Mod}{Mod}

\newcommand{\Z}{{\mathbb{Z}}}
\newcommand{\R}{{\mathbb{R}}}

\newcommand{\F}{{\mathbb{F}}}
\newcommand{\U}{{\mathbb{H}}}

\newcommand{\I}{{\mathbb{I}}}
\newcommand{\BB}{{\mathcal{B}}}

\newcommand{\TT}{{\mathcal{T}}}
\newcommand{\FF}{{\mathcal{F}}}

\newcommand{\II}{{\mathcal{I}}}
\newcommand{\CC}{{\mathcal{C}}}
\newcommand{\LL}{{\mathcal{L}}}
\newcommand{\QQ}{{\mathcal{Q}}}

\newcommand{\Hom}{\operatorname{Hom}}

\title{Khovanov Homology in Connected Sums}
\author{Alan Du}

\begin{document}

\maketitle

\begin{abstract}
Khovanov homology is an invariant for links in the three sphere that categorizes the Jones polynomial. We extend Khovanov's construction to links in 3-manifolds that are connected sums of orientable interval bundles over surfaces. Cutting the 3-manifold along a separating sphere, we construct type D and type A structures that are invariants of tangles in the two halves following the work of Roberts. Gluing the type D and type A structures along the common boundary recovers the Khovanov homology of the link.
\end{abstract}

\section{Introduction}

Khovanov homology is an invariant for links in $S^3$ in the form of a bigraded chain complex. The polynomial Euler characteristic of Khovanov homology is the Jones polynomial \cite{khovanov-categorification}. Asaeda, Przytycki, and Sikora defined homology invariants of links in all orientable interval-bundles over surfaces that categorify the Kauffman bracket skein module \cite{aps-i-bundles}. Their homology groups were defined with $\Z$-coefficients for bundles over surfaces $F$ for $F\ne\R P^2$, and with $\Z/2\Z$-coefficients for bundles over $\R P^2$. In particular, their construction works for links in $\R P^3$ by viewing $\R P^3$ with a point removed as the twisted $I$-bundle over $\R P^2$. Later, Gabrov\v sek extended the definition for links in $\R P^3$ with $\Z$-coefficients \cite{gabrovsek-rp3}. Khovanov homology was also defined for links in $S^2\times S^1$ by Rozansky and for links in $\#^r(S^2\times S^1)$ by Willis \cite{rozansky-categorification} \cite{willis-kh}.

In this paper, we extend Khovanov homology to links in $3$-manifolds of the form $M=M_1\#\dots\# M_r$, where $M_i$ is an orientable $I$-bundle over a surface $F_i$ for $i=1,\dots,r$. We show how to define a diagram $L$ for a link $\LL\subset M$. Then, we construct a bigraded chain complex $CKh(L)$ whose graded euler characteristic is the Kauffman bracket polynomial of $L$.

\begin{theorem}
The bigraded chain homotopy type of $CKh(L)$ is an invariant of the link $\LL$ under isotopy.
\end{theorem}

As a special case of our construction, we obtain a definition of Khovanov homology for links in $\#^r\R P^3$, where $\R P^3$ with a point removed is viewed as the twisted interval bundle over $\R P^2$.

Our Khovanov complex is constructed by passing through an analogue of Bar-Natan's cobordism category \cite{bar-natan}. That is, there is a chain complex $BN(L)$ whose objects are formal crossingless diagrams and whose morphisms are cobordisms, and we obtain $CKh(L)$ by applying a Khovanov TQFT functor to $BN(L)$. 

The main novelty of our approach is our proof that $CKh(L)$ is invariant under handleslide moves. Roughly speaking, a handleslide involves taking an arc of the link and wrapping it over or under a separating sphere for the connected sum $M=M_1\#\dots\# M_r$. On the diagram level, this creates $2n$ new crossings if the geometric intersection number between the link and the separating sphere is $2n$. 

We show that these $2n$ new crossings can be cancelled out by a sequence of moves inducing chain homotopy equivalences on the $CKh$ level. However, some of the moves we consider are unable to be turned into cobordisms, and hence, we do not know whether or not the chain homotopy type of $BN(L)$ is a link invariant.

For tangles $\TT$ in $D^3$, Lawrence Roberts defined bigraded type D and type A structures $\llbracket \TT\rrangle$ and $\llangle \TT\rrbracket$ that are invariants of the tangles. Furthermore, given two tangles $\overleftarrow{\TT_1}$ and $\overrightarrow{\TT_2}$ with the same number of endpoints on $\partial D^3$, one can glue the two tangles along their common boundary to obtain a link $\LL=\overleftarrow{\TT_1}\natural\overrightarrow{\TT_2}$ in $S^3$. Roberts showed that the box tensor product of $\llangle\overleftarrow{\TT_1}\rrbracket$ and $\llbracket\overrightarrow{\TT_2}\rrangle$ recovers the Khovanov homology of $\LL$ \cite{roberts-type-d}\cite{roberts-type-a}. 

We provide an extension of Roberts' construction over $\Z/2\Z$-coefficients to tangles in $M_1\#\dots\#M_p\setminus\mathring D^3$ for any interval bundles over surfaces $M_1,\dots,M_p$ for any integer $p\ge 0$. That is, given a tangle $\TT_1$ in $M_1\#\dots\#M_p\setminus \mathring D^3$ with $2n$ endpoints on $\partial D^3=S^2$, we construct a bigraded type A structure $\llangle \TT_1\rrbracket$. For a tangle $\TT_2$ in $M_{p+1}\#\dots\#M_r\setminus \mathring D^3$, another connected sum of interval bundles over surfaces with a ball removed, with $2n$ endpoints on $\partial D^3$, we construct a bigraded type D structure $\llbracket \TT_2\rrangle$. These type D and type A structures are over bigraded algebras $\BB\Gamma_n(M_1\#\dots\#M_r,p)$ equipped with $(1,0)$ differentials $d_{\BB\Gamma_n(M_1\#\dots \#M_r,p)}$. We prove that the type A structure $\llangle \TT_1\rrbracket$ and the type D structure $\llbracket \TT_2\rrangle$ are invariants of the tangles $\TT_1$ and $\TT_2$ under isotopy.

Furthermore, given $\TT_1$ and $\TT_2$ as in the preceding paragraph, one can form their connected sum $\LL=\TT_1\natural\TT_2$, which is a link in $M=M_1\#\dots\#M_r$. Note that every link in $M$ can be written as the connected sum of two tangles in this way.

Let $\llangle \LL\rrangle = \llangle \TT_1\rrbracket\boxtimes\llbracket \TT_2\rrangle$ with the box tensor product differential $\partial^\boxtimes$.

\begin{theorem}
Let $\LL=\TT_1\# \TT_2$ as before, then $(\CKh(\LL),d_{Kh})$ is isomorphic to $(\llangle \LL \rrangle,\partial^\boxtimes)$. 
\end{theorem}

In section 2, we define diagrams for links in $M$. In section 3, we define the Bar-Natan category and the Bar-Natan complex of a tangle $BN(T)$. In section 4, we define the Khovanov chain complex of a link $CKh(L)$ and prove it is a link invariant. In section 5, we discuss tangles in $M1\#\dots\# M_p\setminus\mathring D^3$ and establish notation. In section 6 we define the cleaved links algebra $\BB\Gamma_n(M,p)$. In section 7 we define the underlying bigraded modules $\llbracket \overrightarrow{T}\rrangle$ and $\llangle \overleftarrow{T}\rrbracket$ with differentials $\overrightarrow{d}_{\APS}$ and $\overleftarrow{d}_{\APS}$. In sections 8 and 9, we define the type D and type A structures, respectively, and prove their invariance. In section 10, we prove the Khovanov chain complex $CKh(\LL)$ is isomorphic to the box tensor product of the type A and type D structures $(\llangle\LL\rrangle,\partial^\boxtimes)$ and give another proof of the invariance of $CKh(\LL)$ by using this isomorphism.

\subsection{Future Work}

Following work of \cite{willis-kh} who generalized Khovanov homology to links in $\#^q(S^1\times S^2)$, it is possible to define Khovanov homology for links in $M^q=M\#(\#^q S^1\times S^2)$. In addition, just as the Khovanov chain complex for a link in $S^3$ can be deformed to give Lee homology, this can be done for the Khovanov chain complex of a link in $M$ or $M^q$. This allows for a definition of an $s$-invariant of a link in $M^q$ similar to \cite{rasmussen-s-inv} and \cite{bw-categorification}. In a future paper, we will explore these constructions and give some applications.

\textbf{Acknowledgments} The author would like to thank Yi Ni for advising him in this project and Daren Chen for helpful conversations.

\section{Links in Connected Sums of Interval Bundles}

Consider the connected sum $M=M_1\#\dots\#M_r$, where $M_i$ is an orientable $3$-manifold that is an interval bundle over a surface $F_i$. Choose $r-1$ separating spheres $S^2$ in $M$: cutting along these spheres separates $M$, without loss of generality, into $M_1\setminus\mathring D^3$, $M_i\setminus(\mathring D^3\sqcup\mathring D^3)$ for $i=2,\dots,r-1$, and $M_r\setminus\mathring D^3$. Let $F_1^\circ=F_1\setminus \mathring D_1^-$, $F_i^\circ=F_i\setminus(\mathring D_{i-1}^+\sqcup \mathring D_i^-)$ for $i=2,\dots,r-1$, and $F_r^\circ=F_r\setminus\mathring D_{r-1}^+$. We glue these copies together along their boundaries to obtain the surface $$\Sigma=\left(\bigsqcup_{i=1}^r F_r^\circ\right)/(\partial D_i^-\sim\partial D_i^+)\subset M,$$
where if we parametrize $\partial D_i^-$ and $\partial D_i^+$ as the unit circle inside $\R^2$, the identification $\partial D_i^-\sim -\partial D_i^+$ is by reflection across the axis $y=-x$. 

\begin{figure}[h]
\centering
\includegraphics[width=\textwidth]{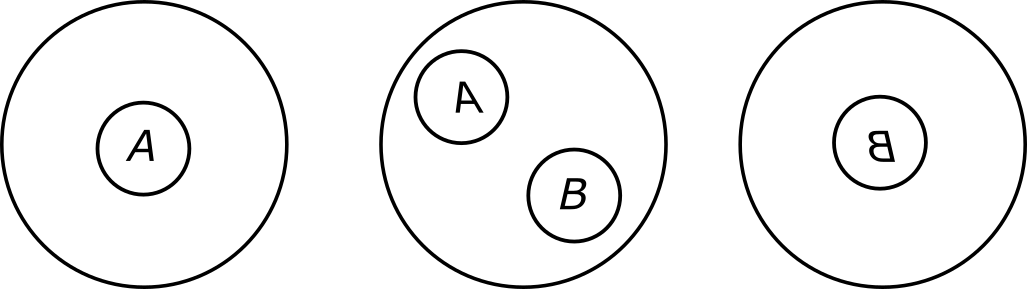}
\caption{$\Sigma$ for $M=\#^3\R P^3$, where we consider $\R P^3$ minus a point as the twisted interval bundle over $\R P^2$, which we draw as a disk quotiented by the antipodal map on the boundary. The disks marked $A$ and $B$ are cut out and glued together by a reflection.}
\label{fig:sigma3}
\end{figure}

Just like for links in $S^3$, links in $M$ can be represented by their projections to $\Sigma$, where we represent the relative heights of the strands by over and under crossings. We first isotope the link so that it is disjoint from each $I=[0,1]$ fiber over $\mathring D_i^\pm$, and the intersections of the link with each separating sphere occur at $\partial D_i\times\{1/2\}$. We choose the link projections to be transverse to itself, have no triple points, and be transverse to $\partial D_i^\pm$ for all $i=1,\dots, r-1$. Clearly these conditions are generic. The link now lives in $\Sigma\times I$, so the projection is the projection of the link onto $\Sigma$, with over and under crossings marked according to the $I$-bundle structure.

The following types of local moves of link diagrams, pictured in Figures ~\ref{fig:r123}, ~\ref{fig:finger-mirror}, and ~\ref{fig:handleslide} describe isotopies of links in $M$:

\begin{itemize}
    \item Reidemeister I, II, and III moves locally in a $D^2$ disjoint from $\partial D_i^\pm$,
    \item Finger moves across a $\partial D_i$, where a finger move takes an arc $a$ that is parallel to a portion of $\partial D_i$ and pushes it across $\partial D_i$ so it becomes the arcs $b,c,$ and $d$.
    \item Mirror moves across a $\partial D_i$, where a mirror move takes a crossing $x$ that is adjacent to $\partial D_i$ and pushes it across $\partial D_i$ so it becomes the crossing $x'$ on the other side.
    \item Handleslides of an arc $a$ above or below $\partial D_i$, where a handleslide takes the arc and moves it above or below the separating sphere. 
\end{itemize}

\begin{figure}[h]
\centering
\begin{subfigure}[b]{0.4\textwidth}
    \centering
    \includegraphics[width=\textwidth]{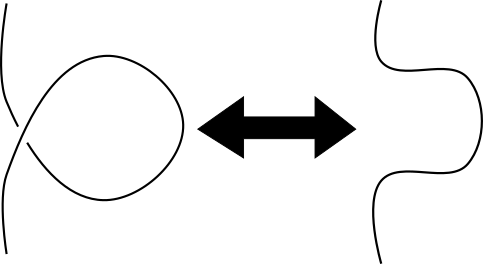}
    \caption{Reidemeister I Move}
    \label{fig:r1}
\end{subfigure}
\hfill
\begin{subfigure}[b]{0.4\textwidth}
    \centering
    \includegraphics[width=\textwidth]{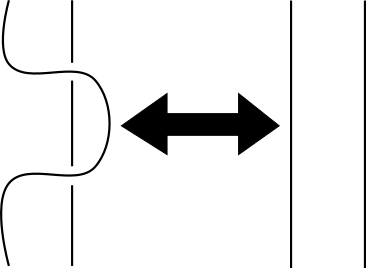}
    \caption{Reidemeister II Move}
    \label{fig:r2}
\end{subfigure}
\hfill
\begin{subfigure}[b]{0.6\textwidth}
    \centering
    \def\svgwidth{.925\linewidth}
    %% Creator: Inkscape 1.1 (c68e22c387, 2021-05-23), www.inkscape.org
%% Accompanies image file 'drawing.pdf' (pdf, eps, ps)
\begingroup%
  \makeatletter%
  \providecommand\color[2][]{%
    \errmessage{(Inkscape) Color is used for the text in Inkscape, but the package 'color.sty' is not loaded}%
    \renewcommand\color[2][]{}%
  }%
  \providecommand\transparent[1]{%
    \errmessage{(Inkscape) Transparency is used (non-zero) for the text in Inkscape, but the package 'transparent.sty' is not loaded}%
    \renewcommand\transparent[1]{}%
  }%
  \providecommand\rotatebox[2]{#2}%
  \newcommand*\fsize{\dimexpr\f@size pt\relax}%
  \newcommand*\lineheight[1]{\fontsize{\fsize}{#1\fsize}\selectfont}%
  \ifx\svgwidth\undefined%
    \setlength{\unitlength}{396.78175078bp}%
    \ifx\svgscale\undefined%
      \relax%
    \else%
      \setlength{\unitlength}{\unitlength * \real{\svgscale}}%
    \fi%
  \else%
    \setlength{\unitlength}{\svgwidth}%
  \fi%
  \global\let\svgwidth\undefined%
  \global\let\svgscale\undefined%
  \makeatother%
  \begin{picture}(1,0.4174387)%
    \lineheight{1}%
    \setlength\tabcolsep{0pt}%
    \put(0,0){\includegraphics[width=\unitlength,page=1]{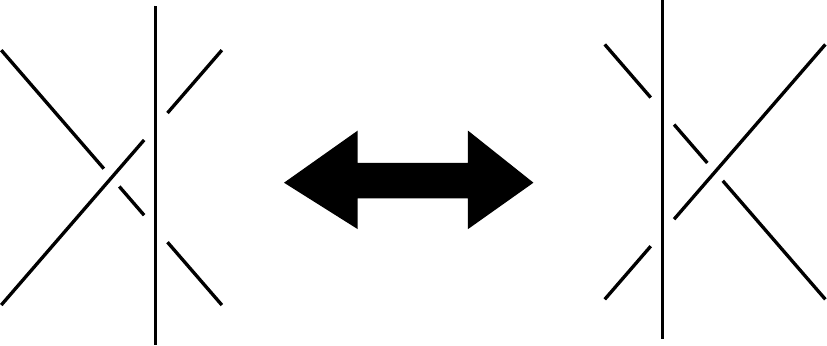}}%
    \put(0.06280528,0.19714496){\makebox(0,0)[lt]{\lineheight{1.25}\smash{\begin{tabular}[t]{l}$e$\end{tabular}}}}%
    \put(0.217888,0.26116164){\makebox(0,0)[lt]{\lineheight{1.25}\smash{\begin{tabular}[t]{l}$c_1$\end{tabular}}}}%
    \put(0.22059319,0.12591496){\makebox(0,0)[lt]{\lineheight{1.25}\smash{\begin{tabular}[t]{l}$c_2$\end{tabular}}}}%
    \put(0.73002209,0.26476813){\makebox(0,0)[lt]{\lineheight{1.25}\smash{\begin{tabular}[t]{l}$c_1$\end{tabular}}}}%
    \put(0.71920261,0.12411161){\makebox(0,0)[lt]{\lineheight{1.25}\smash{\begin{tabular}[t]{l}$c_2$\end{tabular}}}}%
    \put(0.89321974,0.19624322){\makebox(0,0)[lt]{\lineheight{1.25}\smash{\begin{tabular}[t]{l}$d$\end{tabular}}}}%
  \end{picture}%
\endgroup%

    \caption{Reidemeister III Move}
    \label{fig:r3}
\end{subfigure}
\caption{Reidemeister I, II, and III moves}
\label{fig:r123}
\end{figure}
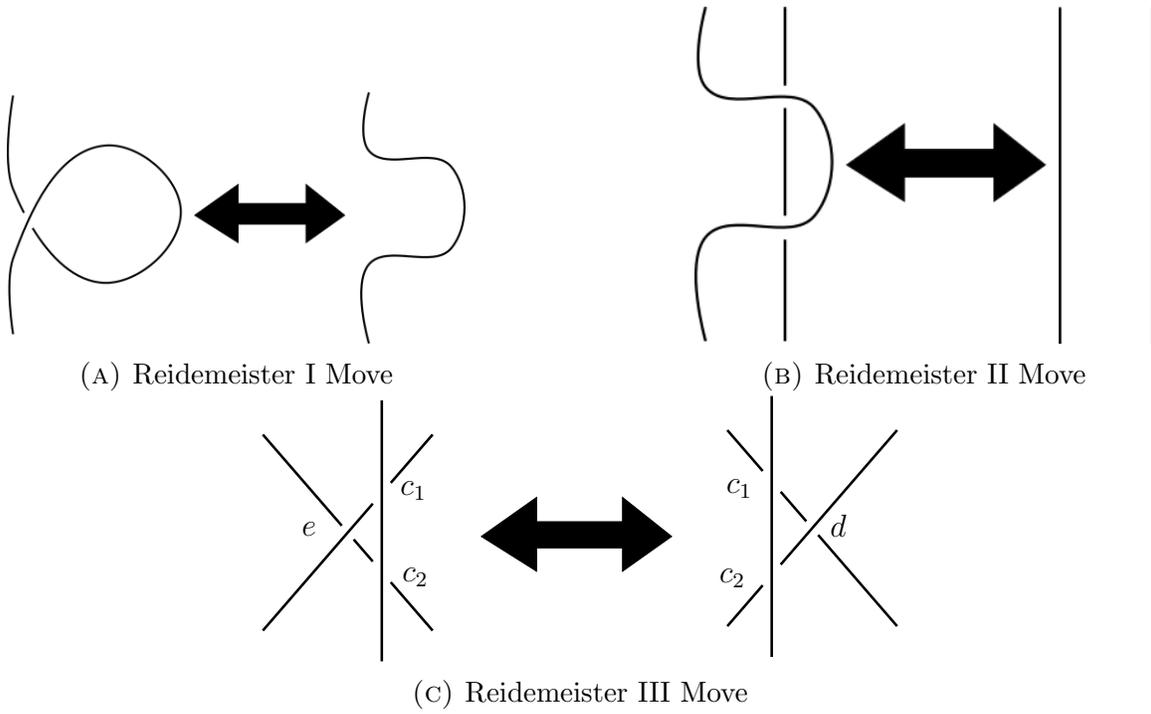

\begin{figure}[h]
\centering
\begin{subfigure}[b]{0.9\textwidth}
    \centering
    \includegraphics[width=\textwidth]{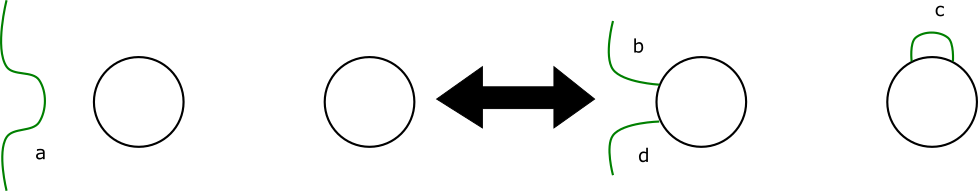}
    \caption{Finger Move}
    \label{fig:finger}
\end{subfigure}
\hfill
\begin{subfigure}[b]{0.9\textwidth}
    \centering
    \includegraphics[width=\textwidth]{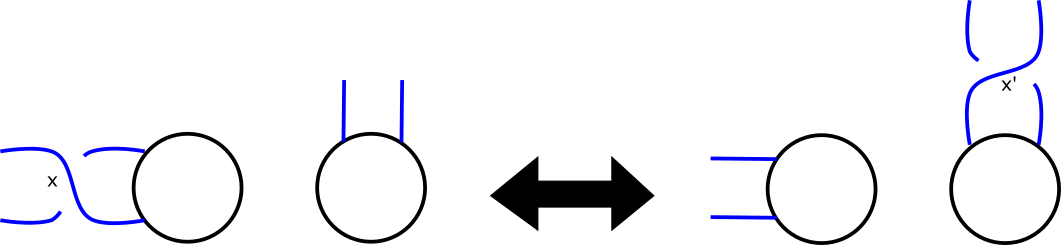}
    \caption{Mirror Move}
    \label{fig:mirror}
\end{subfigure}
\caption{Finger and Mirror moves: the two moves take the left diagram to the right diagram on the same row. In each diagram, the two black circles are $\partial D_i^\pm$ and are identified by a reflection across the line $y=-x$.}
\label{fig:finger-mirror}
\end{figure}

\begin{figure}[h]
    \centering
    \includegraphics[width=0.6\textwidth]{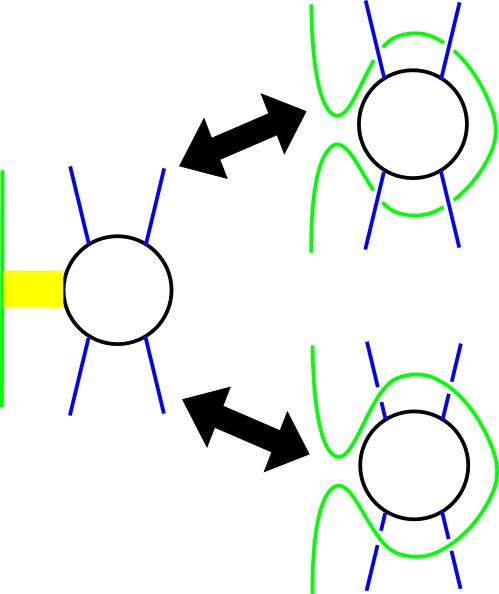}
    \caption{A handleslide of the blue arc below and above $\partial D_i$.}
    \label{fig:handleslide}
\end{figure}

\begin{lemma}
Two link diagrams $L_1,L_2$ in $\Sigma$ represent isotopic links if and only if they are related by a sequence of the above moves.
\end{lemma}
\begin{proof}
As usual, we consider the types of singularities that can occur when we isotope a link diagram $L$. We first fix the set of $r-1$ separating spheres beforehand and consider isotopies that are allowed to cross the separating spheres. Thus, in our diagram, we fix $D_i^\pm$. Within each $F_i^\circ$, cusps, tangencies, and triple points of intersections of $L$ correspond to Reidemeister I, II, and III moves, respectively. Having fixed $D_i^\pm$, there are no cusps of these borders.

Isotopies that cross the separating spheres induce diagram moves that move a portion of the link across $\partial D_i$ for some $i$. A tangency between a strand and $\partial D_i$ corresponds to a finger move across $\partial D_i$. Finally, a triple point between a pair of strands and $\partial D_i$ corresponds to a mirror move across $\partial D_i$.

In addition to isotopies of the link diagram $L$ in $\Sigma$, there are isotopies of the link in $M$ that move a portion of the link above or below the balls $D^3$ that we cut out. Taking an arc and moving it above or below a $D^3$ to the other side corresponds in the link diagram to a handleslide of the arc across $\partial D_i$ for some $i$.
\end{proof}

\section{Bar-Natan Category}

In this section we define an analogue of Bar-Natan's cobordism category for tangles (with boundary not necessarily on a separating sphere) in $M$.

A crossingless tangle diagram is a diagram of a tangle in $M$ with no crossings. Given a finite set of points $B$ in $M$, which may be empty and may contain points on one or more separating spheres, let $Cob^3(M,B)$ be the category whose objects consists of crossingless tangle diagrams in $M$ with boundary $B$, and whose morphisms consists of cobordisms between crossingless diagrams, that is, oriented, smoothly-embedded surfaces in $M\times I$ whose boundary lies entirely on $M\times\{0,1\}$, regarded up to boundary-preserving isotopy. Composition of cobordisms is given by placing one cobordism on top of the other.

Then, let $Cob^3_{/l}(M,B)$ be $Cob^3(M,B)$ modulo the local relations $S,T,$ and $4Tu$ as described in section 4.1.2 of \cite{bar-natan}. Let $\CC(M,B)$ be the category 
$$Kom(Mat(Cob^3_{/l}(M,B))),$$ that is, the category of finite chain complexes over the matrix closure of $Cob^3_{/l}(M,B)$.

Now let $T$ be a tangle diagram for a tangle in $M$. A resolution of $T$ is a crossingless diagram obtained by replacing each crossing of $T$ with either the $0$ or $1$-smoothing of the crossing as depicted in Figure~\ref{fig:resolutions}. Let $\RES(T)$ denote the set of all resolutions of $T$. If we choose an ordering of the crossings, then a resolution corresponds to a unique function $r:\{1,\dots,n\}\to\{0,1\}$, where $n$ is the number of crossings, and we let $T_r$ denote the resolution of $T$ obtained by smoothing the $i$-th crossing of $T$ according to $r(i)$ for $i=1,\dots,n$. 

\begin{figure}[h]
\centering
    \def\svgwidth{.9\linewidth}
    %% Creator: Inkscape 1.1 (c68e22c387, 2021-05-23), www.inkscape.org
%% Accompanies image file 'resolutions.pdf' (pdf, eps, ps)
\begingroup%
  \makeatletter%
  \providecommand\color[2][]{%
    \errmessage{(Inkscape) Color is used for the text in Inkscape, but the package 'color.sty' is not loaded}%
    \renewcommand\color[2][]{}%
  }%
  \providecommand\transparent[1]{%
    \errmessage{(Inkscape) Transparency is used (non-zero) for the text in Inkscape, but the package 'transparent.sty' is not loaded}%
    \renewcommand\transparent[1]{}%
  }%
  \providecommand\rotatebox[2]{#2}%
  \newcommand*\fsize{\dimexpr\f@size pt\relax}%
  \newcommand*\lineheight[1]{\fontsize{\fsize}{#1\fsize}\selectfont}%
  \ifx\svgwidth\undefined%
    \setlength{\unitlength}{627.74395319bp}%
    \ifx\svgscale\undefined%
      \relax%
    \else%
      \setlength{\unitlength}{\unitlength * \real{\svgscale}}%
    \fi%
  \else%
    \setlength{\unitlength}{\svgwidth}%
  \fi%
  \global\let\svgwidth\undefined%
  \global\let\svgscale\undefined%
  \makeatother%
  \begin{picture}(1,0.18652512)%
    \lineheight{1}%
    \setlength\tabcolsep{0pt}%
    \put(0,0){\includegraphics[width=\unitlength,page=1]{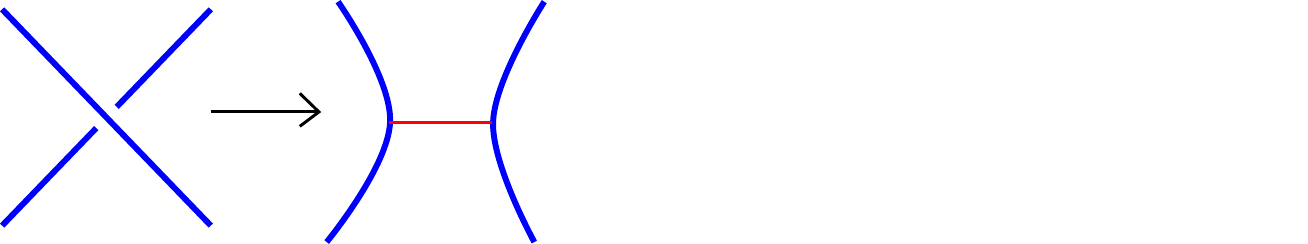}}%
    \put(0.16196115,0.12210005){\makebox(0,0)[lt]{\lineheight{1.25}\smash{\begin{tabular}[t]{l}$\rho(c)=0$\end{tabular}}}}%
    \put(0,0){\includegraphics[width=\unitlength,page=2]{resolutions.pdf}}%
    \put(0.71124227,0.11748611){\makebox(0,0)[lt]{\lineheight{1.25}\smash{\begin{tabular}[t]{l}$\rho(c)=1$\end{tabular}}}}%
    \put(0,0){\includegraphics[width=\unitlength,page=3]{resolutions.pdf}}%
  \end{picture}%
\endgroup%

\caption{$0$ and $1$ resolutions of a crossing}
\label{fig:resolutions}
\end{figure}

Let $BN(T)$ denote the object in $\CC(M,\partial T)$ whose underlying vector space is
$$\bigoplus_{r\in\RES(T)} T_r,$$ and with differential $d_{BN}$ given by the following rule: suppose $r$ and $r'$ are two resolutions that only differ in one of the crossings $i$, where $r(i)=0$ and $r'(i)=1$. Then there is a differential from $T_r$ to $T_{r'}$, where the morphism is the saddle cobordism as described in \cite{bar-natan} Section 2.6. There is a rule for the sign on the differential as described in \cite{bar-natan} Section 2.7.

Changing the smoothing of a crossing has the same effect as doing surgery on a band joining the two arcs of the resolution. We will often denote this with a bridge $\gamma$, where $\gamma$ is the core of the band that is surgered, and call it surgery on the arc or bridge $\gamma$.

The crossings of $T$ form a set $\CR(T)$. If we orient $T$ then each crossing is either a positive crossing or a negative crossing: take a small neighborhood $B^3\subset M$ around the crossing with orientation given by the orientation of $M$ such that the intersection with the tangle is a pair of disjoint (oriented) line segments, then the tangle in that neighborhood is isotopic to one of the two local pictures in Figure~\ref{fig:positive-negative}.

\begin{figure}[h]
\centering
\includegraphics[width=0.4\textwidth]{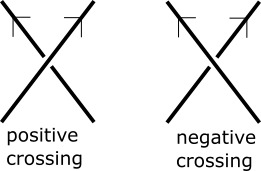}
\caption{Positive and Negative crossings}
\label{fig:positive-negative}
\end{figure}

We will denote by $n_+(T)$ the number of positive crossings and by $n_-(T)$ the number of negative crossings.

We use Bar-Natan's conventions for the homological grading, that is, the homological degree of $T_r$ is $\sum_{\CR(T)}r(c)-n^-$. There is also a quantum grading, which is equal to $\sum_{\CR(T)}r(c)+n^+-2n^-$. From now on, we consider $\CC(M,\partial T)$ as a bigraded category.

\begin{theorem}
The Bar-Natan complex of $T$ is an object in $\CC(M,\partial T)$, that is, $d_{BN}^2=0$.
\end{theorem}
\begin{proof}
Just like in \cite{bar-natan}, the composition of two saddle cobordisms can be time-reordered to be composed in the other way. The signs are chosen such that the two compositions anti-commute.
\end{proof}

\begin{theorem}\label{thm:bn-invariance}
The Bar-Natan complex of $T$, viewed as an object in $\Kom(\Mat(Cob^3_{/l}))$, is invariant up to homotopy under Reidemeister moves I, II, and III, and finger and mirror moves through $\partial D_i$. 
\end{theorem}

We do not know whether $BN(T)$ is invariant under handleslides.

\begin{proof}
The proof of invariance under Reidemeister moves is done in the same way as in \cite{bar-natan}, as the cobordisms between diagrams related by Reidemeister moves are constructed locally.

Consider a finger or mirror move from tangle diagram $T_1$ to $T_2$. There is a clear 1-1 correspondence between crossings of $T_1$ and crossings of $T_2$, where for a mirror move, the crossing in $T_1$ that is moved across a separating sphere is identified with the crossing in $T_2$ after the move. It is evident that, under this 1-1 correspondence, each resolution $(T_1)_r$ is isotopic to $(T_2)_r$, and furthermore, the differentials are identified. Thus, $BN(T_1)$ is isomorphic to $BN(T_2)$.  
\end{proof}

\section{Khovanov Chain Complex}

Now we consider links in $M$, viewed as tangles with empty boundary. To define the Khovanov chain complex over rings other than $\Z/2\Z$, we need to introduce some extra structure, as defined by \cite{gabrovsek-rp3}. Given a link diagram $L$, we orient $L$ and we orient the circles of $L_r$ for each resolution $r$ of $L$ in an arbitrary manner. 

At each crossing of $L$, there are three possibilities for the circles involved when changing from a $0$-smoothing to a $1$-smoothing:
\begin{itemize}
    \item A merge, that is, two circles join into one circle,
    \item A split, that is, one circle splits into two circles,
    \item A 1-1 bifurcation.
\end{itemize}

Given a resolution $L_r$, at each crossing that is a merge or a split, each of the circles involved with crossing can be locally consistently oriented at the crossing or locally inconsistently oriented according to Figure 6 of \cite{gabrovsek-rp3}. The local consistency is indetermined when the crossing is a 1-1 bifurcation.

We extend the notion of locally consistent orientation to cobordisms with the following definition.

\begin{definition}
Let $S$ be an oriented cobordism in $M\times I$, and let $C$ be a component of the boundary of $S$. We say that $C$ is locally consistently oriented with $S$ if the orientation of $C$ agrees with the boundary orientation of $S$.
\end{definition}

This notion of local consistent orientation agrees with the notion in \cite{gabrovsek-rp3}: let $r$ be a resolution of $L$ and $c$ a crossing of $L$,

\subsection{1-1 bifurcations in pairs of bridges}

We must understand where 1-1 bifurcations may arise when performing surgery on two arcs. Let $r$ be a resolution of a tangle diagram, and let $\alpha,\beta$ be two arcs in the diagram for $r$. There are four cases that involve one or more 1-1 bifurcations:

\begin{enumerate}
    \item[I)] The two arcs $\alpha,\beta$ are both 1-1 bifurcations.
    \item[II)] $\alpha$ is a 1-1 bifurcation, $\beta$ is not a 1-1 bifurcation, and the image of $\alpha$ in $r_\beta$ is still a 1-1 bifurcation.
    \item[III)] $\alpha$ is a 1-1 bifurcation, $\beta$ is not a 1-1 bifurcation, but the image of $\alpha$ in $r_\beta$ is no longer a 1-1 bifurcation.
    \item[IV)] $\alpha$ and $\beta$ are both not 1-1 bifurcations, but the image of $\alpha$ in $r_\beta$ is a 1-1 bifurcation.
\end{enumerate}

\begin{lemma}\label{lem:bif-case-iii}
In the setting of case III), let $\alpha$ have ends on the circle $C$, then $\beta$ splits $C$ into two circles $C_1,C_2$, and the image of $\alpha$ in $r_\beta$ merges $C_1$ and $C_2$.
\end{lemma}
\begin{proof}
Clearly $\beta$ must have at least one foot on $C$. Suppose $\beta$ merges $C$ with $C'$ to get $C''$, then the image of $\alpha$ in $r_\beta$ is a 1-1 bifurcation: the result of surgery on $\alpha$ in $r_\beta$ is to change $C''$ to a new circle which is the result of surgery on $\alpha$ in $r$ with $C'$ attached along $\beta$.

If $\beta$ splits $C$ into $C_1,C_2$, but both feet of $\alpha$ are on one of the new circles, say $C_1$, then by assumption the image of $\alpha$ in $r_\beta$ splits $C_1$ into two circles $C_3,C_4$ where the image of $\beta^\dagger$ in $r_{\beta,\alpha}$ has a foot on $C_2$ and a foot on $C_3$. Then in $r$, $\alpha$ must be a split map: it splits $C$ into $C_2\#C_3$ and $C_4$. This is a contradiction.
\end{proof}

\begin{lemma}\label{lem:bif-case-iv}
In the setting of case IV), $\alpha$ and $\beta$ are both merges of two circles. Furthermore, the image of $\beta$ in $r_\alpha$ is a 1-1 bifurcation.
\end{lemma}
\begin{proof}
The setting of IV) is obtained from III) by doing surgery on $\beta$, so the first assertion is implied by the previous lemma.

Let $C_1,C_2$ be the two circles that $\alpha$ and $\beta$ have feet on, so that surgery on $\alpha$ or $\beta$ merges $C_1$ and $C_2$ into $C$. Let $\overline{\beta}$ denote the image of $\beta$ in $r_\alpha$, $\overline{\beta}$ has both feet on $C$. Suppose surgery on $\overline{\beta}$ splits $C$ into two circles $C_3$ and $C_4$, then $\overline{\alpha^\dagger}$, the image of $\alpha^\dagger$ in $r_{\alpha,\overline{\beta}}$, has feet on $C_3$ or $C_4$ (or both). Do surgery on $\overline{\alpha^\dagger}$ to get back $r_\beta$. We know $\alpha$ has both feet on $C$ in $r_\beta$, therefore $\overline{\alpha^\dagger}$ had to merge the two circles $C_3$ and $C_4$, so the image of $\alpha$ in $r_\beta$ splits $C$ into $C_3$ and $C_4$. This is a contradiction.
\end{proof}

\subsection{Khovanov Functor}

Given a diagram $L$ of a link in $M$, we define the Khovanov chain complex of $L$, $CKh(L)$, to be the chain complex obtained by applying the Khovanov functor $\FF$ to $BN(L)$.

Let $V$ be the two-dimensional graded $R$-module freely generated by $\{v_+,v_-\}$ with the quantum degree of $v_\pm$ being $\pm 1$. Let $$o(S,C)=\begin{cases}
    1 & \text{$C$ is locally consistently oriented with $S$} \\
    -1 & \text{otherwise.}
\end{cases}$$ Define
$$w_-(S,C)=(-1)^{o(S,C)}v_-.$$

\begin{definition}
Let $\FF$ be the TQFT defined by $\FF(O)=V$ and
\begin{enumerate}
    \item $\FF(birth) = \epsilon:1\mapsto v_+$
    \item $\FF(death)=\eta:\begin{cases}v_+\mapsto 0,\ v_-\mapsto 1\end{cases}$
    \item If $S$ is a $1$-handle cobordism that splits circle $C$ into two circles $C_1$ and $C_2$, then $$\FF(S)=\Delta:\begin{cases}
        v_+\mapsto v_+\otimes w_-(S,C_2)+w_-(S,C_1)\otimes v_+ \\
        w_-(S,C)\mapsto w_-(S,C_1)\otimes w_-(S,C_2)
    \end{cases}$$
    \item If $S$ is a $1$-handle cobordism that merges circles $C_1$ and $C_2$ into $C$, then $$\FF(merge)=m:\begin{cases}
        v_+\otimes v_+\mapsto v_+,\ v_+\otimes w_-(S,C_2)\mapsto w_-(S,C) \\
        w_-(S,C_1)\otimes v_+\mapsto w_-(S,C),\ w_-(S,C_1)\otimes w_-(S,C_2)\mapsto 0
    \end{cases}$$
    \item $\FF(bifurcation) = 0$.
\end{enumerate}
\end{definition}

\begin{proposition}\cite[Proposition 7.2]{bar-natan}\label{prop:kh-functor}
$\FF$ is well-defined and degree-respecting. It descends to a functor $Cob^3_{/l}(\emptyset)\to R \Mod$.
\end{proposition}
\begin{proof}
To check that $\FF$ is well-defined, we need to show that the composition of two elementary cobordisms give the same map no matter in which way we compose them. 

It is clear that the induced maps of the composition of two $0$-handle cobordisms, the composition of a $0$-handle and a $2$-handle with disjoint attaching curves, and the composition of two $2$-handle cobordisms commute. Furthermore, for a composition of a $1$-handle with either a $0$-handle or $2$-handle, the attaching regions must be disjoint for the compositions to make sense, and it is easily seen that the cobordism maps commute.

Since the 1-1 bifurcation map is zero, the composition of two cobordism maps is zero if one of them is a 1-1 bifurcation. As shown in Lemma~\ref{lem:bif-case-iv}, the only way a composition of two saddle maps can have a 1-1 bifurcation in one order and no 1-1 bifurcation in the other order is for one order to have two 1-1 bifurcations and the other order to be a split of one circle $C$ into two circles $C_1$ and $C_2$ followed by merging $C_1$ and $C_2$. Studying the sign conventions shows that both of the compositions give the zero map.

The proof that $\FF$ is degree-respecting and descends to a functor from $Cob^3_{/l}$ is the same as in \cite{bar-natan}.
\end{proof}

It follows that $CKh(L)=\FF(BN(L))$ is a well-defined chain complex.

\begin{theorem}\label{thm:ckh-invariance}
The chain homotopy type of $CKh(L)$ is a well-defined invariant of the link $\LL$.
\end{theorem}
\begin{proof}
First, the proofs of Lemmas 6.2 and 6.3 of \cite{gabrovsek-rp3} apply to show that the homotopy type of $CKh(L)$ is invariant under changing the orientation of the link and changing the orientations of the circles in the states.

We need to check invariance under Reidemeister moves, finger moves, mirror moves, and handleslides. From Theorem~\ref{thm:bn-invariance} and Proposition~\ref{prop:kh-functor}, we get that the homotopy type of $CKh(L)$ is invariant under Reidemeister moves, finger moves, and mirror moves. 
\end{proof}

\subsection{Proof of Handleslide Invariance}

Consider a handleslide of an arc $a$ of $L$ above $\partial D_i$ for some $i=1,\dots,n$ to obtain a new link diagram $L'$. If there are $2m$ intersections between $L$ and $\partial D_i$, then the handleslide will create $2m$ new crossings. If the connect-sum sphere were not there, we could cancel out these new crossings by sliding $a$ above all the arcs that intersect $\partial D_i$ back to its original position in $L$. This corresponds to a sequence of Reidemeister moves, and indeed, on the chain complex level, $CKh(L)$ and $CKh(L')$ are related by a sequence of Reidemeister-type homotopy equivalences.

To make the above argument precise, we will pass through the notion of tangle diagrams with nodes. A tangle diagram with nodes is a tangle diagram where some of the crossings are designated as nodes and are treated as virtual crossings.

\begin{definition}
A tangle diagram with nodes $T$ in $\R^2$ is an immersion of some number of $S^1$ circle components and $n$ $[0,1]$ interval components in $\R^2$ such that
\begin{enumerate}
    \item The self-intersections of $T$ is a set of distinct points where two arcs intersect transversely. 
    \item Each self-intersection is decorated as either a crossing, and one arc is drawn going above the other, or as a node.
\end{enumerate}
\end{definition}

Let the set of crossings, not including nodes, of $T$ be $\CR(T)$. A resolution $r$ of $T$ is a resolution $\rho:\CR(T)\to\{0,1\}$ of the crossings of $T$. The resolution diagram $r(T)$ is the crossingless link obtained by replacing neighborhoods of crossings by smoothings according to $\rho$. The nodes are not resolved, but when determining the components of the resolution diagram, the two strands that intersect are treated as disjoint from each other.

\begin{figure}
    \centering
    \includegraphics[width=0.8\textwidth]{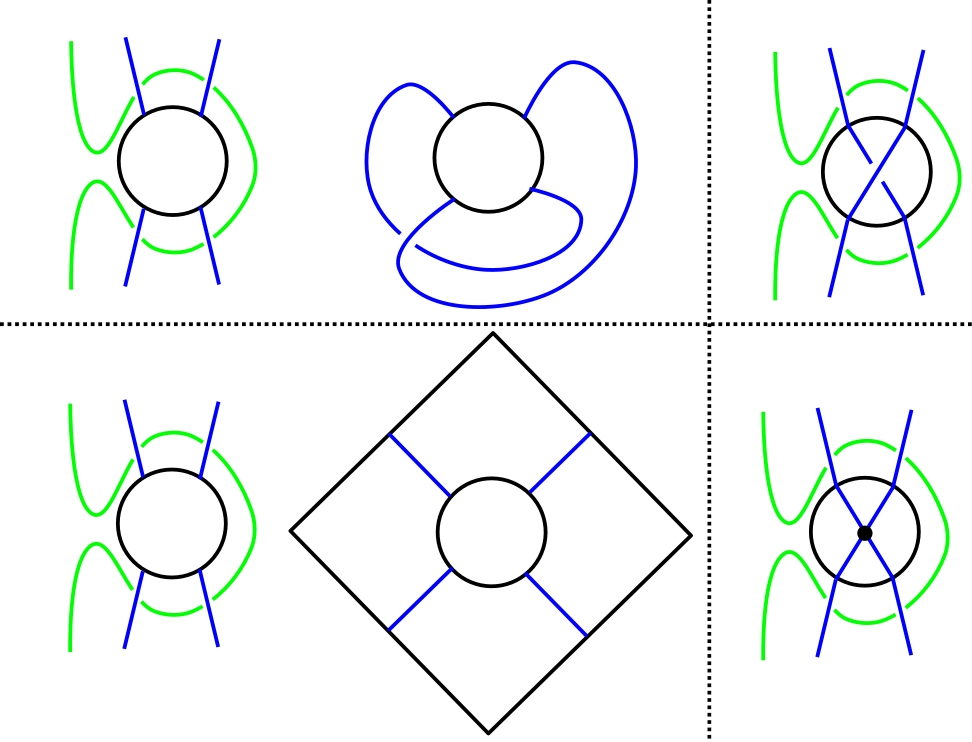}
    \caption{Two examples of local pictures of tangle diagrams with nodes. On the left is the actual tangle diagram and on the right is the tangle diagram with nodes that represents it. On the top row, there is a crossing between the blue strands. On the bottom, the blue strands do not cross in the torus (drawn as a square with opposite sides glued together), so the intersection on the right is marked as a node.}
    \label{fig:virtual-example}
\end{figure}

Given a tangle diagram $T$ in $\Sigma$, it is easy to construct a tangle diagram with nodes $T_*$ in $\R^2$ that represents $T$ in the sense that
\begin{itemize}
    \item There is a 1-1 correspondence between interval components of $T$ and interval components of $T_*$ and a 1-1 correspondence between circle components of $T$ and circle components of $T_*$.
    \item There is a 1-1 correspondence between crossings of $T$ and crossings of $T_*$ that respects the correspondence between components, the over-under positioning of the strands at each crossing, and the linear or cyclical order of the crossings along each component. 
\end{itemize}

Fix a circle $\CC$ in $\R^2$, say $\CC=\{(x,y):x^2+y^2=1\}$. The circle divides $\R^2$ into two regions: let $R_1=\{(x,y):x+y^2\le 1\}$ and $R_2=\R^2\setminus\mathring R_1$. 

Now suppose $\LL\subset M$ is a link and $L$ is a diagram of $\LL$. For the sake of the proof of handleslide invariance, we construct a link diagram with nodes $L_*$ so that \begin{itemize}
    \item the tangle diagram with nodes $L_*\cap R_1$ represents $$\LL\cap((M_{i+1}\setminus\mathring D^3)\#\dots\#M_r),$$ 
    \item the tangle diagram with nodes $L_*\cap R_2$ represents $$\LL\cap((M_{1}\setminus\mathring D^3)\#M_{2}\#\dots\#(M_i\setminus\mathring D^3)).$$
\end{itemize} In particular, doing a handleslide of some arc of $L$ over $\partial D_i^-$ corresponds to doing a handleslide of the corresponding arc of $L$ in $R_2$ over $\CC$.

\begin{proof}[Proof of Handleslide Invariance]
We will prove invariance of $CKh(L)$ under handleslides below $\partial D_i$; the proof for invariance under handleslides above $\partial D_i$ is the same. As shown in Figure~\ref{fig:handleslide-plus-minus}, handleslide invariance over $\partial D_i^+$ follows from handleslide invariance over $\partial D_i^-$: first do a finger move of the arc across $\partial D_i$ where the handleslide is to be performed, then do a handleslide of the arc in $F_i^\circ$ below $\partial D_i^-$, then do a sequence of finger and mirror moves across $\partial D_i$, and finally do one more finger move across $\partial D_i$ to obtain the diagram for a handleslide below $\partial D_i^+$.

\begin{figure}[h]
    \centering
    \includegraphics[width=0.5\textwidth]{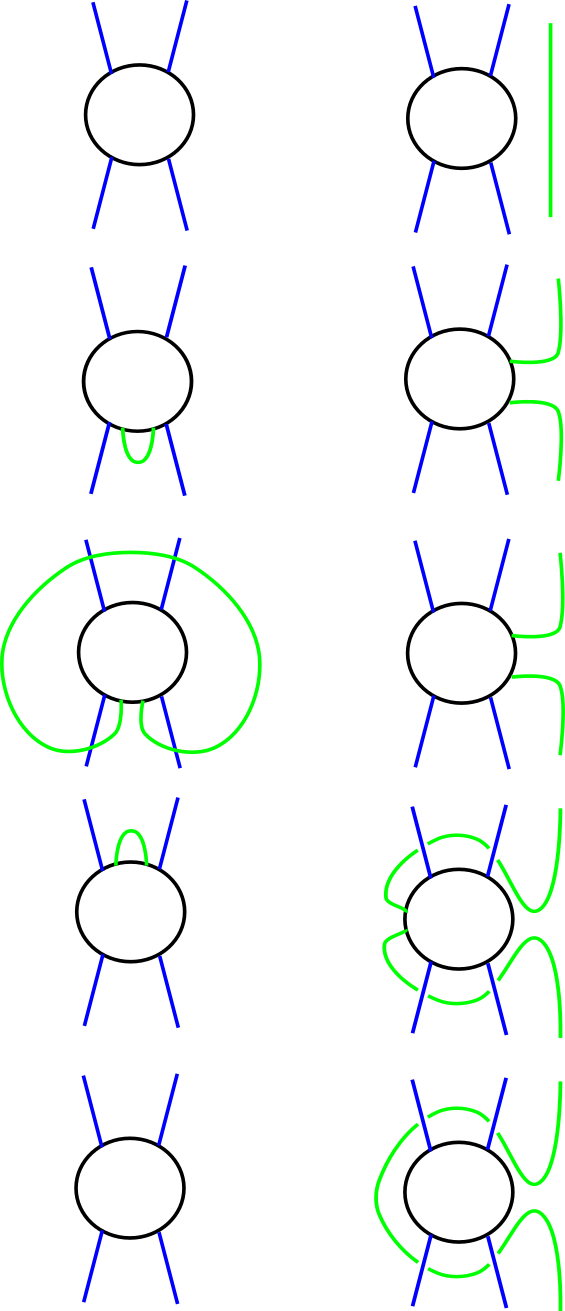}
    \caption{A handleslide over $\partial D_i^+$ can be realized by a handleslide over $\partial D_i^-$ together with a sequence of finger and mirror moves, and vice versa.}
    \label{fig:handleslide-plus-minus}
\end{figure}

Let $L$ be a link diagram, and let $L'$ be the diagram obtained by performing a handleslide of an arc $a$ of $L$ over $\partial D_i^-$. In $L'$, $a$ corresponds to a strand that circles around $\partial D_i^-$, which we will call the distinguished strand (shown in blue in the figures). As described above, we get link diagrams with nodes $L_*$ and $L'_*$ corresponding to $L$ and $L'$, respectively, such that $L'_*$ is obtained from $L_*$ by doing a handleslide of $a$ over $\CC$.

First, in $L'_*$, push the strand $a$ across $\CC$ into $R_1$. This corresponds to a sequence of mirror and finger moves of $a$ and the crossings with arcs coming out of $\partial D_i^-$ over $\partial D_i^-$. We will also use the distinguished strand to refer to the strand in $L'_*$ that represents the distinguished strand in $L'$.

Now in $L'_*$, it is apparent that one can shrink the distinguished strand to a small arc parallel to a portion of $\CC$ by performing a sequence of moves supported entirely in $R_2$ that only move the distinguished strand and leave all other strands and crossings fixed. These moves are either planar isotopies, Reidemeister moves, or Reidemeister moves where one or more crossing is instead a node. Since we only move the distinguished strand and we do not start with any nodes with the distinguished strand, we do not need to consider any Reidemeister moves with nodes where one or two of the intersecting arcs at the node belong to the distinguished strand. The only such Reidemeister moves are Reidemeister III moves with nodes among other strands. 

\begin{figure}[h]
    \centering
    \includegraphics[width=0.9\textwidth]{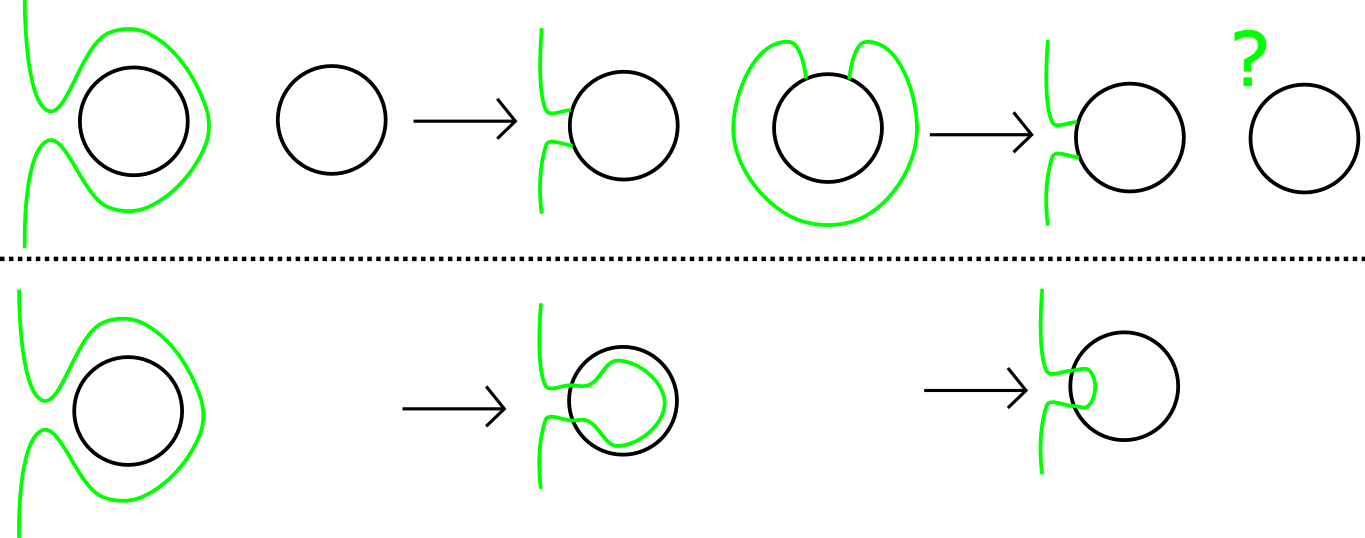}
    \caption{Steps of the proof of handleslide invariance. A schematic of the actual tangle diagram is shown on top while the corresponding virtual diagram is shown on the bottom.}
    \label{fig:handleslide-proof}
\end{figure}

A planar isotopy of the distinguished strand in $L'_*$ corresponds to either an isotopy in $\Sigma$ or a move that changes the homotopy class of the distinguished strand but keeps the set of all crossings fixed. There is a 1-1 correspondence between resolutions before and after the move, and the set of circles and bridges remain the same, so the chain complexes before and after the move are isomorphic.

Each Reidemeister move in the virtual diagram corresponds in $L'$ to either a Reidemeister move or one of the moves shown in Figure~\ref{fig:r123-prime}, which we will call Reidemeister moves I prime, II prime, and III prime. A Reidemeister III move with nodes in $L'_*$ corresponds to an isotopy of $L'$ in $\Sigma$. Call the link diagram obtained after this sequence of moves $L''$.

\begin{figure}[h]
    \centering
    \begin{subfigure}[b]{0.4\textwidth}
        \centering
        \includegraphics[width=\textwidth]{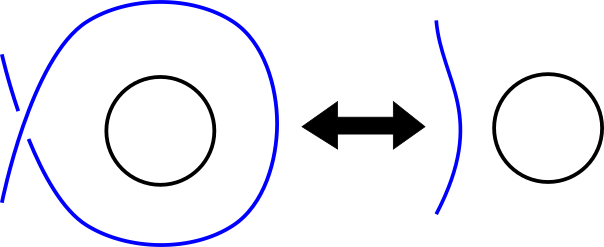}
        \caption{Reidemeister I Prime}
        \label{fig:r1-prime}
    \end{subfigure}
    \hfill
\begin{subfigure}[b]{0.4\textwidth}
    \centering
    \includegraphics[width=\textwidth]{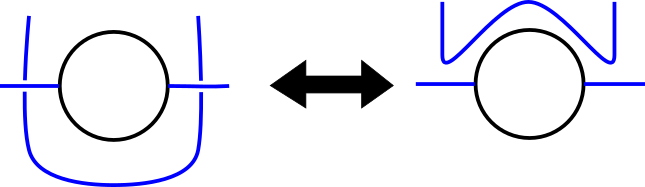}
    \caption{Reidemeister II Prime}
    \label{fig:r2-prime}
\end{subfigure}
\hfill
\begin{subfigure}[b]{0.6\textwidth}
    \centering
    \includegraphics[width=\textwidth]{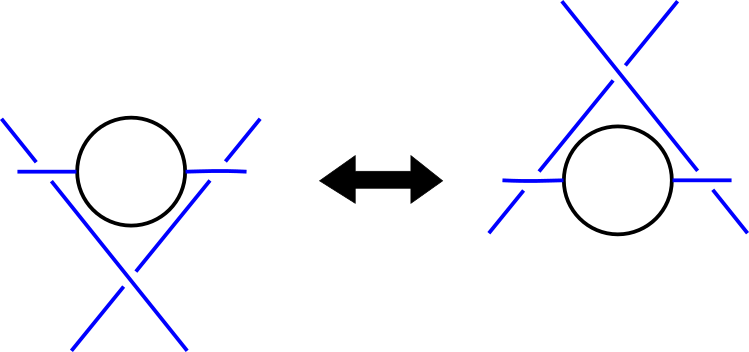}
    \caption{Reidemeister III Prime}
    \label{fig:r3-prime}
\end{subfigure}
\caption{Reidemeister I, II, and III Prime}
\label{fig:r123-prime}
\end{figure}

Note that in the prime Reidemeister moves, there are no other crossings among the strands pictured except for the crossings that are shown. We see that the partial cube of resolutions for the crossings pictured are the same as the cube of resolutions for the usual Reidemeister I, II, and III moves, and so the proof of invariance under the usual Reidemeister moves shows invariance under the prime Reidemeister moves. Therefore $CKh(T'')$ is homotopy equivalent to $CKh(T)'$.

After this sequence of moves, we have canceled out the crossings introduced by the handleslide and have not modified, created, or canceled out any other crossings. Thus, $CKh(T'')$ is homotopy equivalent to $CKh(T)$.
\end{proof}

\begin{remark}
In \cite{aps-i-bundles}, the maps in the differential in their Khovanov chain complex depends on whether the circles are trivial or nontrivial. The reason we chose a different differential is because the Khovanov chain complex is not invariant under handleslides if we do not have maps in the differential that merge two nontrivial circles to get a nontrivial circle or split a nontrivial circle into two nontrivial circles. For instance, Figure~\ref{fig:handleslide-counterexample} shows two links related by a handleslide that do not have homotopy equivalent Khovanov chain complexes.

\begin{figure}[h]
    \centering
    \includegraphics[width=\textwidth]{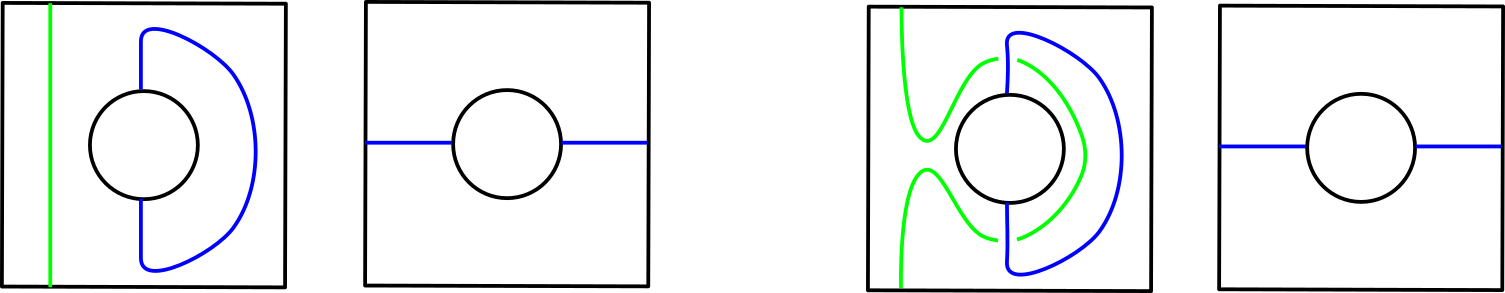}
    \caption{An example where the Khovanov homology of a link is not invariant under handleslides if we use the APS conventions. The link on the left is a disjoint union of two homologically nontrivial circles in the connected sum of two trivial interval bundles over the torus.}
    \label{fig:handleslide-counterexample}
\end{figure}
\end{remark}

\section{Tangles and Resolutions}

Much of the notation in the rest of the paper will be borrowed from \cite{roberts-type-d} and \cite{roberts-type-a}. Fix $p$ such that $0\le p\le r$. As described in the previous section, we draw links in $M_1\#\dots\#M_r$ using link diagrams in $\Sigma$, which we think of as $$(\overleftarrow{\U}\sqcup\overrightarrow{\U})/(\partial D_p^-\sim\partial D_p^+),$$ where 
$$\overleftarrow{\U}=\left(\bigsqcup_{i=1}^p F_i^\circ\right)/(\partial D_i^-\sim\partial D_i^+),\hspace{10mm}\overrightarrow{\U}=\left(\bigsqcup_{i=p+1}^r F_i^\circ\right)/(\partial D_i^-\sim\partial D_i^+)$$ are surfaces with boundary $\partial D_p^-$ and $\partial D_p^+$, respectively. We call $\overleftarrow{\U}$ and $\overrightarrow{\U}$ the left and right halves of $\Sigma$, respectively. When we do not need to be specific about which half it is, we write $\U$ to mean either $\overleftarrow{\U}$ or $\overrightarrow{\U}$.

Let $P_n$ be the set of points $p_1,\dots,p_{2n}$ on $\partial D_p$ ordered cyclically counterclockwise on $\partial D_p^+$ and clockwise on $\partial D_p^-$. 

The following notions that we define for right tangles have similar definitions for left tangles.

\begin{definition}
A right tangle $\overrightarrow{\TT}$ is a smooth, proper embedding of $n=n_A(\overrightarrow{\TT})$ copies of the interval $[0,1]$ and $n_C(\overrightarrow{\TT})$ copies of $S^1$ into $\overrightarrow{\U}$ whose boundary is the set of $2n$ points $P_n$ in $\partial D_p$.
\end{definition}

We consider two right tangles $\overrightarrow{\TT}_1$ and $\overrightarrow{\TT}_2$ equivalent if there is an isotopy of $\overrightarrow{\U}$ taking $\overrightarrow{\TT}_1$ to $\overrightarrow{\TT}_2$ pointwise fixing the boundary $\partial D_p$. 

Tangles can be isotoped to have a generic projection as described in Section 2. We record over/under-crossing of strands to obtain a tangle diagram for $\overrightarrow{\TT}$, which we will denote $\overrightarrow{T}$. Different diagrams for $\overrightarrow{\TT}$ are related by sequences of isotopies in $F_i$, Reidemeister moves, and finger moves, mirror moves, and handleslides across $\partial D_j$ for $j\ne p$. 

Following the terminology of \cite{aps-i-bundles}, we say that a circle is trivial if it bounds a disk in $\Sigma$, and a circle is bounding if it bounds either a disk or a M\"obius band in $\Sigma$. 

\subsection{Resolutions}

A resolution of a tangle diagram $\overrightarrow{T}$ consists of a chosen resolution $\rho$ for each crossing in $\CR(T)$ together with a crossingless matching $\overleftarrow m$ in $\overleftarrow{\U}$.

\begin{definition}
A crossingless matching, or matching for short, $m$ of $P_n$ in $\U$ is a proper embedding of $n$ unoriented arcs $\alpha_i$, $i=1,\dots,n$ in $\U$, such that $\partial\alpha_i\subset P_n$ for each $i$.
\end{definition}

The arcs of a matching are allowed to intersect $\partial D_i$ for $i\ne p$. Two matchings $m_1$ and $m_2$ in $\U$ are considered equivalent if they are isotopic in $\U$ by an isotopy that pointwise fixes $\partial\U=\partial D_p$. The equivalence classes of matchings on $P_n$ is denoted $\MATCH(n)$.

\begin{definition}
A bridge for a matching $m$ in $\U$ is an embedding $\gamma:[0,1]\to\U$ such that
\begin{enumerate}
    \item $\gamma(0)$ and $\gamma(1)$ are on distinct arcs of $m$,
    \item the image $\gamma((0,1))$ is disjoint from $m$.
\end{enumerate}

\end{definition}

We consider a bridge $\gamma_1$ for $m_1$ to be equivalent to a bridge $\gamma_2$ for $m_2$ if there is an isotopy of crossingless matchings, fixing $\partial\U$, that takes $m_1$ to $m_2$ and $\gamma_1$ to $\gamma_2$. When $m=m_1=m_2$, one such isotopy we will use involves sliding a foot of a bridge along the arc it abuts in $m$. We denote the equivalence classes of bridges for a crossingless matching $m$ to be $\BRIDGE(m)$.

\begin{figure}[h]
    \centering
    \includegraphics[width=0.8\textwidth]{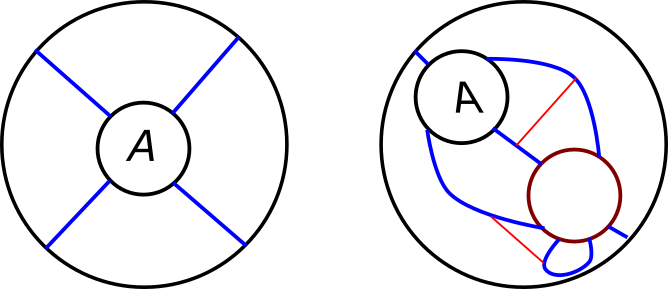}
    \caption{An example of a $6$-pointed matching (shown in blue) in $\protect\overleftarrow{\U}=\R P^2\#\R P^2\setminus\mathring D^2$ and two bridges for the matching (the red arcs that have endpoints on the matching).}
    \label{fig:matching-bridge-example}
\end{figure}

\begin{definition}
Given a bridge $\gamma$ for a crossingless matching $m$, we define $m_\gamma$ to be the crossingless matching obtained by surgery along $\gamma$. The co-core of the surgery is a bridge in $m_\gamma$ denoted by $\gamma^\dagger$.
\end{definition}

The notion of surgery descends to equivalence classes of crossingless matchings. 

\begin{figure}[h]
    \centering
    \def\svgwidth{.9\linewidth}
    %% Creator: Inkscape 1.1 (c68e22c387, 2021-05-23), www.inkscape.org
%% Accompanies image file 'surgery.pdf' (pdf, eps, ps)
\begingroup%
  \makeatletter%
  \providecommand\color[2][]{%
    \errmessage{(Inkscape) Color is used for the text in Inkscape, but the package 'color.sty' is not loaded}%
    \renewcommand\color[2][]{}%
  }%
  \providecommand\transparent[1]{%
    \errmessage{(Inkscape) Transparency is used (non-zero) for the text in Inkscape, but the package 'transparent.sty' is not loaded}%
    \renewcommand\transparent[1]{}%
  }%
  \providecommand\rotatebox[2]{#2}%
  \newcommand*\fsize{\dimexpr\f@size pt\relax}%
  \newcommand*\lineheight[1]{\fontsize{\fsize}{#1\fsize}\selectfont}%
  \ifx\svgwidth\undefined%
    \setlength{\unitlength}{450.28832623bp}%
    \ifx\svgscale\undefined%
      \relax%
    \else%
      \setlength{\unitlength}{\unitlength * \real{\svgscale}}%
    \fi%
  \else%
    \setlength{\unitlength}{\svgwidth}%
  \fi%
  \global\let\svgwidth\undefined%
  \global\let\svgscale\undefined%
  \makeatother%
  \begin{picture}(1,0.26436863)%
    \lineheight{1}%
    \setlength\tabcolsep{0pt}%
    \put(0,0){\includegraphics[width=\unitlength,page=1]{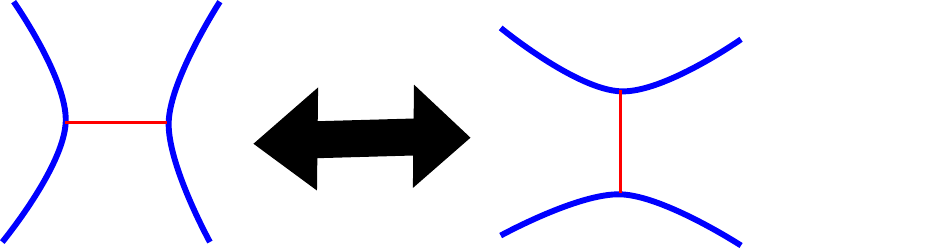}}%
    \put(0.11613851,0.15264169){\makebox(0,0)[lt]{\lineheight{1.25}\smash{\begin{tabular}[t]{l}$\gamma$\end{tabular}}}}%
    \put(0.68316674,0.09881957){\makebox(0,0)[lt]{\lineheight{1.25}\smash{\begin{tabular}[t]{l}$\gamma^\dagger$\end{tabular}}}}%
  \end{picture}%
\endgroup%

    \caption{Surgery on a bridge $\gamma$.}
    \label{fig:surgery}
\end{figure}

\begin{definition}
A resolution $r$ of $\overrightarrow{T}$ is a pair $(\rho,\overleftarrow{m})$, where $\rho:\CR(\overrightarrow{T})\to\{0,1\}$ and $\overleftarrow{m}$ is a crossingless matching of $P_n$ in $\overleftarrow{\U}$. The resolution diagram, $r(\overrightarrow{T})$, is the crossingless link in $\Sigma$ obtained by gluing $\overrightarrow{T}\subset\overrightarrow{\U}$ to $\overleftarrow{m}\subset\overleftarrow{\U}$ along $P_n$, and locally replacing (disjoint) neighborhoods of each crossing $c\in\CR(\overrightarrow{T})$ using the following rule depicted in Figure~\ref{fig:resolutions}.

\end{definition}

The set of resolutions will be denoted $\RES(\overrightarrow{T})$. The local arcs introduced by $\rho(\overrightarrow{T})$ are called resolution bridges. Since we choose the projection to have the crossings away from $\partial D_j$, the resolution bridges will not intersect $\partial D_j$ as well. We will use the terms bridge and arc interchangeably. The resolution bridge for a crossing $c$ will be denoted $\gamma_{r,c}$ (or just $\gamma_c$ when the resolution is understood). If $\rho(c)=0$ we will call $\gamma_c$ an active bridge for $r$, while if $\rho(c)=1$ it will be called inactive. Let $\overrightarrow{\Br}(r)$ denote the set of resolution bridges for $r$ and let $\overrightarrow{\ACT}(r)$ be the set of active bridges for $r$. 

A resolution diagram $r(\overrightarrow{T})$ consists of a crossingless diagram of circles, which we divide into two groups: the free circles that are contained in the interior of $\overrightarrow{\U}$, and cleaved circles that intersect $\partial D_p$. Let $\FREE(r)$ denote the set of free circles and $\cl(r)$ denote the set of cleaved circles. 

Let $r=(\rho,\overleftarrow{m})$ be a resolution, and $\gamma$ be a resolution bridge or a bridge for $\overleftarrow{m}$. We denote by $r_{\gamma}$ the resolution obtained by surgering the diagram for $r$ along $\gamma$. The cocore of the surgery is a bridge $\gamma^\dagger$ for $r_{\gamma}$.

There are three cases for the effect of surgery on a bridge $\gamma$ for $r$:

\begin{enumerate}
    \item surgery on $\gamma$ merges two distinct circles $\{C_a,C_b\}$ containing $\gamma(0)$ and $\gamma(1)$ to obtain a circle $C_\gamma$ that contains both feet of $\gamma^\dagger$.
    \item surgery on $\gamma$ divides a circle $C$ of $r$, which contains both feet of $\gamma$, into two circles $C_a$ and $C_b$ that contain the feet of $\gamma^\dagger$
    \item surgery on $\gamma$ is a 1-1 bifurcation, that is $\gamma$ has both feet on a circle $C$ of $r$ and $\gamma^\dagger$ also has both feet on a circle $C_\gamma$ of $r_\gamma$.
\end{enumerate}

\begin{figure}[h]
\centering
\def\svgwidth{.9\linewidth}
%% Creator: Inkscape 1.1 (c68e22c387, 2021-05-23), www.inkscape.org
%% Accompanies image file '1-1-bifurcation-example.pdf' (pdf, eps, ps)
\begingroup%
  \makeatletter%
  \providecommand\color[2][]{%
    \errmessage{(Inkscape) Color is used for the text in Inkscape, but the package 'color.sty' is not loaded}%
    \renewcommand\color[2][]{}%
  }%
  \providecommand\transparent[1]{%
    \errmessage{(Inkscape) Transparency is used (non-zero) for the text in Inkscape, but the package 'transparent.sty' is not loaded}%
    \renewcommand\transparent[1]{}%
  }%
  \providecommand\rotatebox[2]{#2}%
  \newcommand*\fsize{\dimexpr\f@size pt\relax}%
  \newcommand*\lineheight[1]{\fontsize{\fsize}{#1\fsize}\selectfont}%
  \ifx\svgwidth\undefined%
    \setlength{\unitlength}{448.80305253bp}%
    \ifx\svgscale\undefined%
      \relax%
    \else%
      \setlength{\unitlength}{\unitlength * \real{\svgscale}}%
    \fi%
  \else%
    \setlength{\unitlength}{\svgwidth}%
  \fi%
  \global\let\svgwidth\undefined%
  \global\let\svgscale\undefined%
  \makeatother%
  \begin{picture}(1,0.33549262)%
    \lineheight{1}%
    \setlength\tabcolsep{0pt}%
    \put(0,0){\includegraphics[width=\unitlength,page=1]{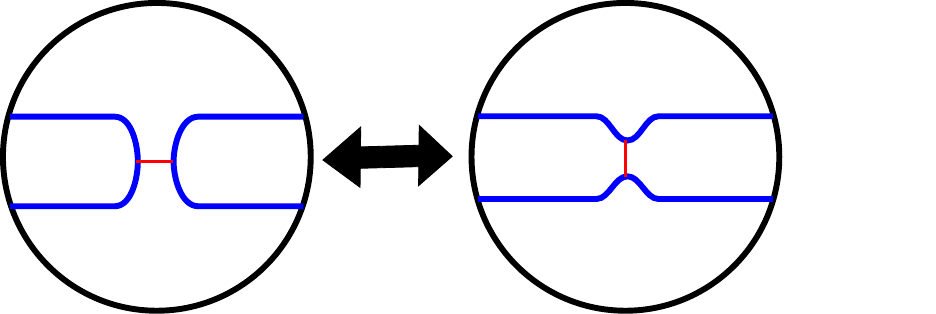}}%
    \put(0.15800288,0.17603846){\makebox(0,0)[lt]{\lineheight{1.25}\smash{\begin{tabular}[t]{l}$\gamma$\end{tabular}}}}%
    \put(0.68211819,0.15770433){\makebox(0,0)[lt]{\lineheight{1.25}\smash{\begin{tabular}[t]{l}$\gamma^\dagger$\end{tabular}}}}%
  \end{picture}%
\endgroup%

\caption{An example of a 1-1 bifurcation: surgery on $\gamma$ takes the circle on the left in $\R P^3$ to the circle on the right. Surgery on $\gamma^\dagger$ is also a 1-1 bifurcation.}
\label{fig:1-1-bifurcation-example}
\end{figure}

The last case does not happen for tangles in $S^3\setminus D^3$.

If $\eta$ is a bridge that is not equal to $\gamma$ if $\gamma$ is a resolution bridge or not in the same equivalence class as $\gamma$ if $\gamma$ is a bridge for $\overleftarrow{m}$, then $\eta$ is also a bridge in $r_\gamma$, which we call the image of $\eta$ in $r_\gamma$ and we will often denote by $\overline\eta$. We write $r_{\gamma,\eta}$ to denote the resolution obtained by surgery on $\gamma$ followed by surgery on $\overline\eta$.

\section{Cleaved Links Algebra}

Let $n\ge 1$. We will define differential bigraded algebras $\BB\Gamma_n(M,p)$ for all $0\le p\le r$, where $p$ specifies which direct summand of $M=M_1\#\dots \#M_r$ to cut along. These are analogous to the bigraded algebras $\BB\Gamma_n$ defined in \cite{roberts-type-d}, and indeed, $\BB\Gamma_n(S^3,0)=\BB\Gamma_n$. They are defined as quotients of quiver algebras on directed graphs $\Gamma_n(M,p)$ that we will construct.

We work over $\Z/2\Z$-coefficients for the rest of the paper.

\subsection{The vertices of \texorpdfstring{$\Gamma_n$}{Gamma_n}}

The vertices of $\Gamma_n(M,p)=\Gamma_n$ consist of crossingless link diagrams in $\Sigma$ where each component intersects $\partial D_p$ and is decorated with a sign in $\{+,-\}$. 

\begin{definition}
An $n$-cleaved link $L$ is any embedding of circles in $\Sigma$ such that
\begin{enumerate}
    \item the link projection is transverse to $\partial D_i$ for all $i=1,\dots,r-1$,
    \item each point in $P_n$ is on a circle in $L$,
    \item each circle in $L$ intersects $\partial D_p$ in a non-empty subset of $P_n$.
\end{enumerate}
\end{definition}

We denote the circle components of an $n$-cleaved link by $\CIR(L)$. Two $n$-cleaved links are considered equivalent if they are related by an isotopy of $\Sigma$ that pointwise fixes $\partial D_p$. Let $\wh{\CC\LL}_n$ denote the set of equivalence classes of $n$-cleaved links.

We see that $L=(L\cap\overleftarrow{\U})\cup_{P_n}(L\cap\overrightarrow{\U})$ is the result of gluing a crossingless matching of $P_n$ in $\overleftarrow{\U}$ to a crossingless matching in $\overrightarrow{\U}$.

The constituents of $L$ are the crossingless matchings $\overleftarrow{L}=L\cap\overleftarrow{\U}$ and $\overrightarrow{L}=L\cap\overrightarrow{\U}$ that are glued to obtain $L$.

\begin{definition}
A decoration for an $n$-cleaved link $L$ is a map $\sigma:\CIR(L)\to\{+,-\}$.
\end{definition}

Denote by $\CC\LL_n$ the set of decorated $n$-cleaved links, that is
$$\CC\LL_n=\{(L,\sigma)\ \vert\ L\in\wh{\CC\LL}_n,\ \sigma\text{ is a decoration for }L\}.$$ The vertex set of $\Gamma_n$ is the set $\CC\LL_n$.

Given a decorated $n$-cleaved link $(L,\sigma)$, the restrictions of $\sigma$ to $\overleftarrow{L}$ and $\overrightarrow{L}$ assign decorations to each arc. Denote these pairs by $(\overleftarrow{L},\sigma)$ and $(\overrightarrow{L},\sigma)$. The sets of these restrictions is denoted $\overleftarrow{\CC\LL}_n$ and $\overrightarrow{\CC\LL}_n$.

\subsection{The edges of \texorpdfstring{$\Gamma_n$}{Gamma_n}}

The edge set of $\Gamma_n$ is split into a set of left edges and a set of right edges. Each set contains bridge edges and decoration edges. 

For a link $L\in\wh{\CC\LL}_n$, a bridge for $L$ means a bridge for either the crossingless matching $\overleftarrow{L}$ or $\overrightarrow{L}$. In particular, a bridge for $L$ must occur on either the left or right half, they cannot cross $\partial D_p$, and isotopies of bridges must also occur on one side or the other, one cannot slide any part of a bridge across $\partial D_p$. We write the set of equivalence classes of bridges for $L$ in the left and right halves as $\overleftarrow{\Br}(L)$ and $\overrightarrow{\Br}(L)$, respectively (these are not to be confused with $\overrightarrow{\Br}(r)$ for a resolution $r$). The set of equivalence classes of bridges for $L$ is denoted $\BRIDGE(L)=\overleftarrow{\Br}(L)\cup\overrightarrow{\Br}(L)$. Let $L_\gamma$ denote the $n$-cleaved link found by surgering on $\gamma$ in the appropriate consitutuent of $L$.

If $r=(\rho,\overleftarrow{m})$ is a resolution for a tangle $\overrightarrow{T}$ with $\cl(r)=L\in\wh{\CC\LL_n}$ the $n$-cleaved link formed by the resolution diagram, we let $\BRIDGE(r)=\overleftarrow{\Br}(L)\cup\overrightarrow{\Br}(r)$ and $\overleftarrow{\Br}(r)=\overleftarrow{\Br}(L)$. There is a natural map $\cl:\BRIDGE(r)\to\BRIDGE(\cl(r))$ that takes a bridge and sends it to its equivalence class in $\BRIDGE(L)$.

Given any bridge $\gamma$ for $L$, where $\gamma$ is in $\U$, define
\begin{enumerate}
    \item $B_\pitchfork(L,\gamma)$ is the set of classes in $\BRIDGE(L)\setminus\{\gamma\}$ all of whose representatives intersect $\gamma$,
    \item $B_\parallel(L,\gamma)$ is the set of classes in $\BRIDGE(L)\setminus\{\gamma\}$ admitting a representative which does not intersect $\gamma$.
\end{enumerate}

An arc $a$ of a crossingless matching has a tubular neighborhood in $\U$ with two distinct sides. An end of a bridge $\gamma:[0,1]\to\U$ is on one of the two sides of the arc $a$ depending on which side of the tubular neighborhood contains $\gamma((0,\epsilon))$ or $\gamma((1-\epsilon,1))$ for sufficiently small $\epsilon>0$. We define

\begin{enumerate}
    \item $B_d(L,\gamma)$ is the subset of $B_\parallel(L,\gamma)$ admitting a representative neither of whose ends is on an arc with $\gamma$,
    \item $B_s(L,\gamma)$ is the subset of $B_\parallel(L,\gamma)$ admitting a representative with a single end on the same arc as $\gamma$ and lying on the same side of the arc as $\gamma$
    \item $B_o(L,\gamma)$ is the subset of $B_\parallel(L,\gamma)$ admitting a representative with a single end on the same arc as $\gamma$ and lying on the opposite side of the arc as $\gamma$.
\end{enumerate}

\begin{figure}[h]
    \centering
    \begin{subfigure}[b]{0.4\textwidth}
    \centering
    \def\svgwidth{.925\linewidth}
    %% Creator: Inkscape 1.1 (c68e22c387, 2021-05-23), www.inkscape.org
%% Accompanies image file 'bridge-same.pdf' (pdf, eps, ps)
\begingroup%
  \makeatletter%
  \providecommand\color[2][]{%
    \errmessage{(Inkscape) Color is used for the text in Inkscape, but the package 'color.sty' is not loaded}%
    \renewcommand\color[2][]{}%
  }%
  \providecommand\transparent[1]{%
    \errmessage{(Inkscape) Transparency is used (non-zero) for the text in Inkscape, but the package 'transparent.sty' is not loaded}%
    \renewcommand\transparent[1]{}%
  }%
  \providecommand\rotatebox[2]{#2}%
  \newcommand*\fsize{\dimexpr\f@size pt\relax}%
  \newcommand*\lineheight[1]{\fontsize{\fsize}{#1\fsize}\selectfont}%
  \ifx\svgwidth\undefined%
    \setlength{\unitlength}{203.50304311bp}%
    \ifx\svgscale\undefined%
      \relax%
    \else%
      \setlength{\unitlength}{\unitlength * \real{\svgscale}}%
    \fi%
  \else%
    \setlength{\unitlength}{\svgwidth}%
  \fi%
  \global\let\svgwidth\undefined%
  \global\let\svgscale\undefined%
  \makeatother%
  \begin{picture}(1,0.68563953)%
    \lineheight{1}%
    \setlength\tabcolsep{0pt}%
    \put(0,0){\includegraphics[width=\unitlength,page=1]{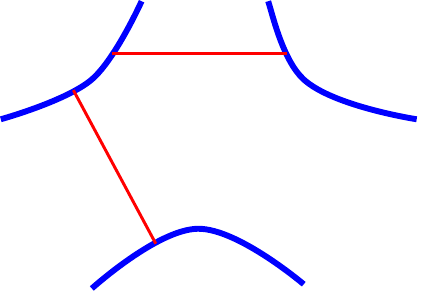}}%
    \put(0.40777857,0.58548279){\makebox(0,0)[lt]{\lineheight{1.25}\smash{\begin{tabular}[t]{l}$\gamma$\end{tabular}}}}%
    \put(0.31887925,0.2718192){\makebox(0,0)[lt]{\lineheight{1.25}\smash{\begin{tabular}[t]{l}$\eta$\end{tabular}}}}%
    \put(0,0){\includegraphics[width=\unitlength,page=2]{bridge-same.pdf}}%
    \put(0.71490975,0.26691178){\makebox(0,0)[lt]{\lineheight{1.25}\smash{\begin{tabular}[t]{l}$\delta$\end{tabular}}}}%
  \end{picture}%
\endgroup%

    \caption{The bridges $\gamma$ and $\eta$ satisfy $\eta\in B_s(L,\gamma)$. The third bridge $\delta$ is obtained by sliding an end of $\gamma$ over $\eta$. After surgery on $\gamma$, the images of $\eta$ and $\delta$ become isotopic.}
    \label{fig:bridge-same}
\end{subfigure}
\hfill
    \begin{subfigure}[b]{0.4\textwidth}
    \centering
    \def\svgwidth{.925\linewidth}
    %% Creator: Inkscape 1.1 (c68e22c387, 2021-05-23), www.inkscape.org
%% Accompanies image file 'bridge-opposite.pdf' (pdf, eps, ps)
\begingroup%
  \makeatletter%
  \providecommand\color[2][]{%
    \errmessage{(Inkscape) Color is used for the text in Inkscape, but the package 'color.sty' is not loaded}%
    \renewcommand\color[2][]{}%
  }%
  \providecommand\transparent[1]{%
    \errmessage{(Inkscape) Transparency is used (non-zero) for the text in Inkscape, but the package 'transparent.sty' is not loaded}%
    \renewcommand\transparent[1]{}%
  }%
  \providecommand\rotatebox[2]{#2}%
  \newcommand*\fsize{\dimexpr\f@size pt\relax}%
  \newcommand*\lineheight[1]{\fontsize{\fsize}{#1\fsize}\selectfont}%
  \ifx\svgwidth\undefined%
    \setlength{\unitlength}{221.7477441bp}%
    \ifx\svgscale\undefined%
      \relax%
    \else%
      \setlength{\unitlength}{\unitlength * \real{\svgscale}}%
    \fi%
  \else%
    \setlength{\unitlength}{\svgwidth}%
  \fi%
  \global\let\svgwidth\undefined%
  \global\let\svgscale\undefined%
  \makeatother%
  \begin{picture}(1,0.62995547)%
    \lineheight{1}%
    \setlength\tabcolsep{0pt}%
    \put(0,0){\includegraphics[width=\unitlength,page=1]{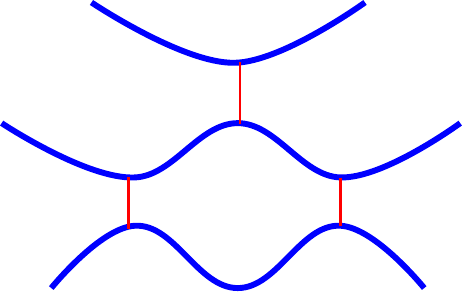}}%
    \put(0.55716811,0.38429631){\makebox(0,0)[lt]{\lineheight{1.25}\smash{\begin{tabular}[t]{l}$\gamma$\end{tabular}}}}%
    \put(0.32318851,0.1518559){\makebox(0,0)[lt]{\lineheight{1.25}\smash{\begin{tabular}[t]{l}$\eta$\end{tabular}}}}%
    \put(0.77729366,0.14877733){\makebox(0,0)[lt]{\lineheight{1.25}\smash{\begin{tabular}[t]{l}$\eta$\end{tabular}}}}%
  \end{picture}%
\endgroup%

    \caption{The bridges $\gamma$ and $\eta$ satisfy $\eta\in B_o(L,\gamma)$. After surgery on $\gamma$, the images of the two isotopic bridges labeled $\eta$ become distinct bridges.}
    \label{fig:bridge-opposite}
\end{subfigure}
    \caption{Bridges that have nondisjoint support.}
    \label{fig:bridge-nondisjoint}
\end{figure}

It is clear that $\BRIDGE(L)\setminus\{\gamma\}$ is the disjoint union of $B_\pitchfork(L,\gamma)$ and $B_\parallel(L,\gamma)$, and $B_\parallel(L,\gamma)$ is the disjoint union of $B_d(L,\gamma)$, $B_s(L,\gamma)$, and $B_o(L,\gamma)$. If $\eta\in B_*(L,\gamma)$ where $*$ is any of the five subscripts as above then $\gamma\in B_*(L,\eta)$. 

The following proposition and its proof are the same as \cite[Proposition 13]{roberts-type-d}.

\begin{proposition}\cite[Proposition 13]{roberts-type-d}\label{prop:same-other-identification}
Surgery on $\gamma$ induces an identification of $B_d(L,\gamma)$ with $B_d(L,\gamma)$ and a $2:1$ map $B_s(L,\gamma)\to B_o(L_\gamma,\gamma^\dagger)$. Dually, there is a $2:1$ map $B_s(L_\gamma,\gamma^\dagger)\to B_o(L,\gamma)$. 
\end{proposition}

The $2:1$ maps in the proposition are as follows. In the case where $\eta\in B_s(L,\gamma)$, there is a third bridge $\delta$ obtained by sliding an end of $\gamma$ over $\eta$. Specifically, as shown in Figure~\ref{fig:bridge-same}, $\delta$ is constructed by taking the concatenation of a pushoff of $\gamma$ on the same side as $\eta$, a segment parallel to the arc of the matching between the foot of $\gamma$ and the foot of $\eta$, and a pushoff of $\eta$ on the same side as $\gamma$. The bridges $\gamma,\eta$, and $\delta$ together with the segments of the arcs connecting their feet bound an embedded hexagon in $\U$. After surgery on $\gamma$, the images of $\eta$ and $\delta$ belong in $B_o(L_\gamma,\gamma^\dagger)$ and are isotopic.

When $\eta\in B_o(L,\gamma)$, take a pushoff of $\eta$ on the same side as $\gamma$. One of the ends of the pushoff is on the same arc as an end of $\gamma$. Slide that end of the pushoff along the arc until it has gone past the end of $\gamma$. This creates a bridge $\eta'$ that is isotopic to $\eta$. After surgery on $\gamma$, the images of $\eta$ and $\eta'$ will be non-isotopic and both in $B_s(L_\gamma,\gamma^\dagger)$.

For any bridge $\gamma\in\BRIDGE(L)$, we call the set of circles in $L$ and $L_\gamma$ that are abutted by the feet of $\gamma$ and $\gamma^\dagger$ the support of $\gamma$.

\begin{definition}
Let $(L,\sigma)\in\CC\LL_n$ and $\gamma\in\BRIDGE(L)$. A bridge edge $\gamma:(L,\sigma)\to(L',\sigma')$ of $\Gamma_n$ exists when  $L'=L_\gamma$, $\sigma'=\sigma_\gamma$ satisfies $\sigma_\gamma(C)=\sigma(C)$ for all $C\in\CIR(L)$ not in the support of $\gamma$, and for circles in the support of $\gamma$ the decorations follow the rules of \cite{aps-i-bundles} Table 2.1, that is, they match one of the following cases (note that our convention for $+$ and $-$ decorations is opposite that of \cite{aps-i-bundles}):
\begin{enumerate}
    \item when $\gamma$ merges two circles $C_a$ and $C_b$ to get a circle $C_\gamma$, $\sigma$ and $\sigma_\gamma$ satisfy one of the following
    \begin{align}
    \begin{split}
        \sigma(C_a)=+,\ \sigma(C_b)=+, &\hspace{10mm} \sigma_\gamma(C_\gamma)=+, \\
        \sigma(C_a)=-,\ \sigma(C_b)=+, &\hspace{10mm} \sigma_\gamma(C_\gamma)=-, \\
        \sigma(C_a)=+,\ \sigma(C_b)=-, &\hspace{10mm} \sigma_\gamma(C_\gamma)=-.
    \end{split}
    \end{align}
    \item when $\gamma$ divides a circle $C$ into two circles $C_a$ and $C_b$, $\sigma$ and $\sigma_\gamma$ satisfy one of the following
    \begin{align}
    \begin{split}
        \sigma(C)=+, &\hspace{10mm} \sigma_\gamma(C_a)=+,\ \sigma_\gamma(C_b)=-, \\
        \sigma(C)=+, &\hspace{10mm} \sigma_\gamma(C_a)=-,\ \sigma_\gamma(C_b)=+, \\
        \sigma(C)=-, &\hspace{10mm} \sigma_\gamma(C_a)=-,\ \sigma_\gamma(C_b)=-.
    \end{split}
    \end{align}
    
\end{enumerate}
\end{definition}

There are no edges for 1-1 bifurcations. 

\begin{remark}
We chose to have zero differential corresponding to 1-1 bifurcations, but other choices are possible. They must satisfy the condition of being a dyad, as defined in \cite{chen-rp3}, that is, a choice of a map $f:V\to V$, where $V$ is the 2-dimensional $\F_2$-vector space $V=\langle v_+,v_-\rangle$, satisfying $f\circ f=0$.
\end{remark}

\begin{definition}
For each circle $C\in\CIR(L)$ with $\sigma(C)=+$, there are two distinct decoration edges $(L,\sigma)\to(L,\sigma_C)$, where $\sigma_C$ is the decoration on $L$ with $\sigma_C(C)=-$ and $\sigma_C(C')=\sigma(C')$ for all $C'\in\CIR(L)\setminus\{C\}$.
\end{definition}

The circle whose decoration changes along a decoration edge will be called the support of that edge.

We partition the bridge edges into those coming from $\overleftarrow{L}$ and those from $\overrightarrow{L}$. We also assign one of the two decoration edges for each circle $C\in\CIR(L)$ with $\sigma(C)=+$ to $\overleftarrow{L}$ and one to $\overrightarrow{L}$. Let $\overleftarrow{\Gamma}_n$ be the subgraph of $\Gamma_n$ with the same vertices but whose edges correspond to $\overleftarrow{L}$ and similarly for $\overrightarrow{\Gamma}_n$. The location of an edge is the subgraph $\overleftarrow{\Gamma}_n$ or $\overrightarrow{\Gamma}_n$ that contains the edge.

\begin{definition}
$\QQ\Gamma_n$ is the quiver category induced from the graph $\Gamma_n$.
\end{definition}

We use $I_{(L,\sigma)}$ to denote the idempotent of $\QQ\Gamma_n$ corresponding to the vertex $(L,\sigma)$ of $\Gamma_n$. The ring of idempotents of $\QQ\Gamma_n$ is
$$\II=\II_n=\{I_{(L,\sigma)}\ :\ (L,\sigma)\in\CC\LL_n\}.$$

\begin{proposition}
$\QQ\Gamma_n$ is finite dimensional.
\end{proposition}

\begin{proof}
The proof is similar to the proof of \cite[Proposition 18]{roberts-type-d}. Define
$$\iota(L,\sigma)=\#\{C\in\CIR(L):\sigma(C)=+\}-\#\{C\in\CIR(L):\sigma(C)=-\}.$$ We have defined bridge edges to exist only if they decrease $\iota$ by $1$, and a decoration edge for a circle $C$ with $\sigma(C)=+$ decreases $\iota$ by $2$. Therefore, the value of $\iota$ strictly decreases along a directed path in $\Gamma$, so there cannot be any directed cycles.
\end{proof}

For a bridge edge $\gamma:(L,\sigma)\to(L_\gamma,\sigma')$ as above, we use $e_{(\gamma,\sigma,\sigma')}$ to denote the algebra element it corresponds to in $\QQ\Gamma_n$. When we do not need to specify the decorations $\sigma$ and $\sigma'$, we sometimes write $e_\gamma$. If $\gamma\in\overleftarrow{\Br}(L)$ we write $\overleftarrow{e}_{(\gamma,\sigma,\sigma')}$, and likewise for $\overrightarrow{\Br}(L)$. 

For a decoration edge $(L,\sigma)\to (L,\sigma_C)$, we write $\overleftarrow{e}_C$ for the decoration edge that has been assigned to $\overleftarrow{L}$ and $\overrightarrow{e}_C$ for the edge that has been assigned to $\overrightarrow{L}$.

\begin{definition}
The bigrading on $\QQ\Gamma_n$ is the bigrading of paths in $\Gamma_n$ induced by homomorphically extending the following bigrading of the edges:
\begin{align*}
    I_{(L,\sigma)}&\longrightarrow(0,0), \\
    \overrightarrow{e}_C&\longrightarrow(0,-1), \\
    \overleftarrow{e}_C&\longrightarrow(1,1), \\
    \overrightarrow{e}_\gamma&\longrightarrow(0,-1/2), \\
    \overleftarrow{e}_\gamma&\longrightarrow(1,1/2).
\end{align*}
\end{definition}

The first entry of the bigrading on the path $\alpha$ will be denoted by $\overleftarrow{l}(\alpha)$ while the second entry will be denoted $q(\alpha)$. Elements corresponding to $\overrightarrow{\U}$ will act as even elements for the $\Z/2\Z$-grading from $\overleftarrow{l}(\alpha)$, while elements in $\overleftarrow{\U}$ will be odd elements.

\subsection{The cleaver category}

We now describe relations we impose on $\QQ\Gamma_n$, and the pre-additive category we obtain by quotienting by these relations will be called the cleaver category $\BB\Gamma_n$. As each of these relations will be homogeneous for the bigrading, $\BB\Gamma_n$ will also be bigraded.

Most of the relations can be characterized by the following: suppose $(L,\sigma)$ and $(L',\sigma')$ are decorated $n$-cleaved links. Then we impose the relation that any length $2$ path from $(L,\sigma)$ to $(L',\sigma')$ involving only right edges are equal. We also impose the relation that the sum over all length $2$ paths from $(L,\sigma)$ to $(L',\sigma')$ involving only left edges is equal to $0$. The situations involving length $1$ paths and mixtures of left and right edges are slightly more nuanced, but follow the same general pattern.

Many of the relations come from commuting squares in $\Gamma_n$ of the form in Figure~\ref{fig:commuting-square}, where $e_\alpha$ and $e_\beta$ can be either bridge or decoration edges, and $e_{\alpha'}$ and $e_{\beta'}$ are in the same locations as $e_\alpha$ and $e_\beta$, respectively. For each we will impose the relation 
$$e_\alpha e_{\beta'} = e_\beta e_{\alpha'}.$$

\begin{figure}[h]
\centering
\begin{tikzcd}
{(L,\sigma)} \arrow[rr, "e_\alpha"] \arrow[d, "e_\beta"] &  & {(L_\alpha,\sigma_\alpha)} \arrow[d, "e_{\beta'}"] \\
{(L_\beta,\sigma_\beta)} \arrow[rr, "e_{\alpha'}"]       &  & {(L_2,\sigma_2)}                                  
\end{tikzcd}
\caption{A commuting square}
\label{fig:commuting-square}
\end{figure}

\begin{enumerate}
    \item \textbf{The relations for disjoint support:} Suppose $(L,\sigma)\xrightarrow{e_\alpha}(L_\alpha,\sigma_\alpha)$ and $(L,\sigma)\xrightarrow{e_\beta}(L_\beta,\sigma_\beta)$ are edges with the same source and disjoint supports. Whether $e_\alpha$ and $e_\beta$ are bridge edges or decoration edges, applying the change to the cleaved link from $e_\alpha$ commutes with applying the change from $e_\beta$, so there is a square as in Figure~\ref{fig:commuting-square}.

    More precisely, if $e_\alpha=e_C$ and $e_\beta=e_D$ for distinct cleaved circles $C,D$ with $\sigma(C)=\sigma(D)=+$, then changing the decoration on $C$ commutes with changing the decoration on $D$. Let $e_{\alpha'}=e_C$ and $e_{\beta'}=e_D$ in the same locations as $e_\alpha$ and $e_\beta$; we obtain the relation
    \begin{equation}\label{relation-eCeC}
        e_Ce_D = e_De_C.
    \end{equation}

    If $e_\alpha=e_{(\gamma,\sigma,\sigma_\gamma)}$ and $e_\beta=e_C$ for $C\in\CIR(L)$ with $\sigma(C)=+$ and $\gamma\in\BRIDGE(L)$ whose support does not include $C$. Then $\sigma_\gamma(C)=+$ as well, so there is a decoration edge $e_C:(L_\gamma,\sigma_\gamma)\to(L_\gamma,\sigma_{\gamma,C})$ with the same location as $e_\beta$, where $\sigma_{\gamma,C}=\sigma_{C,\gamma}$ is equal to $\sigma_\gamma$ for all $C'\in\CIR(L_\gamma)\setminus\{C\}$ and $\sigma_{\gamma,C}(C)=-$. Call this edge $e_{\beta'}$. On the other hand, there is a bridge edge $e_{\alpha'}=e_{(\gamma,\sigma_C,\sigma_{C,\gamma})}:(L,\sigma_C)\to(L_\gamma,\sigma_{C,\gamma})$. We obtain a square and the relation
    \begin{equation}\label{relation-eCbridge}
        e_\gamma e_C = e_C e_\gamma.
    \end{equation}

    If $e_\alpha=e_\gamma$ and $e_\beta=e_\eta$ for distinct equivalence classes of bridges $\gamma,\eta$ for $L$ with disjoint supports, then $\eta\in\BRIDGE(L_\gamma)$ and $\gamma\in\BRIDGE(L_\eta)$. Furthermore, $L_{\gamma,\eta}=L_{\eta,\gamma}$. Since the supports are disjoint, there is an edge $e_{\beta'}:(L_\gamma,\sigma_\gamma)\to(L_{\gamma,\eta},\sigma_{\gamma,\eta})$ where $\sigma_{\gamma,\eta}$ assigns the same decorations to circles in the support of $\gamma$ as $\sigma_\gamma$ and to circles in the support of $\eta$ as $\sigma_\eta$. The same argument shows there is an edge $e_{\alpha'}:(L_\eta,\sigma_\eta)\to(L_{\eta,\gamma},\sigma_{\eta,\gamma})$, and $\sigma_{\gamma,\eta}=\sigma_{\eta,\gamma}$. There is a square and we obtain the relation 
    \begin{equation}\label{relation-bridgebridge}
        e_\gamma e_\eta=e_\eta e_\gamma.
    \end{equation}
    
    \item \textbf{The relations for $\overrightarrow{e}_C$:} Suppose first that $\gamma$ merges $C_1$ and $C_2$ to get $C\in\CIR(L_\gamma)$, and $\sigma(C_1)=\sigma(C_2)=+$. The following squares exist in $\Gamma_n$ for $i=1,2$:

\begin{figure}[h]
\centering
\begin{tikzcd}
{(L,\sigma)} \arrow[rr, "e_{\gamma}"] \arrow[d, "e_{C_i}"] &  & {(L_\gamma,\sigma_{\gamma})} \arrow[d, "e_{C}"] \\
{(L,\sigma_{C_i})} \arrow[rr, "e_{\gamma}"]       &  & {(L_{\gamma},\sigma_{\gamma,C})}                                  
\end{tikzcd}
\caption{Square for a merge $\gamma$ and $e_C$}
\label{fig:square-dec-bridge-merge}
\end{figure}

    Both of these squares provide relations, giving us
    
    \begin{equation}\label{relation-righteCmerge}
    \overrightarrow{e}_{C_1}e_{(\gamma,\sigma_{C_1},\sigma_{\gamma,C})} = \overrightarrow{e}_{C_2}e_{(\gamma,\sigma_{C_2},\sigma_{\gamma,C})}=e_{(\gamma,\sigma,\sigma_\gamma)}\overrightarrow{e}_C.
    \end{equation}

    If $\sigma(C_i)=-$ for $i=1$ or $2$, then no such squares exist in $\Gamma_n$.

    Similarly, if $\gamma$ divides circle $C$ into $C_1$ and $C_2$ in $\CIR(L_\gamma)$, and $\sigma(C)=+$, then there are squares

\begin{figure}[h]
\centering
\begin{tikzcd}
{(L,\sigma)} \arrow[rr, "e_{\gamma}"] \arrow[d, "e_{C}"] &  & {(L_\gamma,\sigma_{\gamma})} \arrow[d, "e_{C_i}"] \\
{(L,\sigma_{C})} \arrow[rr, "e_{C,\gamma}"]       &  & {(L_{\gamma},\sigma_{C_i})}                                  
\end{tikzcd}
\caption{Square for a division $\gamma$ and $e_C$}
\label{fig:square-dec-bridge-divide}
\end{figure}

    and the relation
    \begin{equation}\label{relation-righteCdivide}
        \overrightarrow{e}_Ce_{(\gamma,\sigma_C,\sigma_{C,\gamma})}=e_{(\gamma,\sigma,\sigma_\gamma^1)}\overrightarrow{e}_{C_1}=e_{(\gamma,\sigma,\sigma_\gamma^2)}\overrightarrow{e}_{C_2},
    \end{equation}
    where $\sigma_\gamma^i$ assigns $+$ to $C_i$ and $-$ to $C_{i+1}$ (indices counted mod $2$). There are no such relations if $\sigma(C)=-$.

    \item \textbf{The relations for $\overleftarrow{e}_C$:} The squares found for $\overrightarrow{e}_C$ still exist for $\overleftarrow{e}_C$. However, the relations we define are different. When $\gamma$ merges $C_1$ and $C_2$ with $\sigma(C_1)=\sigma(C_2)=+$ to get $C$, we impose the relation
    \begin{equation}\label{relation-lefteCmerge}
        \overleftarrow{e}_{C_1}e_{(\gamma,\sigma_{C_1},\sigma_{\gamma,C})} + \overleftarrow{e}_{C_2}e_{(\gamma,\sigma_{C_2},\sigma_{\gamma,C})} + e_{(\gamma,\sigma,\sigma_\gamma)}\overleftarrow{e}_C=0.
    \end{equation}
    When $\gamma$ divides $C$ with $\sigma(C)=+$ into $C_1$ and $C_2$, we impose the relation
    \begin{equation}\label{relation-lefteCdivide}
        \overleftarrow{e}_Ce_{(\gamma,\sigma_C,\sigma_{C,\gamma})} + e_{(\gamma,\sigma,\sigma_\gamma^1)}\overleftarrow{e}_{C_1} + e_{(\gamma,\sigma,\sigma_\gamma^2)}\overleftarrow{e}_{C_2}=0.
    \end{equation}

    \begin{figure}[h]
        \centering
        \begin{subfigure}[b]{0.45\textwidth}
    \centering
    \def\svgwidth{.925\linewidth}
    %% Creator: Inkscape 1.1 (c68e22c387, 2021-05-23), www.inkscape.org
%% Accompanies image file 'merge-decoration.pdf' (pdf, eps, ps)
\begingroup%
  \makeatletter%
  \providecommand\color[2][]{%
    \errmessage{(Inkscape) Color is used for the text in Inkscape, but the package 'color.sty' is not loaded}%
    \renewcommand\color[2][]{}%
  }%
  \providecommand\transparent[1]{%
    \errmessage{(Inkscape) Transparency is used (non-zero) for the text in Inkscape, but the package 'transparent.sty' is not loaded}%
    \renewcommand\transparent[1]{}%
  }%
  \providecommand\rotatebox[2]{#2}%
  \newcommand*\fsize{\dimexpr\f@size pt\relax}%
  \newcommand*\lineheight[1]{\fontsize{\fsize}{#1\fsize}\selectfont}%
  \ifx\svgwidth\undefined%
    \setlength{\unitlength}{475.45160249bp}%
    \ifx\svgscale\undefined%
      \relax%
    \else%
      \setlength{\unitlength}{\unitlength * \real{\svgscale}}%
    \fi%
  \else%
    \setlength{\unitlength}{\svgwidth}%
  \fi%
  \global\let\svgwidth\undefined%
  \global\let\svgscale\undefined%
  \makeatother%
  \begin{picture}(1,0.18656845)%
    \lineheight{1}%
    \setlength\tabcolsep{0pt}%
    \put(0,0){\includegraphics[width=\unitlength,page=1]{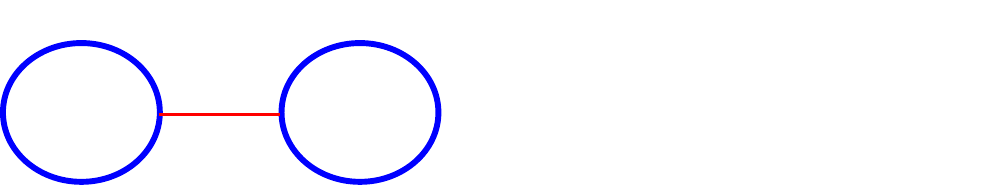}}%
    \put(0.02734867,0.16390902){\makebox(0,0)[lt]{\lineheight{1.25}\smash{\begin{tabular}[t]{l}$C_1$\end{tabular}}}}%
    \put(0.32179112,0.16187841){\makebox(0,0)[lt]{\lineheight{1.25}\smash{\begin{tabular}[t]{l}$C_2$\end{tabular}}}}%
    \put(0.1857384,0.08674481){\makebox(0,0)[lt]{\lineheight{1.25}\smash{\begin{tabular}[t]{l}$\gamma$\end{tabular}}}}%
    \put(0.04257846,0.03800953){\makebox(0,0)[lt]{\lineheight{1.25}\smash{\begin{tabular}[t]{l}$+$\end{tabular}}}}%
    \put(0.34615878,0.026841){\makebox(0,0)[lt]{\lineheight{1.25}\smash{\begin{tabular}[t]{l}$+$\end{tabular}}}}%
    \put(0,0){\includegraphics[width=\unitlength,page=2]{merge-decoration.pdf}}%
    \put(0.72994923,0.14157201){\makebox(0,0)[lt]{\lineheight{1.25}\smash{\begin{tabular}[t]{l}$C$\end{tabular}}}}%
    \put(0.74416367,0.02277969){\makebox(0,0)[lt]{\lineheight{1.25}\smash{\begin{tabular}[t]{l}$+$\end{tabular}}}}%
    \put(0,0){\includegraphics[width=\unitlength,page=3]{merge-decoration.pdf}}%
  \end{picture}%
\endgroup%

    \caption{A bridge that merges two circles with $+$ decoration.}
    \label{fig:merge-decoration}
\end{subfigure}
\hfill
    \begin{subfigure}[b]{0.45\textwidth}
    \centering
    \def\svgwidth{.925\linewidth}
    %% Creator: Inkscape 1.1 (c68e22c387, 2021-05-23), www.inkscape.org
%% Accompanies image file 'divide-decoration.pdf' (pdf, eps, ps)
\begingroup%
  \makeatletter%
  \providecommand\color[2][]{%
    \errmessage{(Inkscape) Color is used for the text in Inkscape, but the package 'color.sty' is not loaded}%
    \renewcommand\color[2][]{}%
  }%
  \providecommand\transparent[1]{%
    \errmessage{(Inkscape) Transparency is used (non-zero) for the text in Inkscape, but the package 'transparent.sty' is not loaded}%
    \renewcommand\transparent[1]{}%
  }%
  \providecommand\rotatebox[2]{#2}%
  \newcommand*\fsize{\dimexpr\f@size pt\relax}%
  \newcommand*\lineheight[1]{\fontsize{\fsize}{#1\fsize}\selectfont}%
  \ifx\svgwidth\undefined%
    \setlength{\unitlength}{379.12650521bp}%
    \ifx\svgscale\undefined%
      \relax%
    \else%
      \setlength{\unitlength}{\unitlength * \real{\svgscale}}%
    \fi%
  \else%
    \setlength{\unitlength}{\svgwidth}%
  \fi%
  \global\let\svgwidth\undefined%
  \global\let\svgscale\undefined%
  \makeatother%
  \begin{picture}(1,0.19577147)%
    \lineheight{1}%
    \setlength\tabcolsep{0pt}%
    \put(0,0){\includegraphics[width=\unitlength,page=1]{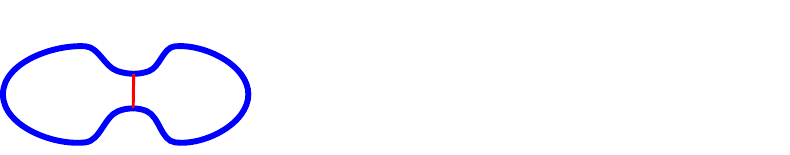}}%
    \put(0.02665745,0.13713067){\makebox(0,0)[lt]{\lineheight{1.25}\smash{\begin{tabular}[t]{l}$C$\end{tabular}}}}%
    \put(0.04957654,0.04036126){\makebox(0,0)[lt]{\lineheight{1.25}\smash{\begin{tabular}[t]{l}$+$\end{tabular}}}}%
    \put(0.18136121,0.06964661){\makebox(0,0)[lt]{\lineheight{1.25}\smash{\begin{tabular}[t]{l}$\gamma$\end{tabular}}}}%
    \put(0,0){\includegraphics[width=\unitlength,page=2]{divide-decoration.pdf}}%
    \put(0.59245055,0.16608169){\makebox(0,0)[lt]{\lineheight{1.25}\smash{\begin{tabular}[t]{l}$C_1$\end{tabular}}}}%
    \put(0.85365414,0.16735489){\makebox(0,0)[lt]{\lineheight{1.25}\smash{\begin{tabular}[t]{l}$C_2$\end{tabular}}}}%
    \put(0,0){\includegraphics[width=\unitlength,page=3]{divide-decoration.pdf}}%
  \end{picture}%
\endgroup%

    \caption{A bridge that divides a circle with $+$ decoration into two circles.}
    \label{fig:divide-decoration}
\end{subfigure}
        \caption{}
        \label{fig:merge-divide-decoration}
    \end{figure}

    \item \textbf{Bridge relations from squares:} Suppose $\gamma\in\BRIDGE(L)$ and $\eta\in B_{\parallel}(L,\gamma)$, then $\overline\eta\in\BRIDGE(L_\gamma)$, $\overline\gamma\in\BRIDGE(L_\eta)$, and $L_{\gamma,\eta}=L_{\eta,\gamma}$. There are squares in $\Gamma_n$ as in Figure~\ref{fig:square-bridges} as long as there are decorations for which the edges in the square exist.

\begin{figure}[h]
\centering
\begin{tikzcd}
{(L,\sigma)} \arrow[rr, "e_{(\gamma,\sigma_{00},\sigma_{01})}"] \arrow[d, "e_{(\eta,\sigma_{00},\sigma_{10})}"] &  & {(L_\gamma,\sigma_{01})} \arrow[d, "e_{(\overline\eta,\sigma_{01},\sigma_{11})}"] \\
{(L_\eta,\sigma_{10})} \arrow[rr, "e_{(\overline\gamma,\sigma_{10},\sigma_{11})}"]       &  & {(L_{\gamma,\eta},\sigma_{11})}                                  
\end{tikzcd}
\caption{Square for bridges}
\label{fig:square-bridges}
\end{figure}

The following proposition is a consequence of the more general Theorem~\ref{thm:existence-squares} that we will prove later, but we put it here to set up notation.

    \begin{proposition}\cite[Proposition 20]{roberts-type-d} 
    Let $(r,s)$ be a state of some tangle $T$. If $\gamma$ and $\eta$ are bridges such that the following path exists
    $$(L,\sigma)\xrightarrow{e_{(\gamma,\sigma,\sigma')}}(L_\gamma,\sigma')\xrightarrow{e_{(\overline\eta,\sigma',\sigma'')}}(L_{\gamma,\eta},\sigma''),$$
    and $\eta\in\Br(r)$ is not a 1-1 bifurcation, then there is at least one decoration $\overline\sigma$ on $L_\eta$ for which we can find edges
    $$(L,\sigma)\xrightarrow{e_{(\eta,\sigma,\overline\sigma)}}(L_\eta,\overline\sigma)\xrightarrow{e_{(\overline\gamma,\overline\sigma,\sigma'')}}(L_{\gamma,\eta},\sigma'').$$
    \end{proposition}

    For these squares, we impose the relation
    \begin{equation}\label{relation-bridgesquare}
        e_{(\gamma,\sigma,\sigma')}e_{(\overline\eta,\sigma',\sigma'')}=e_{(\eta,\sigma,\overline\sigma)}e_{(\overline\gamma,\overline\sigma,\sigma'')}
    \end{equation}
    whenever
    \begin{itemize}
        \item either of the arcs $\gamma$ or $\eta$ is in $\overrightarrow{\U}$, or
        \item the arcs $\gamma$ and $\eta$ are both in $\overleftarrow{\U}$ with $\gamma\in B_o(L,\eta)$.
    \end{itemize}
\end{enumerate}

\begin{figure}[h]
    \centering
    \def\svgwidth{.925\linewidth}
    %% Creator: Inkscape 1.1 (c68e22c387, 2021-05-23), www.inkscape.org
%% Accompanies image file 'bridge-square.pdf' (pdf, eps, ps)
\begingroup%
  \makeatletter%
  \providecommand\color[2][]{%
    \errmessage{(Inkscape) Color is used for the text in Inkscape, but the package 'color.sty' is not loaded}%
    \renewcommand\color[2][]{}%
  }%
  \providecommand\transparent[1]{%
    \errmessage{(Inkscape) Transparency is used (non-zero) for the text in Inkscape, but the package 'transparent.sty' is not loaded}%
    \renewcommand\transparent[1]{}%
  }%
  \providecommand\rotatebox[2]{#2}%
  \newcommand*\fsize{\dimexpr\f@size pt\relax}%
  \newcommand*\lineheight[1]{\fontsize{\fsize}{#1\fsize}\selectfont}%
  \ifx\svgwidth\undefined%
    \setlength{\unitlength}{486.29372424bp}%
    \ifx\svgscale\undefined%
      \relax%
    \else%
      \setlength{\unitlength}{\unitlength * \real{\svgscale}}%
    \fi%
  \else%
    \setlength{\unitlength}{\svgwidth}%
  \fi%
  \global\let\svgwidth\undefined%
  \global\let\svgscale\undefined%
  \makeatother%
  \begin{picture}(1,0.40396447)%
    \lineheight{1}%
    \setlength\tabcolsep{0pt}%
    \put(0,0){\includegraphics[width=\unitlength,page=1]{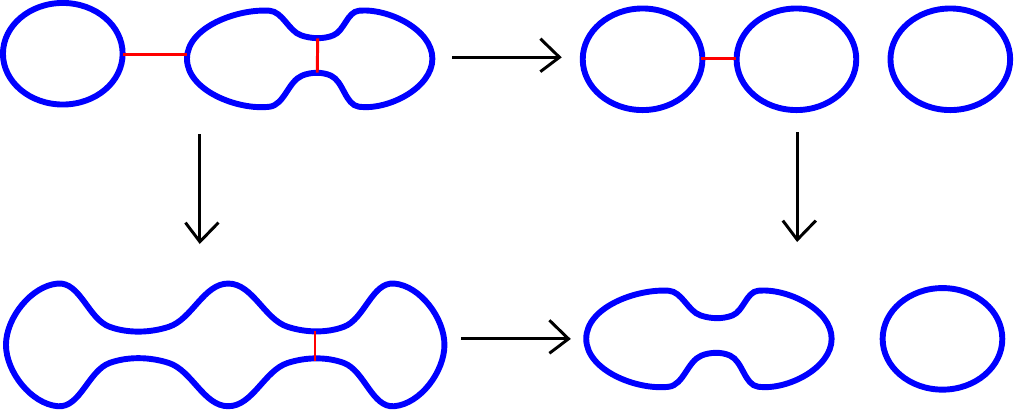}}%
    \put(0.1299778,0.37524018){\makebox(0,0)[lt]{\lineheight{1.25}\smash{\begin{tabular}[t]{l}$\gamma$\end{tabular}}}}%
    \put(0.33049952,0.34347438){\makebox(0,0)[lt]{\lineheight{1.25}\smash{\begin{tabular}[t]{l}$\eta$\end{tabular}}}}%
    \put(0.70275524,0.36729875){\makebox(0,0)[lt]{\lineheight{1.25}\smash{\begin{tabular}[t]{l}$\overline\gamma$\end{tabular}}}}%
    \put(0.33149223,0.05758198){\makebox(0,0)[lt]{\lineheight{1.25}\smash{\begin{tabular}[t]{l}$\overline\eta$\end{tabular}}}}%
  \end{picture}%
\endgroup%

    \caption{A square involving two bridges $\gamma$ and $\eta$.}
    \label{fig:bridge-square}
\end{figure}

There are also squares in the case where $\alpha$ and $\beta$ are left bridges and $\alpha\in B_s(L,\beta)$. However, these occur in triples: for pairs chosen from $\alpha$, $\beta$, and $\gamma$, the result of sliding the end of $\alpha$ over $\beta$.

In $L_\alpha$, $\overline\beta$ and $\overline\gamma$ represent the same bridge. Similarly, the images of $\alpha$ and $\gamma$ in $L_\beta$ are the same, and the images of $\alpha$ and $\beta$ in $L_\gamma$ are the same. Then we impose the relation
\begin{equation}\label{relation-bridge-same}
    \overleftarrow{e}_\alpha\overleftarrow{e}_{\overline\beta}+\overleftarrow{e}_\beta\overleftarrow{e}_{\overline\gamma}+\overleftarrow{e}_\gamma\overleftarrow{e}_{\overline\alpha}=0
\end{equation}

whenever there are compatible decorations on $L_\alpha$, $L_\beta$, and $L_\gamma$. Whenever one of the terms does not exist, we obtain commutativity for the other two terms.

\begin{figure}[h]
\centering
\def\svgwidth{.925\linewidth}
%% Creator: Inkscape 1.1 (c68e22c387, 2021-05-23), www.inkscape.org
%% Accompanies image file 'bridge-same-surgery.pdf' (pdf, eps, ps)
\begingroup%
  \makeatletter%
  \providecommand\color[2][]{%
    \errmessage{(Inkscape) Color is used for the text in Inkscape, but the package 'color.sty' is not loaded}%
    \renewcommand\color[2][]{}%
  }%
  \providecommand\transparent[1]{%
    \errmessage{(Inkscape) Transparency is used (non-zero) for the text in Inkscape, but the package 'transparent.sty' is not loaded}%
    \renewcommand\transparent[1]{}%
  }%
  \providecommand\rotatebox[2]{#2}%
  \newcommand*\fsize{\dimexpr\f@size pt\relax}%
  \newcommand*\lineheight[1]{\fontsize{\fsize}{#1\fsize}\selectfont}%
  \ifx\svgwidth\undefined%
    \setlength{\unitlength}{605.94093623bp}%
    \ifx\svgscale\undefined%
      \relax%
    \else%
      \setlength{\unitlength}{\unitlength * \real{\svgscale}}%
    \fi%
  \else%
    \setlength{\unitlength}{\svgwidth}%
  \fi%
  \global\let\svgwidth\undefined%
  \global\let\svgscale\undefined%
  \makeatother%
  \begin{picture}(1,0.23608262)%
    \lineheight{1}%
    \setlength\tabcolsep{0pt}%
    \put(0,0){\includegraphics[width=\unitlength,page=1]{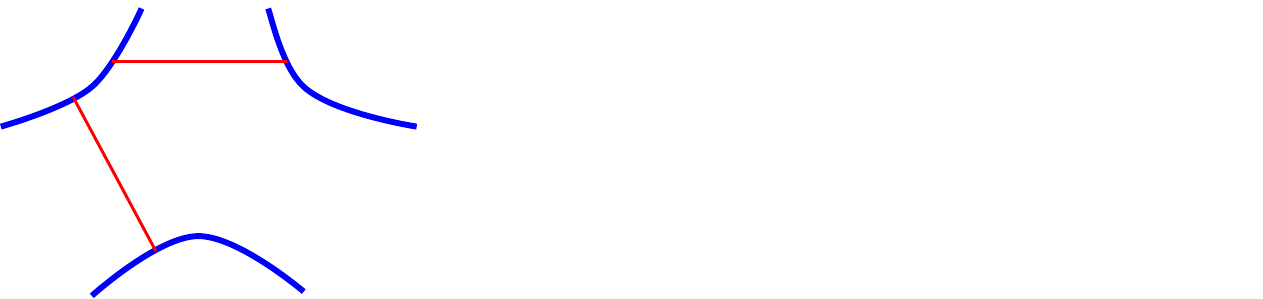}}%
    \put(0.13695095,0.19663227){\makebox(0,0)[lt]{\lineheight{1.25}\smash{\begin{tabular}[t]{l}$\alpha$\end{tabular}}}}%
    \put(0.10709444,0.0912895){\makebox(0,0)[lt]{\lineheight{1.25}\smash{\begin{tabular}[t]{l}$\beta$\end{tabular}}}}%
    \put(0,0){\includegraphics[width=\unitlength,page=2]{bridge-same-surgery.pdf}}%
    \put(0.2400998,0.08964136){\makebox(0,0)[lt]{\lineheight{1.25}\smash{\begin{tabular}[t]{l}$\gamma$\end{tabular}}}}%
    \put(0,0){\includegraphics[width=\unitlength,page=3]{bridge-same-surgery.pdf}}%
    \put(0.69157879,0.14618217){\makebox(0,0)[lt]{\lineheight{1.25}\smash{\begin{tabular}[t]{l}$\alpha^\dagger$\end{tabular}}}}%
    \put(0.60595254,0.06111923){\makebox(0,0)[lt]{\lineheight{1.25}\smash{\begin{tabular}[t]{l}$\overline\beta$\end{tabular}}}}%
    \put(0.77213528,0.05999261){\makebox(0,0)[lt]{\lineheight{1.25}\smash{\begin{tabular}[t]{l}$\overline\gamma$\end{tabular}}}}%
    \put(0,0){\includegraphics[width=\unitlength,page=4]{bridge-same-surgery.pdf}}%
  \end{picture}%
\endgroup%

\caption{On the left, bridges $\alpha$ and $\beta$ with $\beta\in B_s(L,\alpha)$. $\gamma$ is the result of sliding the end of $\alpha$ over $\beta$. On the right, after surgery on $\alpha$, $\beta$ and $\gamma$ map to the same bridge.}
\label{fig:bridge-same-surgery}
\end{figure}

Suppose $\gamma\in\overleftarrow{\Br}(L)$ and $\eta\in B_\pitchfork(L_\gamma,\gamma^\dagger)$. Then $\eta\in\overleftarrow{\Br}(L_\gamma)$ as well. For any path
$$(L,\sigma)\xrightarrow{e_{(\gamma,\sigma,\sigma')}}(L_\gamma,\sigma')\xrightarrow{e_{(\eta,\sigma'\sigma'')}}(L_{\gamma,\eta},\sigma'')$$
we impose the relation
\begin{equation}\label{relation-bridge-intersect}
    \overleftarrow{e}_{(\gamma,\sigma,\sigma')}\overleftarrow{e}_{(\eta,\sigma',\sigma'')}=0.
\end{equation}

The result of surgering along $\gamma$ followed by $\gamma^\dagger$ gives back $L$, but with a different decoration. The following proposition and its proof are the same as in \cite{roberts-type-d}.

\begin{proposition}\cite[Proposition 21]{roberts-type-d}
If there is a path $\rho$
$$(L,\sigma)\xrightarrow{e_{(\gamma,\sigma,\sigma')}}(L_\gamma,\sigma')\xrightarrow{e_{(\gamma^\dagger,\sigma'\sigma'')}}(L,\sigma''),$$
then there is a circle $C\in\CIR(L)$ in the support of $\gamma$ where $\sigma(C)=+$ and $\sigma''=\sigma_C$.
\end{proposition}

For each path $\rho$ as above, the circle $C$ is called the active circle for $\rho$. When $\gamma\in\overrightarrow{\Br}(L)$, we impose the relation
\begin{equation}\label{relation-thereandback}
    \overrightarrow{e}_{(\gamma,\sigma,\sigma')}\overrightarrow{e}_{(\gamma^\dagger,\sigma',\sigma_C)}=\overrightarrow{e}_C.
\end{equation}

Lastly, suppose $C\in\CIR(L)$ is as in Figure~\ref{fig:1-1-bifurcation-relation} such that $\sigma(C)=+$, $\gamma$ and $\eta$ are bridges both of which have both feet on $C$, $\eta$ is a 1-1 bifurcation of $C$ to $C''$ and $\overline\gamma\in\BRIDGE(L_\eta)$ is a 1-1 bifurcation of $C''$ to $C'$, surgery on $\gamma$ splits $C$ into two circles $C_1$ and $C_2$, and $\overline\eta\in\BRIDGE(L_\gamma)$ merges $C_1$ and $C_2$ to get $C'$. We call this the special 1-1 bifurcation configuration. There are two paths
$$(L,\sigma)\xrightarrow{e_{(\gamma,\sigma,\sigma_i)}}(L_\gamma,\sigma_i)\xrightarrow{e_{(\overline\eta,\sigma_i\sigma')}}(L_{\gamma,\eta},\sigma')$$ for $i=1,2$ with $\sigma_i(C_i)=+$, $\sigma_i(C_{i+1})=-$ (again with indices counted mod $2$), and $\sigma'$ satisfies $\sigma'(D)=\sigma(D)$ for all $D$ not in the support of $\overline\eta$ and $\sigma'(C')=-$. We impose the relation
\begin{equation}\label{relation-1-1-bifurcation}
    e_{(\gamma,\sigma,\sigma_1)}e_{(\overline\eta,\sigma_1,\sigma')}+e_{(\gamma,\sigma,\sigma_2)}e_{(\overline\eta,\sigma_2,\sigma')}=0.
\end{equation}

\begin{figure}[h]
\centering
\def\svgwidth{.925\linewidth}
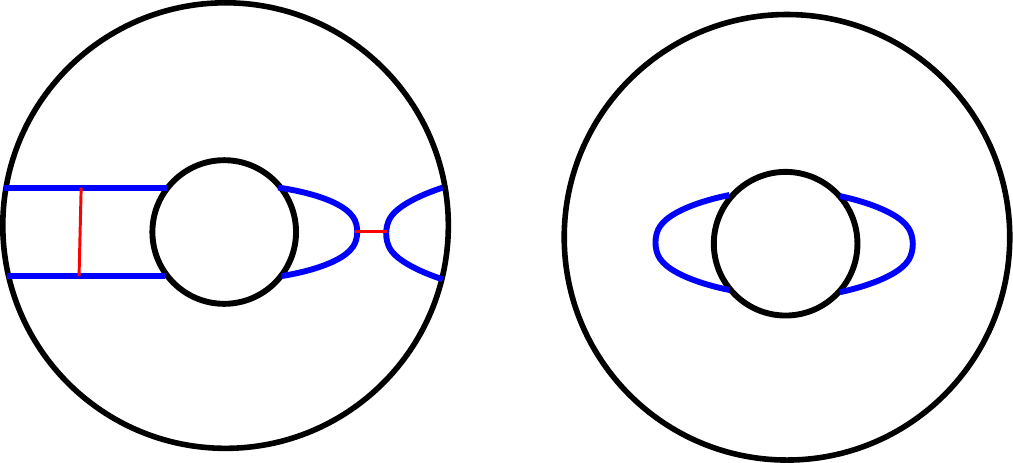
\caption{The two bridges in the 1-1 bifurcation relation}
\label{fig:1-1-bifurcation-relation}
\end{figure}

\begin{proposition}\cite[Proposition 23]{roberts-type-d}\label{prop:relation-homogeneous}
The relations defined in equations \ref{relation-eCeC}, \ref{relation-eCbridge}, \ref{relation-bridgebridge}, ~\ref{relation-righteCmerge}, ~\ref{relation-righteCdivide}, ~\ref{relation-lefteCmerge}, ~\ref{relation-lefteCdivide}, ~\ref{relation-bridgesquare}, ~\ref{relation-bridge-same}, ~\ref{relation-bridge-intersect}, ~\ref{relation-thereandback}, and ~\ref{relation-1-1-bifurcation} are homogeneous for the bigrading on $\QQ\Gamma_n$.
\end{proposition}
\begin{proof}
The proof is the same as in \cite{roberts-type-d}.
\end{proof}

\begin{definition}
The cleaver algebra of $n$-cleaved links $\BB\Gamma_n$ is the bigraded algebra obtained from $\QQ\Gamma_n$ by quotienting by all the above relations.
\end{definition}

\subsection{A differential on \texorpdfstring{$\BB\Gamma_n$}{BGamma_n}}

Surgery along a bridge $\gamma$ followed by $\gamma^\dagger$ gives a relation when $\gamma$ is in $\overrightarrow{\U}$. When they are in $\overleftarrow{\U}$, they correspond to a differential:

\begin{proposition}\cite[Proposition 26]{roberts-type-d}
For each $(L,\sigma)$ and circle $C\in\CIR(L)$ with $\sigma(C)=+$, let $d_{(L,\sigma),(L,\sigma_C)}$ be the $\Z/2\Z$-linear map on $\Hom((L,\sigma),(L,\sigma_C))$ defined by homomorphically extending to all paths the specification
\begin{align}
d_{(L,\sigma),(L,\sigma_C)}(\overleftarrow{e}_C) = \sum_{\gamma\in\overleftarrow{\Br}(L)} \overleftarrow{e}_{(\gamma,\sigma,\sigma')}\overleftarrow{e}_{(\gamma^\dagger,\sigma',\sigma_C)},
\end{align}
where the sum is over all length two paths
$$(L,\sigma)\xrightarrow{e_{(\gamma,\sigma,\sigma')}}(L_\gamma,\sigma')\xrightarrow{e_{(\gamma^\dagger,\sigma'\sigma_C)}}(L,\sigma_C)$$ for all $\gamma\in\overleftarrow{\Br}(L)$. For every other edge $e$ in $\Gamma_n$ and the identity morphisms $I_{(L,\sigma)}$, $d(e)=0$. Then $d$ can be extended using the Leibniz rule
$$d(\alpha\beta)=d(\alpha)\beta+\alpha d(\beta)$$
to a $(1,0)$ differential on the bigraded category $\BB\Gamma_n$.
\end{proposition}

We call the resulting differential on $\BB\Gamma_n$, $d_{\BB\Gamma_n}$. The Leibniz rule extends $d$ to $\QQ\Gamma_n$ automatically, but for it to extend to $\BB\Gamma_n$, we must check that $d$ is compatible with the relations.

\begin{proof}
The proof that $d$ extends to a bigrading $(0,1)$ differential on $\QQ\Gamma_n$ is the same as in \cite{roberts-type-d}.

It remains to verify the compatibility of $d$ with the relations defining $\BB\Gamma_n$. The proof is modeled on the proof of \cite[Proposition 26]{roberts-type-d}. We only need to consider relations with paths involving at least one $\overleftarrow{e}_C$ edge since the differential is trivial on any path that does not include a left decoration edge. Note that the relation in Equation~\ref{relation-1-1-bifurcation} is one of those cases that does not include a decoration edge.

\textbf{Relations for disjoint support:} Suppose $(L,\sigma)$ has a circle $C$ with $\sigma(C)=+$ and $\gamma\in\BRIDGE(L)$ where $C$ is not in the support of $\gamma$. Like in \cite{roberts-type-d}, we need to verify that for any bridge $\eta\in\overleftarrow{\Br}(L)$ with $\eta$ not isotopic to $\gamma$ that includes $C$ in its support,
\begin{equation}\label{eq:differential-disjoint}
e_\gamma\overleftarrow{e}_\eta\overleftarrow{e}_{\eta^\dagger}=\overleftarrow{e}_\eta\overleftarrow{e}_{\eta^\dagger}e_\gamma.
\end{equation} Note that if $\eta$ and $\gamma$ are in the same equivalence class, then $\eta$ is not a bridge in $L_\gamma$, but trivially the term in $e_\gamma d(\overleftarrow{e}_C)$ coming from $\eta=\gamma^\dagger$ cancels with the term in $d(\overleftarrow{e}_C)e_\gamma$ coming from $\eta=\gamma$.

Since $C$ is not in the support of $\gamma$, we cannot have the special 1-1 bifurcation configuration, so by the bridge relations,
$$e_\gamma e_{\overline\eta}=e_\eta e_{\overline\gamma},\hspace{10mm}e_{\overline\gamma}e_{\eta^\dagger}=e_{\eta^\dagger}e_\gamma,$$ from which Equation~\ref{eq:differential-disjoint} follows.

Now suppose $C'$ is another circle in $(L,\sigma)$ with $\sigma(C')=+$. We need to verify that $d(\overleftarrow{e}_C\overrightarrow{e}_{C'})=d(\overrightarrow{e}_{C'}\overleftarrow{e}_C)$ and $d(\overleftarrow{e}_C\overleftarrow{e}_{C'})=d(\overleftarrow{e}_{C'}\overleftarrow{e}_C)$.

For the first, we need to show for any $\eta\in\overleftarrow{\Br}(L)$ with $C$ as its active circle,
\begin{equation}\label{eq:diff-right-dec}
    e_\eta e_{\eta^\dagger}\overrightarrow{e}_{C'}=\overrightarrow{e}_{C'}e_\eta e_{\eta^\dagger}.
\end{equation} The verification of Equation~\ref{eq:diff-right-dec} is the same as in \cite{roberts-type-d}.

Now to show $d(\overleftarrow{e}_C\overleftarrow{e}_{C'})=d(\overleftarrow{e}_{C'}\overleftarrow{e}_C)$, we have
$$d(\overleftarrow{e}_C\overleftarrow{e}_{C'})=d(\overleftarrow{e}_C)\overleftarrow{e}_{C'}+\overleftarrow{e}_{C}d(\overleftarrow{e}_{C'}),$$
$$d(\overleftarrow{e}_{C'}\overleftarrow{e}_{C})=d(\overleftarrow{e}_{C'})\overleftarrow{e}_{C}+\overleftarrow{e}_{C'}d(\overleftarrow{e}_C).$$ As in \cite{roberts-type-d}, we split into cases: bridges $\gamma\in\overleftarrow{\Br}(L)$ with only one foot on $C$ or $C'$, improper pairs $\alpha,\overline\beta$ in $\overleftarrow{\U}$ with active circle $C$ or $C'$, and bridges $\gamma\in\overleftarrow{\Br}(L)$ with one foot on $C$ and the other foot on $C'$. 

For $\gamma\in\overleftarrow{\Br}(L)$ with one foot on $C$ and disjoint from $C'$, by the relations for disjoint support, 
$$e_\gamma e_{\gamma^\dagger}\overleftarrow{e}_{C'}=e_\gamma\overleftarrow{e}_{C'}e_{\gamma^\dagger}=\overleftarrow{e}_{C'}e_\gamma e_{\gamma^\dagger}.$$ The term on the left is a term in $d(\overleftarrow{e}_C)\overleftarrow{e}_{C'}$ the the term on the right is a term in $\overleftarrow{e}_{C'}d(\overleftarrow{e}_C)$. A similar argument works for improper pairs with $C$ as their active circle since they also have disjoint support from $\overleftarrow{e}_{C'}$. A similar argument holds for bridges with only $C'$ in their support, showing they cancel out. The case where $\gamma\in\overleftarrow{\Br}(L)$ has one foot on $C$ and the other on $C'$ is dealt with in the same way as in \cite{roberts-type-d}.

\textbf{Relations from bridges:} Suppose $\gamma$ is a bridge that merges cleaved circles $C_1$ and $C_2$ with $\sigma(C_1)=\sigma(C_2)=+$ to get $C$, then we have the relation as in Equation~\ref{relation-lefteCmerge}. If we apply $d$ to all the terms and simplify using the Leibniz rule, we need to verify the relation
\begin{equation}\label{eq:diff-double-left-ec}
d(\overleftarrow{e}_{C_1})e_\gamma+d(\overleftarrow{e}_{C_2})e_\gamma+e_\gamma d(\overleftarrow{e}_C)=0.
\end{equation}

As in \cite{roberts-type-d}, we divide the bridges $\eta$ in $\overleftarrow{\Br}(L)$ whose image $\overline\eta$ in $L_\gamma$ abuts $C$ into three groups: 1) those abutting $C_1$ and not $C_2$, 2) those abutting $C_2$ and not $C_1$, and 3) those between arcs in $C_1$ and $C_2$. Furthermore, we must consider the bridges $\gamma$ and $\gamma^\dagger$ themselves.

Let $\eta$ be a bridge in 1), then we do not have the special 1-1 bifurcation configuration. If $\gamma\in\overrightarrow{\Br}(L)$ or $\gamma\in B_d(L,\eta)\cup B_o(L,\eta)$ then we get commuting square relations and
$$e_\gamma e_{\overline\eta}e_{\overline\eta^\dagger}=e_\eta e_{\overline\gamma}e_{\overline\eta^\dagger}=e_\eta e_{\eta^\dagger}e_{\overline\gamma},$$ where the left term is a term in $e_\gamma d(\overleftarrow{e}_C)$ and the right term is a term in $d(\overleftarrow{e}_{C_1})e_\gamma$. 

If $\gamma\in\overleftarrow{\Br}(L)$ and $\gamma\in B_s(L,\eta)$, and $\delta$ is the bridge obtained by sliding an end of $\gamma$ over $\eta$, then $\delta$ is in 2) or 3). From Relation~\ref{relation-bridge-same} and the commuting square relations, we obtain
\begin{align*}
    e_\gamma e_{\overline\eta}e_{\overline\eta^\dagger} &= e_\eta e_{\overline\gamma}e_{\overline\eta^\dagger}+e_{\delta}e_{\overline\gamma}e_{\overline\delta^\dagger} \\
    &= e_\eta e_{\eta^\dagger}e_{\overline\gamma}+e_{\delta}e_{\delta^\dagger}e_{\overline\gamma},
\end{align*} where the left hand side is a term in $e_\gamma d(\overleftarrow{e}_C)$, the first term on the second row is a term in $d(\overleftarrow{e}_{C_1})e_\gamma$, and the third term is a term in $d(\overleftarrow{e}_{C_2})e_\gamma$ or $d(\overleftarrow{e}_{C_1})e_\gamma$ but not both. 

Similarly, bridges from 2) give terms that cancel in Equation~\ref{eq:diff-double-left-ec}.

Now we consider bridges $\eta$ in 3) that abut arcs from $C_1$ and $C_2$ but does not equal $\gamma$ or $\gamma^\dagger$. Again we do not have the special 1-1 bifurcation configuration. If $\gamma\in\overleftarrow{\Br}(L)$ and $\gamma\in B_s(L,\eta)$ then $\delta$, the bridge obtained by sliding an end of $\gamma$ over $\eta$, is in 1) or 2), so this case is already accounted for above. Otherwise, we have the commuting square relations and the argument follows as in \cite{roberts-type-d}.

Lastly, the contribution of $\gamma^\dagger$ to $e_\gamma d(\overleftarrow{e}_C)$ cancels with the contribution of $\gamma$ to $d(\overleftarrow{e}_{C_1})e_\gamma$ and $d(\overleftarrow{e}_{C_2})e_\gamma$ the same way as in \cite{roberts-type-d}.

Now suppose $\gamma$ divides a cleaved circle $C\in\CIR(L)$ with $\sigma(C)=+$ into two cleaved circles $C_1$ and $C_2$. We need to verify the relation
\begin{equation}\label{eq:diff-left-ec-split}
    d(\overleftarrow{e}_C)e_\gamma + e_{(\gamma,\sigma,\sigma_1)}d(\overleftarrow{e}_{C_1})+e_{(\gamma,\sigma,\sigma_2)}d(\overleftarrow{e}_{C_2})=0,
\end{equation} where $\sigma_i$ is the decoration on $L_\gamma$ satisfying $\sigma_i(C_i)=+$. The argument is similar to the case where $\gamma$ is a merge, but in this case, the special 1-1 bifurcation configuration may arise. If $\eta\in\Br(L)$ is a 1-1 bifurcation of $C$ to a circle $C'$ such that $\overline\gamma\in\Br(L_\eta)$ is a 1-1 bifurcation, then
\begin{align*}
    0 &= e_{(\gamma,\sigma,\sigma_1)}e_{\overline\eta}e_{\overline\eta^\dagger}+e_{(\gamma,\sigma,\sigma_2)}e_{\overline\eta}e_{\overline\eta^\dagger} \\
    &= e_{(\gamma,\sigma,\sigma_1)}e_{\overline\eta}e_{\overline\eta^\dagger}+e_{(\gamma,\sigma,\sigma_2)}e_{\overline\eta}e_{\overline\eta^\dagger}.
\end{align*} Each term on the second line appears exactly once in Equation~\ref{eq:diff-left-ec-split}.

\end{proof}

This completes the construction of $(\BB\Gamma_n,d_{\BB\Gamma_n})$ as a differential, bigraded $\Z/2$-algebra.

\section{The APS Chain Complex}

Given a tangle diagram $\overrightarrow{T}$, we define a bigraded module $\llbracket \overrightarrow{T}\rrangle$ and a differential $\overrightarrow{d}_{\APS}$ on $\llbracket \overrightarrow{T}\rrangle$. A similar definition gives a bigraded module $\llangle \overleftarrow{T}\rrbracket$ and a differential $\overleftarrow{d}_{\APS}$ for left tangles $\overleftarrow{T}$.

\begin{definition}
A state for $\overrightarrow{T}$ is a pair $(r,s)$, where
\begin{enumerate}
    \item $r$ is a resolution for $\overrightarrow{T}$,
    \item $s$ is an assignment of an element of $\{+,-\}$ to each circle of $r$. This assignment will be called a decoration on $r$.
\end{enumerate}
\end{definition}

The states for $\overrightarrow{T}$ will be denoted $\STATE(\overrightarrow{T})$.

If we ignore the free circles of a state $(r,s)$ and only look at the cleaved circles and the decoration restricted to those cleaved circles, we get a decorated $n$-cleaved link called the boundary of $(r,s)$.

\begin{definition}
The boundary of a state $(r,s)$ for $\overrightarrow{T}$ is the element $$\partial(r,s)=(\cl(r),\sigma)\in\CC\LL_n,$$ where $\sigma=s\vert_{\cl(r)}$.
\end{definition}

\begin{definition}
For a state $(r,s)\in\STATE(\overrightarrow{T})$ with $r=(\rho,\overleftarrow{m})$, let
\begin{enumerate}
    \item $h(r)=\sum_{c\in\CR(\overrightarrow{T})}\rho(c)$,
    \item $q(r,s)=\sum_{C\in\FREE(r)}s(C)$,
    \item $\iota(r,s)=\iota(\partial(r,s))$.
\end{enumerate}
\end{definition}

\begin{definition}
Given $\overrightarrow{T}$ a right tangle diagram and $(L,\sigma)\in\CC\LL_n$, let
$$\overrightarrow{CK}(\overrightarrow{T},L,\sigma)=\bigoplus_{(r,s)\in\STATE(\overrightarrow{T}),\ \partial(r,s)=(L,\sigma)}\Z/2\Z\cdot(r,s),$$
where $(r,s)$ occurs in bigrading $(h(r)-n_-,h(r)+q(r,s)+\frac{1}{2}\iota(r,s)+n_+-2n_-)$.

Define
$$\llbracket \overrightarrow{T}\rrangle=\bigoplus_{(L,\sigma)\in\CC\LL_n}\overrightarrow{CK}(\overrightarrow{T},L,\sigma).$$
\end{definition}

The first entry of the bigrading is called the homological grading while the second is the quantum grading. There is a left action of $\II_n\subset\BB\Gamma_n$ on $\llbracket\overrightarrow{T}\rrangle$:
$$I_{(L,\sigma)}\cdot(r,s)=\begin{cases}
    (r,s) & \partial(r,s)=(L,\sigma) \\
    0 & \text{otherwise.}
\end{cases}$$
Thus $I_{(L,\sigma)}$ acts non-trivially only on the summand $\overrightarrow{CK}(\overrightarrow{T},L,\sigma)$. 

The definition of the module $\llangle\overleftarrow{T}\rrbracket$ for left tangles $\overleftarrow{T}$ is similar to the definition of $\llbracket\overrightarrow{T}\rrangle$, with the only difference being $\II_n$ acts on the right instead of the left.

\subsection{Types of Bridges}

Let $(r,s)\in\STATE(T)$ and $\alpha=(\alpha,s,s')\in\Br(r)$, then surgery on $\alpha$ has one of three effects: either it does not change the cleaved link $(L,\sigma)=\partial(r,s)$, it changes only the decoration of a cleaved circle component in $L$, or it changes $L$. In the first, $\alpha\in\INT(r,s)$, in the second, $\alpha\in\DEC(r,s)$, and in the third, $\alpha\in R(r)$.

\begin{definition}
A bridge $\alpha=(\alpha,s,s')\in\Br(r)$ is in exactly one of the following three sets:
\begin{enumerate}
    \item The set $\INT(r,s)$ consists of all $\alpha$ such that
    \begin{enumerate}
        \item $\alpha$ has both feet on free circles,
        \item $\alpha$ merges a cleaved circle $C$ with a free circle $D$ such that $s(D)=+$, or
        \item $\alpha$ splits a cleaved circle $C$ into a cleaved circle $C'$ and a free circle $D$ such that $s(C)=s(C')$.
    \end{enumerate}
    \item The set $\DEC(r,s)$ consists of all $\alpha$ such that
    \begin{enumerate}
        \item $\alpha$ merges a cleaved circle $C$ with $s(C)=+$ with a free circle $D$ with $s(D)=-$, or
        \item $\alpha$ splits a cleaved circle $C$ with $s(C)=+$ into a cleaved circle $C'$ and a free circle $D$ with $s(C')=-$.
    \end{enumerate}
    \item The set $R(r)$ consists of all $\alpha$ such that
    \begin{enumerate}
        \item $\alpha$ merges two cleaved circles $C_1$ and $C_2$, or
        \item $\alpha$ splits a cleaved circle $C$ into two cleaved circles $C_1$ and $C_2$.
    \end{enumerate}
\end{enumerate}
\end{definition}

Furthermore, for $C$ a cleaved circle with $s(C)=+$, let $\DEC(r,s,C)$ be the subset of $\DEC(r,s)$ of all bridges $\alpha=(\alpha,s,s')$ that change the decoration of the circle $C$. Every bridge $\alpha\in\DEC(r,s)$ is in $\DEC(r,s,C)$ for some $C$.

We define $e(\alpha)\in\BB\Gamma_n$ to be
$$e(\alpha)=\begin{cases}
    I_{(L,\sigma)} & \alpha\in\INT(r,s) \\
    \overleftarrow{e}_C & \alpha\in\overleftarrow{\DEC}(r,s,C) \\
    \overrightarrow{e}_C & \alpha\in\overrightarrow{\DEC}(r,s,C) \\
    e_\alpha & \alpha\in R(r).
\end{cases}$$

\subsection{Existence of Squares}

The following theorem concerns the existence of commuting squares of bridges. The structural relations for the algebraic objects we will define, namely the differential on the APS chain complex satisfying $d^2=0$, the type D relation, and the type A relation, will almost immediately follow from this theorem.

\begin{theorem}\label{thm:existence-squares}
Let $(r,s)\in\STATE(T)$ for some tangle $T$ in either $\overleftarrow{\U}$ or $\overrightarrow{\U}$ with $\partial(r,s)=(L,\sigma)$. Let $\alpha,\beta\in\Br(r)$ such that at least one of $\alpha$ or $\beta$ is in $\overrightarrow{\Br}(r)$ or $\beta\in B_d(L,\alpha)$, and the path
$$(r,s)\xrightarrow{(\alpha,s,s_\alpha)}(r_\alpha,s_\alpha)\xrightarrow{(\overline\beta,s_\alpha,s_{\alpha,\beta})}(r_{\alpha,\beta},s_{\alpha,\beta})$$ exists for some choice of decorations $\alpha=(\alpha,s,s_\alpha)$ and $\overline\beta=(\overline\beta,s_\alpha,s_{\alpha,\beta})$ such that $e(\alpha)e(\overline\beta)\ne 0$. Then if $\beta$ is not a 1-1 bifurcation, there exists a path
$$(r,s)\xrightarrow{(\beta,s,s_\beta)}(r_\beta,s_\beta)\xrightarrow{(\overline\alpha,s_\alpha,s_{\beta,\alpha})}(r_{\beta,\alpha},s_{\beta,\alpha})$$ with $(r_{\alpha,\beta},s_{\alpha,\beta})=(r_{\beta,\alpha},s_{\beta,\alpha})=(r',s')$. Furthermore, 
\begin{enumerate}
    \item There are at most two choices of decorations for $\alpha$ and $\overline\beta$ to obtain a path from $(r,s)$ to $(r',s')$.
    \item The choice of decorations for $\alpha$ and $\overline\beta$ to obtain a path from $(r,s)$ to $(r',s')$ is unique if and only if a choice of decorations for $\beta$ and $\overline\alpha$ to obtain a path from $(r,s)$ to $(r',s')$ exists and is unique. In this case, there is a relation $$e(\alpha)e(\overline\beta)=e(\beta)e(\overline\alpha),$$ or $e(\alpha)e(\overline\beta)$ and $e(\beta)e(\overline\alpha)$ are two of the three terms in Relation~\ref{relation-lefteCmerge} or Relation~\ref{relation-lefteCdivide}.
    \item If there are two choices of decorations $\alpha_1,\overline\beta_1$ and $\alpha_2,\overline\beta_2$ for $\alpha$ and $\overline\beta$ to obtain a path from $(r,s)$ to $(r',s')$, then either there are two choices of decorations $\beta_1,\overline\alpha_1$ and $\beta_2,\overline\alpha_2$ for $\beta$ and $\overline\alpha$ to go from $(r,s)$ to $(r',s')$, or there are no choices of decorations for $\beta$ and $\overline\alpha$. In the first case, $e(\alpha_1)e(\overline\beta_1)$ appears in a relation
    $$e(\alpha_1)e(\overline\beta_1)=e(\alpha_2)e(\overline\beta_2)=e(\beta_1)e(\overline\alpha_1)=e(\beta_2)e(\overline\alpha_1).$$ In the second, we have the special 1-1 bifurcation configuration and
    $$e(\alpha_1)e(\overline\beta_1)=e(\alpha_2)e(\overline\beta_2).$$
\end{enumerate}
\end{theorem}

The proof is a long but routine checking of cases.

\begin{proof}

We fix $(r,s)$ and $(r',s')$ and consider the kinds of bridges $\alpha,\beta,\overline\alpha,\overline\beta$ that can go from $(r,s)$ to $(r',s')$: merges of two circles into one, fission of one circle into two, and 1-1 bifurcations. Along the two paths $\alpha,\overline\beta$ and $\beta,\overline\alpha$, the change in the number of circles must be the same. 

In the following, we enumerate the possible types of bridges we can have for $(\alpha,\beta,\overline\alpha,\overline\beta)$, up to switching the roles of $\alpha$ and $\beta$.

First, consider configurations without 1-1 bifurcations. If 1-1 bifurcations are not present, then the change in the number of circles must be an even number. For a configuration that decreases the number of circles by $2$, the four bridges involved must all be merges:
\begin{enumerate}
    \item (Merge, Merge, Merge, Merge).
\end{enumerate}

For a configuration that preserves the number of circles, if $\alpha$ is a merge, then $\overline\beta$ must be a fission, and vice versa, and the same with $\beta$ and $\overline\alpha$. There are three such configurations:
\begin{enumerate}
    \item[(2)] (Merge, Fission, Merge, Fission),
    \item[(3)] (Merge, Fission, Merge, Fission),
    \item[(4)] (Merge, Merge, Fission, Fission).
\end{enumerate}

For a configuration that increases the number of circles by $2$, the four bridges must all be fissions:
\begin{enumerate}
    \item[(5)] (Fission, Fission, Fission, Fission).
\end{enumerate}

Now consider configurations that include at least one 1-1 bifurcation. The change in the number of circles must be between $-1$ and $+1$, and the parity of the number of 1-1 bifurcations in $\alpha,\overline\beta$ is equal to the parity of 1-1 bifurcations in $\beta,\overline\alpha$. 

For a configuration that decreases the number of circles by $1$, there cannot be any fissions, and if $\alpha$ is a merge, then $\overline\beta$ must be a 1-1 bifurcation and vice versa, and the same with $\beta$ and $\overline\alpha$. There are two possibilities:
\begin{enumerate}
    \item[(6)] (Bifurcation, Merge, Bifurcation, Merge),
    \item[(7)] (Merge, Merge, Bifurcation, Bifurcation).
\end{enumerate}

(Bifurcation, Bifurcation, Merge, Merge) is impossible because if $\alpha$ is a 1-1 bifurcation, it has both feet on the same circle, and this is still true after surgery on $\beta$, so $\overline\alpha$ cannot be a merge.

For a configuration that increases the number of circles by $1$, there cannot be any merges, and if $\alpha$ is a fission, then $\overline\beta$ must be a 1-1 bifurcation and vice versa, and the same with $\beta$ and $\overline\alpha$. There are two possibilities:
\begin{enumerate}
    \item[(8)] (Bifurcation, Fission, Bifurcation, Fission),
    \item[(9)] (Bifurcation, Bifurcation, Fission, Fission). 
\end{enumerate}

(Fission, Fission, Bifurcation, Bifurcation) is ruled out by Lemma~\ref{lem:bif-case-iv}.

Lastly, for a configuration that preserves the number of circles, each path $\alpha,\overline\beta$ and $\beta,\overline\alpha$ must consist of two 1-1 bifurcations or a merge and a fission. There are two possibilities:
\begin{enumerate}
    \item[(10)] (Bifurcation, Bifurcation, Bifurcation, Bifurcation),
    \item[(11)] (Fission, Bifurcation, Bifurcation, Merge).
\end{enumerate}

(Bifurcation, Merge, Fission, Bifurcation) is impossible because $\beta$ has feet on two different circles, which is still true after surgery on $\alpha$.

The statement of the theorem is trivial if $\alpha$ and $\beta$ have disjoint support by the relations for disjoint support. Thus, we will only consider cases where $\alpha$ and $\beta$ have nondisjoint support. We break each of the cases above into subcases based on whether the circles are free or cleaved.

\textbf{(Merge, Merge, Merge, Merge)}

Suppose $C_1,C_2,C_3$ are circles in $r$ such that $\alpha$ merges $C_1$ and $C_2$ while $\beta$ merges $C_2$ and $C_3$. The two paths $\alpha,\overline\beta$ and $\beta,\overline\alpha$ exist if at least two of the circles have $+$ decorations, and neither path exists otherwise. In the following we consider the cases where the two paths exist.

\begin{enumerate}
    \item If $C_1,C_2$, and $C_3$ are cleaved, then $e(\alpha)e(\overline\beta)=e(\beta)e(\overline\alpha)$ from the bridge square relations.
    \item Suppose $C_1$ and $C_2$ are cleaved but $C_3$ is free. Since $\beta$ has one foot on a free circle, it must be an active arc for $r$. If $s(C_3)=+$, then at least one of $C_1$ or $C_2$ has $+$ decorations. We have $e(\alpha)=e_\alpha$, $e(\overline\beta)=I$, $e(\beta)=I$, and $e(\overline\alpha)=e_{\overline\alpha}$, and
        $$e_\alpha I=e_\alpha=e_{\overline\alpha}=Ie_{\overline\alpha}.$$

        If $s(C_3)=-$, then both $C_1$ and $C_2$ have $+$ decorations. We have $e(\alpha)=e_\alpha$, $e(\overline\beta)=e_{C_1\#C_2}$, $e(\beta)=e_{C_2}$, and $e(\overline\alpha)=e_{\overline\alpha}$. If $\beta\in\overrightarrow{\Br}(r)$, then we have
        $$e_\alpha \overrightarrow{e}_{C_1\# C_2}=\overrightarrow{e}_{C_2}e_{\overline\alpha}.$$ If $\beta\in\overleftarrow{\Br}(r)$, then
        $$e_\alpha\overleftarrow{e}_{C_1\# C_2}+\overleftarrow{e}_{C_1}e_{\overline\alpha} +\overleftarrow{e}_{C_2}e_{\overline\alpha}=0.$$

    \item Suppose $C_1$ is cleaved, $C_2$ is free, and $C_3$ is cleaved. If $s(C_2)=+$, then $e(\alpha)=I$, $e(\overline\beta)=e_{\overline\beta}$, $e(\beta)=I$, and $e(\overline\alpha)=e_{\overline\alpha}$, and
        $$Ie_{\overline\beta}=e_{\overline\beta}=e_{\overline\alpha}=Ie_{\overline\alpha}.$$

        If $s(C_2)=-$, then $e(\alpha)=e_{C_1}$, $e(\overline\beta)=e_{\overline\beta}$, $e(\beta)=e_{C_3}$, and $e(\overline\alpha)=e_{\overline\alpha}$. Note that $\alpha$ and $\beta$ are either both in $\overleftarrow{\U}$ or both in $\overrightarrow{\U}$. If they are both right bridges then we have
        $$\overrightarrow{e}_{C_1}e_{\overline\beta}=\overrightarrow{e}_{C_3}e_{\overline\alpha},$$ and if they are both left bridges then we have
        $$\overleftarrow{e}_{C_1}e_{\overline\beta}+\overleftarrow{e}_{C_3}e_{\overline\alpha}+ e_\gamma\overleftarrow{e}_{C_1\#C_3}=0,$$ where $\gamma$ is the left bridge with one end on $C_1$ and the other end on $C_3$ obtained by sliding an end of $\alpha$ over $\beta$. 
    \item The case where $C_1$ is free and $C_2$ and $C_3$ are cleaved is symmetric to case (2).
    \item Suppose $C_1$ is cleaved and $C_2$ and $C_3$ are free. If $s(C_2)=s(C_3)=+$ then we have
        $$e(\alpha)e(\overline\beta)=I\cdot I=e(\beta)e(\overline\alpha).$$ If $s(C_3)=-$ then we have
        $$e(\alpha)e(\overline\beta)=Ie_{C_1\# C_2}=Ie_{C_1}=e(\beta)e(\overline\alpha).$$ If $s(C_2)=-$ then we have
        $$e(\alpha)e(\overline\beta)=e_{C_1}I=Ie_{C_1}=e(\beta)e(\overline\alpha).$$
    \item Suppose $C_1$ is free, $C_2$ is cleaved, and $C_3$ is free. If $s(C_1)=s(C_3)=+$ then we have
        $$e(\alpha)e(\overline\beta)=I\cdot I=e(\beta)e(\overline\alpha).$$ If one of $C_1$ or $C_3$ has $-$ decoration, say $C_1$, then 
        $$e(\alpha)e(\overline\beta)=e_{C_2}I=Ie_{C_2\#C_3}=e(\beta)e(\overline\alpha).$$
    \item The case where $C_1$ and $C_2$ are free and $C_3$ is cleaved is symmetric to case (5).
    \item If $C_1$, $C_2$, and $C_3$ are free, then $e(\alpha)=e(\overline\beta)=e(\beta)=e(\overline\alpha)=I$.
\end{enumerate}

\textbf{(Merge, Fission, Merge, Fission)}

Suppose $\alpha$ merges two circles $C_1$ and $C_0$ while $\beta$ splits $C_0$ into two circles $C_2$ and $C_3$ with $\overline\alpha$ having one foot on $C_1$ and the other on $C_2$. Both paths exist if at least one of $C_1$ or $C_0$ has $+$ decoration, and neither path exists otherwise. In the cases where paths exist:

\begin{enumerate}
    \item Suppose $C_1,C_2$, and $C_3$ are cleaved. 
       If $s(C_1)=s(C_0)=+$, then there are two different possibilities for $(r',s')$: one with $s'(C_1\# C_2)=+$ and $s'(C_3)=-$, and the other with $s'(C_1\#C_2)=-$ and $s'(C_3)=+$. For each $(r',s')$, there are two paths from $(r,s)$ to $(r',s')$, and we get
        $$e(\alpha)e(\overline\beta)=e(\beta)e(\overline\alpha).$$

        If one of $C_1$ or $C_0$ has $-$ decoration and the other has $+$, then there is one choice for $(r',s')$ and the two paths from $(r,s)$ to $(r',s')$ commute to give
        $$e(\alpha)e(\overline\beta)=e(\beta)e(\overline\alpha).$$
    \item Suppose $C_1$ and $C_2$ are cleaved and $C_3$ is free. If $s(C_1)=s(C_0)=+$, then there are two possibilities for $(r',s')$ as in case (1). If $s'(C_1\#C_2)=+$ and $s'(C_3)=-$, then we have
        $$e(\alpha)e(\overline\beta)=e_\alpha I=Ie_{\overline\alpha}=e(\beta)e(\overline\alpha).$$
        
        If on the other hand $s'(C_1\# C_2)=-$ and $s'(C_3)=+$, then we have
        $e(\alpha)e(\overline\beta)=e_\alpha e_{C_1\# C_0}$ and $e(\beta)e(\overline\alpha)=e_{C_0}e_{\overline\alpha}$. If $\beta$ is in $\overrightarrow{\U}$ then 
        $$e_\alpha \overrightarrow{e}_{C_1\# C_0}=\overrightarrow{e}_{C_0}e_{\overline\alpha},$$ otherwise
        $$e_\alpha \overleftarrow{e}_{C_1\# C_0}+\overleftarrow{e}_{C_0}e_{\overline\alpha}+\overleftarrow{e}_{C_1}e_{\overline\alpha}=0.$$

        If one of $C_1$ or $C_0$ has $-$ decoration and the other has $+$, then we have
        $$e(\alpha)e(\overline\beta)=e_\alpha I=Ie_{\overline\alpha}=e(\beta)e(\overline\alpha).$$

    \item Suppose $C_1$ is cleaved, $C_2$ is free, and $C_3$ is cleaved. If $s(C_1)=s(C_0)=+$, then as before there are two possibilities for $(r',s')$. Note that $\overline\beta=\alpha^\dagger$ as bridges and $\alpha$ and $\beta$ are both in $\overrightarrow{\U}$ by assumption. If $s'(C_1\#C_2)=+$ and $s'(C_3)=-$, then we have
        $$e(\alpha)e(\overline\beta)=\overrightarrow{e}_\alpha \overrightarrow{e}_{\overline\beta}=\overrightarrow{e}_{C_3}I=e(\beta)e(\overline\alpha)$$ by Relation~\ref{relation-thereandback}. If $s'(C_1\# C_2)=-$ and $s'(C_3)=+$, then we have
        $$e(\alpha)e(\overline\beta)=\overrightarrow{e}_\alpha \overrightarrow{e}_{\overline\beta}=I\overrightarrow{e}_{C_1}=e(\beta)e(\overline\alpha)$$ again by Relation~\ref{relation-thereandback}.

        If $s(C_1)=+$ and $s(C_0)=-$ we have
        $$e(\alpha)e(\overline\beta)=\overrightarrow{e}_\alpha \overrightarrow{e}_{\overline\beta}=I\overrightarrow{e}_{C_1}=e(\beta)e(\overline\alpha).$$ If $s(C_1)=-$ and $s(C_0)=+$ we have
        $$e(\alpha)e(\overline\beta)=\overrightarrow{e}_\alpha \overrightarrow{e}_{\overline\beta}=\overrightarrow{e}_{C_3}I=e(\beta)e(\overline\alpha).$$

    \item Suppose $C_1$ is free, and $C_2$ and $C_3$ are cleaved. If $s(C_1)=s(C_0)=+$ then there are two possibilities for $(r',s')$. In either case,
        $$e(\alpha)e(\overline\beta)=Ie_{\overline\beta}=e_\beta I=e(\beta)e(\overline\alpha).$$

        If $s(C_1)=-$ and $s(C_0)=+$ then $e(\alpha)e(\overline\beta)=e_{C_0}e_{\overline\beta}$ and $e(\beta)e(\overline\alpha)=e_\beta e_{C_2}$. If $\alpha$ is in $\overrightarrow{\U}$ then 
        $$\overrightarrow{e}_{C_0}e_{\overline\beta}=e_\beta \overrightarrow{e}_{C_2},$$ otherwise
        $$\overleftarrow{e}_{C_0}e_{\overline\beta}+e_\beta \overleftarrow{e}_{C_2}+e_\beta\overleftarrow{e}_{C_3}=0.$$

        If $s(C_1)=+$ and $s(C_0)=-$ then 
        $$e(\alpha)e(\overline\beta)=Ie_{\overline\beta}=e_\beta I=e(\beta)e(\overline\alpha).$$

    \item Suppose $C_1$ is cleaved and $C_2$ and $C_3$ are free. If $s(C_1)=s(C_0)=+$ then there are two possibilities for $(r',s')$. If $s'(C_1\#C_2)=+$ and $s'(C_3)=-$ then
        $$e(\alpha)e(\overline\beta)=I\cdot I=e(\beta)e(\overline\alpha).$$ If $s'(C_1\#C_2)=-$ and $s'(C_3)=+$ then
        $$e(\alpha)e(\overline\beta)=Ie_{C_1\#C_0}=Ie_{C_1}=e(\beta)e(\overline\alpha).$$

        If $s(C_1)=-$ and $s(C_0)=+$ then $$e(\alpha)e(\overline\beta)=I\cdot I=e(\beta)e(\overline\alpha).$$

        If $s(C_1)=+$ and $s(C_0)=-$ then
        $$e(\alpha)e(\overline\beta)=e_{C_1}I=Ie_{C_1}=e(\beta)e(\overline\alpha).$$
    \item Suppose $C_1$ is free, $C_2$ is cleaved, and $C_3$ is free. If $s(C_1)=s(C_0)=+$ then there are two possibilities for $(r',s')$. If $s'(C_1\#C_2)=+$ and $s'(C_3)=-$ then
        $$e(\alpha)e(\overline\beta)=I\cdot I=e(\beta)e(\overline\alpha).$$ If $s'(C_1\#C_2)=-$ and $s'(C_3)=+$ then
        $$e(\alpha)e(\overline\beta)=Ie_{C_1\# C_0}=e_{C_0}I=e(\beta)e(\overline\alpha).$$

        If $s(C_1)=-$ and $s(C_0)=+$ then
        $$e(\alpha)e(\overline\beta)=e_{C_0}I=Ie_{C_2}=e(\beta)e(\overline\alpha).$$

        If $s(C_1)=+$ and $s(C_0)=-$ then
        $$e(\alpha)e(\overline\beta)=I\cdot I=e(\beta)e(\overline\alpha).$$

    \item Suppose $C_1$ and $C_2$ are free and $C_3$ is cleaved. If $s(C_1)=s(C_0)=+$ then there are two possibilities for $(r',s')$. If $s'(C_1\#C_2)=+$ and $s'(C_3)=-$ then
        $$e(\alpha)e(\overline\beta)=Ie_{C_1\# C_0}=e_{C_0}I=e(\beta)e(\overline\alpha).$$ If $s'(C_1\#C_2)=-$ and $s'(C_3)=+$ then
        $$e(\alpha)e(\overline\beta)=I\cdot I=e(\beta)e(\overline\alpha).$$

        If $s(C_1)=-$ and $s(C_0)=+$ then
        $$e(\alpha)e(\overline\beta)=e_{C_0}I=e_{C_0}I=e(\beta)e(\overline\alpha).$$

        If $s(C_1)=+$ and $s(C_0)=-$ then
        $$e(\alpha)e(\overline\beta)=I\cdot I=e(\beta)e(\overline\alpha).$$

    \item If $C_1$, $C_2$, and $C_3$ are all free, then
    $$e(\alpha)e(\overline\beta)=I\cdot I=e(\beta)e(\overline\alpha).$$
\end{enumerate}

\textbf{(Merge, Merge, Fission, Fission)}

Suppose $\alpha$ and $\beta$ both merge the circles $C_1$ and $C_0$ and $\overline\alpha$ and $\overline\beta$ are both fissions of $C_1\#C_0$. Since the roles of $\alpha$ and $\beta$ are symmetric, either both paths exist or both paths do not exist, and we have $e(\alpha)e(\overline\beta)=e(\beta)e(\overline\alpha)$ whenever the paths exist.

\textbf{(Fission, Fission, Merge, Merge)}

Suppose $\alpha$ splits $C_0$ into two circles $C_1$ and $C_2$ such that $\overline\beta$ merges the two. Meanwhile, $\beta$ splits $C_0$ into two circles $C_3$ and $C_4$ such that $\overline\alpha$ merges the two. Both paths exist if $s(C_0)=+$ and neither path exists otherwise.

\begin{enumerate}
    \item If $C_1$, $C_2$, $C_3$, and $C_4$ are cleaved then there are four paths from $(r,s)$ to $(r',s')$ that give
        $$e_{\alpha_1}e_{\overline\beta_1}=e_{\alpha_2}e_{\overline\beta_2}=e_{\beta_3}e_{\overline\alpha_3}=e_{\beta_4}e_{\overline\beta_4}.$$
    \item If three of the circles are cleaved, without loss of generality $C_1$, $C_2$, and $C_3$, while the other, $C_4$, is free then there are four paths from $(r,s)$ to $(r',s')$, and notice that $\overline\beta=\alpha^\dagger$ so $\alpha,\beta$ are in $\overrightarrow{\U}.$ We get
        $$\overrightarrow{e}_{\alpha_1}\overrightarrow{e}_{\overline\beta_1}=\overrightarrow{e}_{\alpha_2}\overrightarrow{e}_{\overline\beta_2}=\overrightarrow{e}_{C_0}I=I\overrightarrow{e}_{C_3}.$$
    \item If two of the circles are cleaved, without loss of generality $C_1$ and $C_3$, while the other two, $C_2$ and $C_4$, are free (note that it is not possible for, say, $C_1$ and $C_2$ to be cleaved while $C_3$ and $C_4$ are free) then there are four paths that give
        $$e_{C_0}I=Ie_{C_1}=e_{C_0}I=Ie_{C_3}.$$
    \item It is impossible for only one of the circles $C_1,C_2,C_3$, and $C_4$ to be cleaved while the other three are free.
    \item If $C_1,C_2,C_3,$ and $C_4$ are all free, then there are four paths that each give $I\cdot I$.
\end{enumerate}

\textbf{(Fission, Fission, Fission, Fission)}

Suppose $\alpha$ and $\beta$ are both fissions of a circle $C_0$ such that after surgery on $\alpha$ (resp. $\beta$), $\overline\beta$ (resp. $\overline\alpha$) is still a fission. Let $\alpha$ divide $C_0$ into $C_1$ and $C_0'$, $\overline\beta$ divides $C_0'$ into $C_2$ and $C_3$, $\beta$ divides $C_0$ into $C_0''$ and $C_3$, and $\overline\alpha$ divides $C_0''$ into $C_1$ and $C_2$. The two paths $\alpha,\overline\beta$ and $\beta,\overline\alpha$ both divide $C_0$ into three distinct circles $C_1$, $C_2$, and $C_3$. If $s(C_0)=+$, then there are two paths to each $(r',s')$ where $s'$ assigns two of the circles $C_1$, $C_2$, and $C_3$ to $-$ and the other to $+$. If $s(C_0)=-$, then $s'(C_1)=s'(C_2)=s'(C_3)=-$ and there are two paths.

\begin{enumerate}
    \item If $C_1,C_2$, and $C_3$ are all cleaved, then whenever the paths exist, we have
    $$e(\alpha)e(\overline\beta)=e_\alpha e_{\overline\beta}=e_\beta e_{\overline\alpha}=e(\beta)e(\overline\alpha).$$
    \item Suppose $C_1$ and $C_2$ are cleaved but $C_3$ is free. If $s(C_0)=+$, then there are three possibilities for $(r',s')$ as in case (a). If $s'(C_1)=+$ or $s'(C_2)=+$, then we have
        $$e(\alpha)e(\overline\beta)=e_\alpha I=Ie_{\overline\alpha}=e(\beta)e(\overline\alpha).$$ If $s'(C_3)=+$ then
        $e(\alpha)e(\overline\beta)=e_\alpha e_{C_0'}$, $e(\beta)e(\overline\alpha)=e_{C_0}e_{\overline\alpha}$, and either
        $$e_\alpha\overrightarrow{e}_{C_0'}=\overrightarrow{e}_{C_0}e_{\overline\alpha}$$ if $\beta$ is in $\overrightarrow{\U}$ or
        $$e_\alpha\overleftarrow{e}_{C_0'}+e_\alpha\overleftarrow{e}_{C_1}+\overleftarrow{e}_{C_0}e_{\overline\alpha}=0$$ if $\beta$ is in $\overleftarrow{\U}$.

        If $s(C_0)=-$ then we have
        $$e(\alpha)e(\overline\beta)=e_\alpha I=Ie_{\overline\alpha}=e(\beta)e(\overline\alpha).$$

    \item Suppose $C_1$ is cleaved, $C_2$ is free, and $C_3$ is cleaved, then $\alpha=\beta$ as bridges. If $s(C_0)=+$, then there are three possibilities for $(r',s')$ as in case (a). If $s'(C_1)=+$ or $s'(C_3)=+$ then
        $$e(\alpha)e(\overline\beta)=e_\alpha I=e_\beta I=e(\beta)e(\overline\alpha).$$ If $s'(C_2)=+$ then
        $e(\alpha)e(\overline\beta)=e_\alpha e_{C_0'}$ and $e(\beta)e(\overline\alpha)=e_\beta e_{C_0''}$. Note that $\alpha$ and $\beta$ are on the same side: they are either both in $\overrightarrow{\U}$ or both in $\overleftarrow{\U}$. In the first case,
        $$e_\alpha \overrightarrow{e}_{C_0'}=e_\beta \overrightarrow{e}_{C_0''},$$ and in the latter,
        $$e_\alpha \overleftarrow{e}_{C_0'}+e_\beta \overleftarrow{e}_{C_0''}+\overleftarrow{e}_{C_0}e_{\overline\alpha}=0.$$

        If $s(C_0)=-$, then
        $$e(\alpha)e(\overline\beta)=e_\alpha I=e_\beta I=e(\beta)e(\overline\alpha).$$
    \item The case where $C_1$ is free and $C_2$ and $C_3$ are cleaved is symmetric to case (2).
    \item Suppose $C_1$ is cleaved and $C_2$ and $C_3$ are free. If $s(C_0)=+$, then there are three possibilities for $(r',s')$ as in case (a). If $s'(C_1)=+$ then
        $$e(\alpha)e(\overline\beta)=I\cdot I=e(\beta)e(\overline\alpha).$$ If $s'(C_2)=+$ then
        $$e(\alpha)e(\overline\beta)=e_{C_0}I=Ie_{C_0''}=e(\beta)e(\overline\alpha).$$ If $s'(C_3)=+$ then
        $$e(\alpha)e(\overline\beta)=e_{C_0}I=e_{C_0}I=e(\beta)e(\overline\alpha).$$

        If $s(C_0)=-$ then
        $$e(\alpha)e(\overline\beta)=I\cdot I=e(\beta)e(\overline\alpha).$$

    \item Suppose $C_1$ is free, $C_2$ is cleaved, and $C_3$ is free. If $s(C_0)=+$, then there are three possibilities for $(r',s')$ as in case (a). If $s'(C_1)=+$ then
        $$e(\alpha)e(\overline\beta)=e_{C_0}I=Ie_{C_0''}=e(\beta)e(\overline\alpha).$$ If $s'(C_2)=+$ then
        $$e(\alpha)e(\overline\beta)=I\cdot I=e(\beta)e(\overline\alpha).$$ If $s'(C_3)=+$ then
        $$e(\alpha)e(\overline\beta)=Ie_{C_0'}=e_{C_0}I=e(\beta)e(\overline\alpha).$$

        If $s(C_0)=-$ then
        $$e(\alpha)e(\overline\beta)=I\cdot I=e(\beta)e(\overline\alpha).$$
    \item The case where $C_1$ and $C_2$ are cleaved and $C_3$ is free is symmetric to case (5).
    \item If $C_1$, $C_2$ and $C_3$ are all free then whenever the paths exist, we have
    $$e(\alpha)e(\overline\beta)=I\cdot I=e(\beta)e(\overline\alpha).$$
\end{enumerate}

The cases involving 1-1 bifurcations all have a 1-1 bifurcation in both paths $\alpha,\overline\beta$ and $\beta,\overline\alpha$, hence do not give nonzero paths, except for the case (Fission, Bifurcation, Bifurcation, Merge).

\textbf{(Fission, Bifurcation, Bifurcation, Merge)}

Suppose $\alpha$ divides $C_0$ into $C_1$ and $C_2$ such that $\overline\beta$ merges $C_1$ and $C_2$ to $C_3$, while $\beta$ is a 1-1 bifurcation of $C_0$ to $C_0'$ and $\overline\alpha$ is a 1-1 bifurcation of $C_0'$ to $C_3$. There are two nonzero paths from $(r,s)$ to $(r',s')$, both from $\alpha,\overline\beta$, if $s(C_0)=+$. Otherwise, there are no nonzero paths. 

\begin{enumerate}
    \item Suppose $C_1$ and $C_2$ are cleaved, then $e(\alpha_i)e(\beta_i)=e_{\alpha_i}e_{\overline\beta_i}$. We have
    $$e_{\alpha_1}e_{\overline\beta_1}+e_{\alpha_2}e_{\overline\beta_2}=0$$ from Relation~\ref{relation-1-1-bifurcation}.
    \item Suppose one of $C_1$ and $C_2$ is cleaved and the other is free, say $C_1$ is cleaved and $C_2$ is free. We have $e(\alpha_1)e(\overline\beta_1)=Ie_{C_1}$, $e(\alpha_2)e(\overline\beta_2)=e_{C_0}I$, and
        $$I\overrightarrow{e}_{C_1}=\overrightarrow{e}_{C_0}I.$$

    \item If $C_1$ and $C_2$ are both free, then both paths give $I\cdot I$.
\end{enumerate}

\end{proof}

\subsection{The APS Differential}

\begin{definition}
For $(r,s)\in\STATE(\overrightarrow{T})$, we define
$$\overrightarrow{d}_{\APS}(r,s)=\sum_{\gamma\in\overrightarrow{\ACT}(r)}D_{\gamma}(r,s),$$ where $D_\gamma$ is a $(1,0)$-bigraded linear map defined by:
\begin{enumerate}
    \item If surgery on $\gamma$ merges the free circles $C_1$ and $C_2$ to get a free circle $C$ with $s(C_1)=+$ or $s(C_2)=+$, then $D_\gamma(r,s)=(r_\gamma,s_\gamma)$, where $s_\gamma(C)$ is equal to
    \begin{itemize}
        \item $+$ if $s(C_1)=s(C_2)=+$,
        \item $-$ if $s(C_1)=-s(C_2)$.
    \end{itemize}
    \item If surgery on $\gamma$ splits the free circle $C$ into $C_1$ and $C_2$, and $s(C)=+$, then $D_\gamma(r,s)$ is equal to
    \begin{itemize}
        \item $(r_\gamma,s_1)+(r_\gamma,s_2)$, where $s_i(C_i)=+$, $s_i(C_{i+1})=-$, if $s(C)=+$, 
        \item $(r_\gamma,s_\gamma)$, where $s_\gamma(C_1)=s_\gamma(C_2)=-$, if $s(C)=-$.
    \end{itemize}
    \item Suppose $\gamma$ has both feet on the same arc of a cleaved circle $C$ in $\overrightarrow{\U}$ such that surgery on $\gamma$ splits $C$ into a free circle component $D$ and $C'$ a cleaved circle. Then $D_\gamma(r,s)=(r_\gamma,s_\gamma)$, where $s_\gamma(D)=-$ and $s_\gamma(C')=s(C)$.
    \item If $\gamma$ has one foot on a cleaved circle $C$ and the other foot on a free circle $D$, then surgery on $\gamma$ will merge $D$ into $C$ to get a cleaved circle $C'$, leaving the other circles unchanged. If $s(D)=+$, then $D_\gamma(r,s)=(r_\gamma,s_\gamma)$ with $s_\gamma(C')=s(C)$ , while if $s(D)=-$, then $D_\gamma(r,s)=0$.
    \item In all other cases, $D_\gamma(r,s)=0$.
\end{enumerate}
\end{definition}

In other words, 
$$\overrightarrow{d}_{\APS}(r,s)=\sum_{\gamma=(\gamma,s,s_\gamma):e(\gamma)=I}(r_\gamma,s_\gamma).$$

Note that since we are working over $\Z/2\Z$-coefficients, we do not need to specify an ordering of the crossings.

\begin{proposition}
The map $\overrightarrow{d}_{\APS}$ is $(1,0)$-bigraded.
\end{proposition}
\begin{proof}
Since $\overrightarrow{d}_{\APS}$ is always a surgery on an active arc, which changes a crossing of the resolution from the $0$-smoothing to the $1$-smoothing, the homological grading changes by $+1$. 

For the quantum grading, observe that in all cases for $D_\gamma$, the states in $D_\gamma(r,s)$ either have one fewer plus-decorated free circle or one additional minus-decorated free circle compared to $(r,s)$, and there is no other change in circles. Hence the quantum grading is preserved.
\end{proof}

\begin{theorem}\label{thm:aps-differential}
$\overrightarrow{d}_{\APS}^2=0$.
\end{theorem}

The proof that $\overrightarrow{d}_{\APS}$ is a differential on $\llbracket \overrightarrow{T}\rrangle$ is contained in the proof of Theorem~\ref{thm:existence-squares}: any pair of bridges $\alpha,\overline\beta$ with $e(\alpha)e(\overline\beta)=I$ either cancels with the pair $\beta,\overline\alpha$ or there are two choices of decorations for $\alpha,\overline\beta$ with $e(\alpha_1)e(\overline\beta_1)=e(\alpha_2)e(\overline\beta_2)=I$. 

Note that $\overrightarrow{d}_{\APS}$ on a state does not change the set of cleaved circles or their decorations, so $\overrightarrow{d}_{\APS}$ splits along the direct sum definition to give a differential on $\overrightarrow{CK}(\overrightarrow{T},L,\sigma)$ for each $(L,\sigma)\in\CC\LL_n$.

There is a similar definition of a $(1,0)$-differential $\overleftarrow{d}_{\APS}$ on $\llangle\overleftarrow{T}\rrbracket$ for a left tangle $\overleftarrow{T}$.

\section{Type D Structure}

Given a diagram $T$ for a right tangle $\TT$, we define a left differential map $$\overrightarrow{\delta}_T:\llbracket T\rrangle\to\BB\Gamma_n(M,p)\otimes_\II\llbracket T\rrangle[(-1,0)]$$ by
\begin{align*}
    \overrightarrow{\delta}_T(r,s) &= \sum_{(\alpha,s,s_\alpha)\in\BRIDGE(r,s)} e(\alpha)\otimes (r_\alpha,s_\alpha)+\sum_{C\in\CIR(\partial(r,s)) s(C)=+}\overleftarrow{e}_C\otimes(r,s_C).
\end{align*}

\begin{proposition}\cite[Proposition 44]{roberts-type-d} The map $\overrightarrow{\delta}_T$ is grading preserving.
\end{proposition}
\begin{proof}
The proof of $\overrightarrow{\delta}_T$ being grading preserving is similar to the proof in \cite{roberts-type-d}.

For terms of $\overrightarrow{\delta}_T(r,s)$ with $e(\alpha)=I$, we know that the APS differential has bigrading $(1,0)$, and the bigrading of $I$ is $(0,0)$.

If $e(\alpha)=e_\alpha$, then in all cases, either $(r_\alpha,s_\alpha)$ has one fewer $+$-decorated cleaved circle compared to $(r,s)$, or it has one additional $-$-decorated cleaved circle compared to $(r,s)$. For right bridges, there is a change in the resolution of a crossing from $0$ to $1$, while for left bridges, there is no change in resolution. The bigrading of $\overrightarrow{e}_\alpha$ is $(0,-1/2)$ and the bigrading of $\overleftarrow{e}_\alpha$ is $(1,1/2)$.

If $e(\alpha)=\overrightarrow{e}_C$, then one $+$-decorated cleaved circle in $(r,s)$ is changed to have $-$ decoration, while $(r_\alpha,s_\alpha)$ has either one fewer $-$-decorated free circle or one additional $+$-decorated free circle. There is also a change in resolution of a crossing from $0$ to $1$. The bigrading of $\overrightarrow{e}_C$ is $(0,-1)$.

Lastly, $(r,s_C)$ for $C$ a cleaved circle with $s(C)=+$ changes a $+$-decorated cleaved circle to have $-$ decoration, there is no change in resolution, and the bigrading of $\overleftarrow{e}_C$ is $(1,1)$. Therefore, $\overrightarrow{\delta}_T$ is grading preserving.
\end{proof}

\begin{theorem}\label{thm:type-D-relation}
The map $\overrightarrow{\delta}_T$ satisfies the relation
\begin{align}\label{eq:type-D-rel}
(\mu_{\BB\Gamma_n}\otimes\I)(\I\otimes\overrightarrow{\delta}_T)\overrightarrow{\delta}_T+(d_{\BB\Gamma_n}\otimes\I)\overrightarrow{\delta}_T=0.
\end{align} Therefore, $(\llbracket T\rrangle,\overrightarrow{\delta}_T)$ is a type D structure.
\end{theorem}

\begin{proof}[Proof of Theorem~\ref{thm:type-D-relation}]
Fix some generator $(r,s)$ of $\llbracket T\rrangle$ and consider the left hand side of Equation~\ref{eq:type-D-rel} applied to $(r,s)$. Let $\Xi=(\mu_{\BB\Gamma_n}\otimes\I)(\I\otimes\overrightarrow{\delta}_T)\overrightarrow{\delta}_T(r,s)$; $\Xi$ is a sum over all length-$2$ paths 
$$(r,s)\xrightarrow{A}(\overline r,\overline s)\xrightarrow{B}(r',s'),$$ where $A$ and $B$ are either changing the decoration on a cleaved circle from $+$ to $-$ or doing surgery along an arc. We show through casework that the coefficient of each target state $(r',s')$ in $\Xi$ is $0$ except when $(r',s')=(r,s_C)$ for some active circle $C$, and then those terms cancel with terms coming from $(d_{\BB\Gamma_n}\otimes\I)\overrightarrow{\delta}_T(r,s)$ of Equation~\ref{eq:type-D-rel}.

The cases that involve a pair of bridges $\alpha,\beta\in\BRIDGE(r)$ with at least one of $\alpha$ or $\beta$ in $\overrightarrow{\BRIDGE}(r)$ or $\alpha\in B_d(L,\beta)$ all cancel to zero by Theorem~\ref{thm:existence-squares}. 

The other cases are as follows, numbered as in the proof of \cite[Theorem 45]{roberts-type-d}:

\textbf{Cases i-v:} These are handled in the same way as in \cite{roberts-type-d}. When the arc $\gamma$ is a 1-1 bifurcation, there are no nonzero paths.

\textbf{Cases xii-xiv:} These are handled in the same way as in \cite{roberts-type-d}.

The remaining cases are when $\alpha\in\overleftarrow{\Br}(r)$ and $\beta=\alpha^\dagger\in\overleftarrow{\Br}(r_\alpha)$. Summing over the contribution of all such pairs with active circle $C\in\CIR(r)$ gives the term corresponding to $(r,s_C)$ in  $(d_{\BB\Gamma_n}\otimes\I)\overrightarrow{\delta}_T(r,s)$.

\end{proof}

\subsection{Invariance of the Type D Structure}

\begin{theorem}\label{thm:type-d-invariance}
Let $\overrightarrow{\TT}$ be a right tangle with diagram $\overrightarrow{T}$. The homotopy class of the type D structure $(\llbracket T\rrangle, \overrightarrow{\delta}_T)$ is an invariant of the tangle $\overrightarrow{\TT}$.
\end{theorem}

The theorem follows once we show that the homotopy class of $(\llbracket T\rrangle, \overrightarrow{\delta}_T)$ is invariant under Reidemeister moves I, II, and III, and finger moves, mirror moves, and handleslides through $\partial D_i$ for $i>p$.

\begin{proof}[Proof of invariance under finger and mirror moves]
The states and differentials before and after the finger move or mirror move are identified by pushing the circles and arcs along the isotopy given by the move. In the following two paragraphs, we make these identifications precise.

Suppose $a$ is an arc in $T$ and suppose we do a finger move where we push $a$ across $\partial D_i$, resulting in a tangle diagram $T'$ as in Figure~\ref{fig:finger-mirror}. There is an identification of states $\xi$ for $T$ and states $\xi'$ for $T'$: the circle in $\xi$ with the arc $a$ is identified with circle in $\xi'$ with the arcs $b,c,$ and $d$, all the other circles are unchanged, and the decorations also match. Clearly, this identification respects the bigrading. Furthermore, the active arcs for the states $\xi$ and $\xi'$ are identified, the bridges for the left matchings are unchanged, and the set of cleaved circles is unchanged, meaning the type D differentials $\overrightarrow{\delta}_T$ and $\overrightarrow{\delta}_{T'}$ are also identified. Therefore $(\llbracket T\rrangle,\overrightarrow{\delta}_T)$ is isomorphic to $(\llbracket T'\rrangle, \overrightarrow{\delta}_{T'})$ as bigraded type D structures over $\BB\Gamma_n$. 

\begin{figure}[h]
\centering
\def\svgwidth{.925\linewidth}
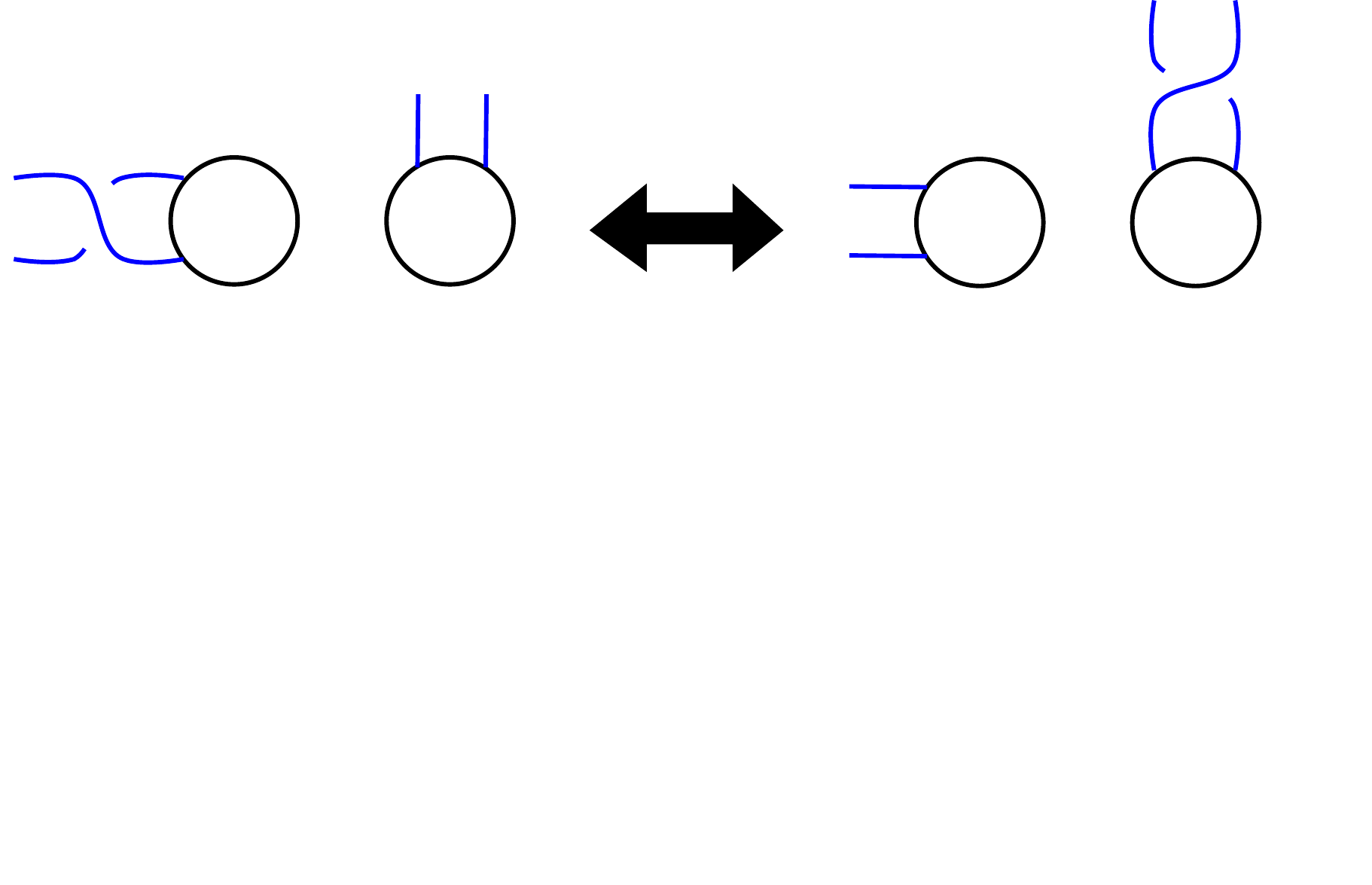
\caption{Top row: mirror move of crossing $x$. Middle row: the $0$-resolutions on both sides. Bottom row: the $1$-resolutions on both sides.}
\label{fig:mirror-states}
\end{figure}

Now suppose $x$ is a crossing in $T$ as in the top row of Figure~\ref{fig:mirror-states}, and suppose we do a mirror move where we push $x$ across $\partial D$ resulting in a diagram $T'$ with the crossing $x'$ that was pushed across the connect sum sphere. The states for $T$ and $T'$ are identified: fix some resolution $r=(\rho,\overleftarrow{m})$ of $T$, let $i\in\{0,1\}$ be the resolution of the crossing $x$ in $\rho$, and choose some decoration $s$ for $r$. Call this state $\xi_i\in \llbracket T \rrangle$. We identify $\xi_i$ with the state $\xi_i'$ which has the resolution $i$ for crossing $x'$, the same resolutions as $\xi$ for all the other crossings, and the same matching $\overleftarrow{m}$. In the $0$ resolutions, pictured in the middle row of Figure~\ref{fig:mirror-states} the circle containing arcs $a$ and $c$ in $\xi_0$ is identified with the circle containing arcs $a'$ and $c'$ in $\xi_0'$, the circle containing $b$ and $d$ is identified with the circle containing $b'$ and $d'$, and all the other circles are the same. In the $1$ resolutions, pictured in the bottom row of Figure~\ref{fig:mirror-states}, the circle containing arc $a$ in $\xi_1$ is identified with the circle containing arcs $a',b',$ and $c'$ in $\xi_1'$, the circle containing $b,c,$ and $d$ is identified with the circle containing $d'$, and all the other circles are the same. Clearly, this identification respects the bigrading. The active arcs for the states $\xi_i$ and $\xi_i'$ are identified: the arcs for the crossings that are not $x$ and $x'$ are the same, and the $0$-resolution arc for $\xi_0$ is identified with the $0$-resolution arc for $\xi_0'$. The left matchings are the same and the two states have the same set of cleaved circles and the circles have the same isotopy classes. Therefore the type D differentials $\overrightarrow{\delta}_T$ and $\overrightarrow{\delta}_{T'}$ are identified, and as a result, $(\llbracket T\rrangle,\overrightarrow{\delta}_T)$ is isomorphic to $(\llbracket T'\rrangle, \overrightarrow{\delta}_{T'})$ as bigraded type D structures over $\BB\Gamma_n$. The situation where the over and under strands of the crossing are changed is similar. 
\end{proof}

The proof of invariance under Reidemeister I, II, and III moves is largely the same as the proof of \cite[Lemma 48]{roberts-type-d}. The proofs rely on applying \cite[Proposition 40]{roberts-type-d} to cancel arrows in the type D differential that induce isomorphisms between the relevant states. The arrows come from the APS differential doing surgery on an active arc. The calculations are all the same as in \cite{roberts-type-d}; we only need to make sure that the surgeries in the arrows that we cancel induce isomorphisms.

\begin{proof}[Proof of Reidemeister I Invariance]
Suppose a local picture as in Figure~\ref{fig:r1-states} of a right-handed crossing $c$ that can be reduced by a Reidemeister I move appears in $T$. The states with $\rho(c)=0$ have an additional free circle $C$ with an active arc between $C$ and $D$. As in \cite{roberts-type-d}, we cancel the arrows in the APS differential that come from surgery on the $0$-resolution arc for $c$ where $C$ has a $+$ sign. By looking at the figure, it is clear that surgery on the $0$-resolution arc must be a merge of a free circle $C$ into $D$.

The argument follows as in \cite{roberts-type-d} that canceling the arrow gives zero perturbation term. 

\begin{figure}[h]
\centering
\includegraphics[width=0.6\textwidth]{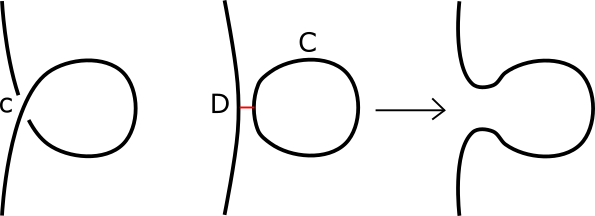}
\caption{Left: a right-handed crossing $c$ that can be reduced by a Reidemeister I move. Middle: the $0$-resolution on $c$. Right: the $1$-resolution on $c$.}
\label{fig:r1-states}
\end{figure}

As in \cite{roberts-type-d}, we obtain invariance for left-handed Reidemeister I moves once we show invariance for right-handed Reidemeister I moves and Reidemeister II moves.
\end{proof}

\begin{proof}[Proof of Reidemeister II Invariance]
Suppose a local picture as in Figure~\ref{fig:r2-states} of two crossings $c_0$ and $c_1$ that can be eliminated by a Reidemeister II move appears in $T$. As in \cite{roberts-type-d}, we cancel two types of arrows in the APS differential: 
\begin{enumerate}
    \item the first comes from surgery on the $0$-resolution arc for $c_0$ in the state where $\rho(c_0)=\rho(c_1)=0$. This splits off a free circle $C$, and we cancel the arrow that has as its target the state with $C$ having a $-$ decoration. 
    \item the second comes from surgery on the $0$-resolution arc for $c_1$ in the state where $\rho(c_0)=1$, $\rho(c_1)=0$, and the free circle $C$ has a $+$ decoration. Surgery on this arc merges the free circle $C$ with the bottom arc.
\end{enumerate}

\begin{figure}[h]
\centering
\def\svgwidth{.8\linewidth}
%% Creator: Inkscape 1.1 (c68e22c387, 2021-05-23), www.inkscape.org
%% Accompanies image file 'r2-states.pdf' (pdf, eps, ps)
\begingroup%
  \makeatletter%
  \providecommand\color[2][]{%
    \errmessage{(Inkscape) Color is used for the text in Inkscape, but the package 'color.sty' is not loaded}%
    \renewcommand\color[2][]{}%
  }%
  \providecommand\transparent[1]{%
    \errmessage{(Inkscape) Transparency is used (non-zero) for the text in Inkscape, but the package 'transparent.sty' is not loaded}%
    \renewcommand\transparent[1]{}%
  }%
  \providecommand\rotatebox[2]{#2}%
  \newcommand*\fsize{\dimexpr\f@size pt\relax}%
  \newcommand*\lineheight[1]{\fontsize{\fsize}{#1\fsize}\selectfont}%
  \ifx\svgwidth\undefined%
    \setlength{\unitlength}{511.27122774bp}%
    \ifx\svgscale\undefined%
      \relax%
    \else%
      \setlength{\unitlength}{\unitlength * \real{\svgscale}}%
    \fi%
  \else%
    \setlength{\unitlength}{\svgwidth}%
  \fi%
  \global\let\svgwidth\undefined%
  \global\let\svgscale\undefined%
  \makeatother%
  \begin{picture}(1,0.91452808)%
    \lineheight{1}%
    \setlength\tabcolsep{0pt}%
    \put(0,0){\includegraphics[width=\unitlength,page=1]{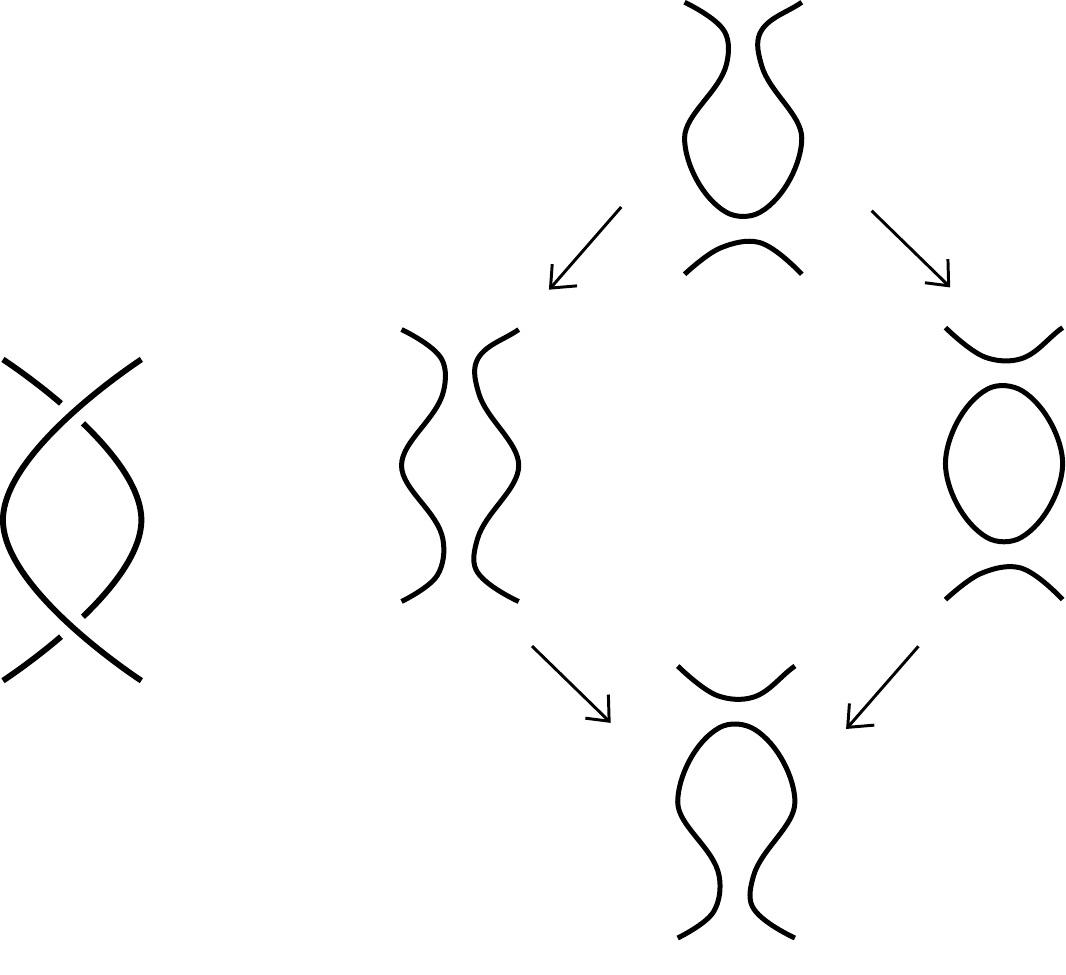}}%
    \put(0.11005133,0.51279854){\makebox(0,0)[lt]{\lineheight{1.25}\smash{\begin{tabular}[t]{l}$c_0$\end{tabular}}}}%
    \put(0.10155366,0.32207312){\makebox(0,0)[lt]{\lineheight{1.25}\smash{\begin{tabular}[t]{l}$c_1$\end{tabular}}}}%
    \put(0.65767904,0.6147706){\makebox(0,0)[lt]{\lineheight{1.25}\smash{\begin{tabular}[t]{l}$00$\end{tabular}}}}%
    \put(0.39708374,0.3060219){\makebox(0,0)[lt]{\lineheight{1.25}\smash{\begin{tabular}[t]{l}$01$\end{tabular}}}}%
    \put(0.66239994,0.00482674){\makebox(0,0)[lt]{\lineheight{1.25}\smash{\begin{tabular}[t]{l}$11$\end{tabular}}}}%
    \put(0.90788824,0.3107428){\makebox(0,0)[lt]{\lineheight{1.25}\smash{\begin{tabular}[t]{l}$10$\end{tabular}}}}%
    \put(0,0){\includegraphics[width=\unitlength,page=2]{r2-states.pdf}}%
  \end{picture}%
\endgroup%

\caption{Local picture of two crossings $c_0$ and $c_1$ that can be eliminated by a Reidemeister II move and the corresponding resolutions.}
\label{fig:r2-states}
\end{figure}

The rest of the proof follows as in \cite{roberts-type-d}.
\end{proof}

\begin{proof}[Proof of Reidemeister III Invariance]
Suppose a local picture as in the left hand side of Figure~\ref{fig:r3} of three crossings $c_1,c_2$, and $e$ appears in $T$, and we do a Reidemeister III move to modify the local picture to have the three crossings $c_1,c_2$, and $d$ to get a tangle diagram $T'$ as in the right hand side of the figure. Order the crossings so that $c_2<e<c_1$ on the left and $c_2<d<c_1$ on the right. As in \cite{roberts-type-d}, we simplify the "top" part of the type D structure corresponding to resolutions with $\rho(e)=0$ using the same process as simplifying the Reidemeister II type D structure. As in the Reidemeister II case, the arrows that we cancel involve splitting off free circles or merging free circles. What remains are the states and arrows in the "bottom" rectangle are unchanged, the states for the $001$ resolution, and a perturbation term in the new differential that goes from the $001$ resolution states to the bottom rectangle. We perform a similar procedure on the right hand side and note that the states and the bottom rectangle arrows are isotopic. Thus, all we need to check is that the perturbation terms match.

By the argument in \cite{roberts-type-d}, after performing the simplification on the left hand side, we obtain a perturbation term $001\to110$ which equals the arrow that comes from performing surgery on the $0$-resolution bridge for the crossing $c_2$. This is equal to the arrow on the right hand side that comes from surgery on the $0$-resolution bridge for the crossing $d$. Similarly, simplifying the top rectangle of the right hand side gives a perturbation term $100\to011$ which equals surgery on the $0$-resolution bridge for $c_1$, and is equal to the arrow on the left hand side that comes from surgery on the $0$-resolution bridge for the crossing $e$. 
\end{proof}

\begin{proof}[Proof of Handleslide Invariance]

The proof that the homotopy type of $\llbracket T\rrangle$ is invariant under handleslide moves is much the same as the proof in the case of the Khovanov chain complex of a link. We outline the necessary modifications to the proof below.

Given a tangle diagram $T$ in $\Sigma$ with $2n$ endpoints on $\partial D_p$, we construct a tangle diagram with nodes $T_*$ in $[0,\infty)\times\R\subset\R^2$ that represents $T$, such that in addition to the other properties of a tangle diagram with nodes representing $T$, it is transverse to $\{0\}\times\R$ and has its boundary equal to $P_n$, a set of $2n$ points on $\{0\}\times\R$. 

A resolution $r$ of $T$ is a resolution $\rho:\CR(T)\to\{0,1\}$ of the crossings of $T$ together with a crossingless matching $m$ of $P_n$ in $(-\infty,0]\times\R$, which is a crossingless matching that is allowed to have nodes. The resolution diagram $r(T)$ is the crossingless link obtained by gluing together $T$ and $m$ and replacing neighborhoods of crossings by smoothings according to $\rho$. As with the link case, the nodes are not resolved, but when determining the components of the resolution diagram, the two strands that intersect are treated as disjoint from each other.

We will also fix a circle $\CC=\{(x,y):(x-2)^2+y^2=1\}$. The circle divides $[0,\infty)\times\R$ into two regions: let $R_1=\{(x,y):(x-2)^2+y^2\le 1\}$ and $R_2=([0,\infty)\times\R)\setminus\mathring R_1$. 

We construct $T_*$ so that \begin{itemize}
    \item the tangle diagram with nodes $T_*\cap R_1$ represents $$\TT\cap((M_{i+1}\setminus\mathring D^3)\#\dots\#M_r),$$ 
    \item the tangle diagram with nodes $T_*\cap R_2$ represents $$\TT\cap((M_{p+1}\setminus\mathring D^3)\#M_{p+2}\#\dots\#(M_i\setminus\mathring D^3)).$$
\end{itemize} In particular, doing a handleslide of some arc of $T$ over $\partial D_i^-$ corresponds to doing a handleslide of the corresponding arc of $T$ in $R_2$ over $\CC$.

The rest of the proof is the same.

\end{proof}

\section{Type A Structure}

Given a diagram $T$ for a left tangle $\TT$, we describe two bigrading preserving maps $$m_1:\llangle T\rrbracket\to\llangle T\rrbracket[(-1,0)]$$ and $$m_2:\llangle T\rrbracket\otimes_\II\BB\Gamma_n(M,p)\to\llangle T\rrbracket.$$ In this section, we show that these two maps determine a type $A$ structure on $\llangle T\rrbracket$ over $\BB\Gamma_n(M,p)=\BB\Gamma_n$ with higher multiplications $m_i=0$ for $i\ge 3$. 

Let $(r,s)$ be a generator of $\llangle T\rrbracket$ with $\partial(r,s)=(L,\sigma)$.

The map $m_1$ is defined to be the same as $d_{\APS}$, that is $m_1(r,s)=d_{\APS}(r,s)$. We know that $d_{\APS}$ is bigrading preserving into $\llangle T\rrbracket[(-1,0)]$.

We define the action $m_2$ by describing the action of the generators of $\BB\Gamma_n$, and then use the relation
$$m_2((r,s)\otimes p_1p_2)=m_2(m_2((r,s)\otimes p_1)\otimes p_2)$$ for any $p_1,p_2\in\BB\Gamma_n$ to extend the action to all elements of $\BB\Gamma_n$.

Let $e$ be a generator of $\BB\Gamma_n$, then $m_2((r,s)\otimes e)$ is defined as follows:

\begin{enumerate}
    \item If $e$ is an idempotent, the action $m_2((r,s)\otimes I_{(L',\sigma')})=(r,s)\cdot I_{(L',\sigma')}$ is the idempotent action defined above, namely it is $(r,s)$ if $(L,\sigma)=(L',\sigma')$ and $0$ otherwise.
    \item If $e=\overrightarrow{e}_C$ for some $C\in L$ with $\sigma(C)=+$, then $m_2((r,s)\otimes_\II\overrightarrow{e}_C)=(r,s_C)$.
    \item If $e=\overleftarrow{e}_C$ for some $C\in L$ with $\sigma(C)=+$, then
    $$m_2((r,s)\otimes_\II \overleftarrow{e}_C)=\sum_{\gamma\in\overleftarrow{\DEC}((r,s),C)}(r_\gamma,s_\gamma),$$
    where $r_\gamma$ is the result of surgery on $\gamma$, and $s_\gamma$ is the new decoration with $s_\gamma(C)=-$.
    \item If $e=\overrightarrow{e}_{(\eta,\sigma,\sigma')}$ for some $\eta\in\overleftarrow{\Br}(L)$, let $s_\gamma$ be the decoration on $r_\gamma$ that equals $\sigma'$ in the cleaved circles and $s$ on the free circles, then
    $$m_2((r,s)\otimes_\II \overrightarrow{e}_{(\eta,\sigma,\sigma')})=(r_\gamma,s_\gamma).$$
    \item If $e=\overleftarrow{e}_{(\eta,\sigma,\sigma')}$ for some $\eta\in\overleftarrow{\Br}(L)$, then
    $$m_2((r,s)\otimes_\II \overleftarrow{e}_{(\eta,\sigma,\sigma')})=\sum_{\gamma\in\overleftarrow{\ACT}(r,\eta)} (r_\gamma,s_\gamma),$$ where $\overleftarrow{\ACT}(r,\eta)$ is the subset of arcs $\gamma$ in $\overleftarrow{\ACT}(r)$ with $\cl(\gamma)=\eta$ and $s_\gamma$ is the decoration on $r_\gamma$ that equals $\sigma'$ on the cleaved circles and $s$ on the free circles.
    \item In all other cases $m_2((r,s)\otimes_\II e)=0$. 
\end{enumerate}

Note that $m_2((r,s)\otimes_\II e)=0$ unless $I_{(L,\sigma)}\cdot e\ne 0$, that is $(L,\sigma)$ is the source of $e$. 

When $e=\overleftarrow{e}_C$ or $e=e_{(\eta,\sigma,\sigma')}$, it is clear that
$$m_2((r,s)\otimes_\II e)=\sum_{(\gamma,s,s_\gamma):e(\gamma)=e}(r_\gamma,s_\gamma).$$ In other words, for $\alpha=(\alpha,s,s_\alpha)\in\BRIDGE(r)$, let
$$m_\alpha((r,s)\otimes_{\II}e)=\begin{cases}
    (r_\alpha,s_\alpha) & e(\alpha)=e \\
    0 & \text{otherwise.}
\end{cases}$$ Then
$$m_2((r,s)\otimes_\II e)=\sum_{(\alpha,s,s_\alpha)\in\BRIDGE(r)}m_\alpha((r,s)\otimes_{\II} e).$$

\begin{proposition}\cite[Proposition 4.1]{roberts-type-a}
The map $m_2$ is bigrading preserving.
\end{proposition}
\begin{proof}
The proof is dual to the proof that the type D differential is bigrading preserving, so we omit it.
\end{proof}

In order for $m_2$ to be well-defined on all elements, it must satisfy the following proposition.

\begin{proposition}
Let $\rho$ be a relation in the algebra $\BB\Gamma_n$, then for all generators $(r,s)\in\llangle T\rrbracket$, 
$$m_2((r,s)\otimes \rho)=0.$$
\end{proposition}

\begin{proof}
Relations $\rho$ that only involve terms that are words in $\overleftarrow{e}_C$, $\overleftarrow{e}_\eta$, or $\overrightarrow{e}_\eta$, with the exception of pairs of left bridges with nondisjoint supports, are handled by Theorem~\ref{thm:existence-squares}. 

Indeed, for such a relation $\rho$, and any $\alpha=(\alpha,s,s_\alpha)\in\BRIDGE(r)$ and $\overline\beta=(\overline\beta,s_\alpha,s_{\alpha,\beta})\in\BRIDGE(r_\alpha)$, define
$$m_{\alpha,\overline\beta}((r,s)\otimes\rho)=\begin{cases}
    (r_{\alpha,\beta},s_{\alpha,\beta}) & e(\alpha)e(\overline\beta)\text{ is a term in }\rho \\
    0 & \text{otherwise.}
\end{cases}$$ Then by Theorem~\ref{thm:existence-squares}, there exists a path $\beta,\overline\alpha$ such that $$m_{\alpha,\overline\beta}((r,s)\otimes\rho)=m_{\beta,\overline\alpha}((r,s)\otimes\rho)$$ or there are two choices of decorations for $\alpha,\overline\beta$ such that 
$$m_{\alpha_1,\overline\beta_1}((r,s)\otimes\rho)=m_{\alpha_2,\overline\beta_2}((r,s)\otimes\rho)$$
whenever $\alpha,\beta\in\BRIDGE(r)$ with one of $\alpha$ or $\beta$ in $\overrightarrow{\BRIDGE}(r)$ or $\alpha\in B_d(L,\beta)$. 

Thus, we have
$$m_2((r,s)\otimes\rho)=\sum_{(\alpha,s,s_\alpha)\in\BRIDGE(r)}\sum_{(\overline\beta,s_\alpha,s_{\alpha,\beta})\in\BRIDGE(r_\alpha)}m_{\alpha,\overline\beta}((r,s)\otimes\rho).$$ If $\alpha$ and $\beta$ are in $\BRIDGE(r)$, then these terms cancel with each other.

For the cases where $\overline\beta\in\BRIDGE(r_\alpha)$ is not the image of a bridge $\beta\in\BRIDGE(r)$, either $\overline\beta=\alpha^\dagger$ or $\overline\beta$ intersects $\alpha^\dagger$. The case where $\overline\beta=\alpha^\dagger$ involves $\overrightarrow{e}_C$ terms and is dealt with in the same way as in the proof of \cite[Proposition 4.2]{roberts-type-a}. 

When $\overline\beta$ intersects $\alpha^\dagger$, there is no relation $\rho$ involving terms with $e(\alpha)e(\overline\beta)=e_\alpha e_{\overline\beta}$ when $\alpha,\overline\beta$ are in $\overrightarrow{\U}$. When $\alpha,\overline\beta$ are in $\overleftarrow{\U}$, we have $\overleftarrow{e}_\alpha \overleftarrow{e}_{\overline\beta}=0$, and indeed $m_2((r,s)\otimes\overleftarrow{e}_\alpha \overleftarrow{e}_{\overline\beta})=0$ since there are no pairs of active arcs for the state $(r,s)$ that realize $\alpha,\overline\beta$. 

Therefore, for $\rho$ as above, we have $m_2((r,s)\otimes\rho)=0$.

The remaining relations either involve $\overrightarrow{e}_C$ terms or pairs of bridges $\alpha,\beta\in\overleftarrow{\BRIDGE}(r)$ with $\beta\in B_s(L,\alpha)$ or $\beta\in B_o(L,\alpha)$. They are handled the same way as in the proof of \cite[Proposition 4.2]{roberts-type-a}. 

\end{proof}

The next proposition shows that $\llangle T\rrbracket$ with the maps $m_1$ and $m_2$ as defined above and $m_n=0$ for $n\ge 3$ defines an $A_\infty$-module over the differential graded algebra $\BB\Gamma_n$.

\begin{proposition}
For $(r,s)$ a generator of $\llangle T\rrbracket$ and $\rho_1,\rho_2\in\BB\Gamma_n$, the maps $m_1$ and $m_2$ defined above satisfy the following properties:
\begin{align}
    0 &= m_1(m_1((r,s))), \\
    0 &= m_2(m_1((r,s))\otimes\rho_1)+m_2((r,s)\otimes d_{\BB\Gamma_n}(\rho_1))+m_1(m_2((r,s)\otimes\rho_1)), \\
    0 &= m_2(m_2((r,s)\otimes\rho_1)\otimes\rho_2)+m_2((r,s)\otimes\rho_1\rho_2).
\end{align}
\end{proposition}

\begin{proof}
Just like in \cite[Proposition 4.3]{roberts-type-a}, the fact that $m_1\circ m_1=0$ follows from $d_{\APS}$ being a differential. The fact that $m_2$ defines a right action follows from the way we extended it from the generators of $\BB\Gamma_n$ to all elements of $\BB\Gamma_n$. Thus, we only need to verify the second relation.

By the same argument as in the proof of Proposition 4.3 in \cite{roberts-type-a}, we only need to prove this relation for words $\rho_1$ of length $0$ or $1$. Furthermore, the argument in that the relation holds for length $0$ elements is the same as in \cite{roberts-type-a}.

The words of length 1 are generators of type $\overleftarrow{e}_C$, $\overrightarrow{e}_C$, $\overleftarrow{e}_\gamma$, or $\overrightarrow{e}_\gamma$. When $\rho_1$ is one of the latter three, since $d_{\BB\Gamma_n}(\rho_1)=0$, the assertion we need to check is 
$$d_{\APS}(m_2((r,s)\otimes \rho_1))=m_2(d_{\APS}(r,s)\otimes\rho_1).$$ The left hand side is a sum over all length two paths $\alpha,\overline\beta$ of $(r_{\alpha,\beta},s_{\alpha,\beta})$ such that $e(\alpha)=\rho_1$ and $e(\overline\beta)=I$ while the right hand side is a sum over those paths with $e(\alpha)=I$ and $e(\overline\beta)=\rho_1$. By the proof of Theorem~\ref{thm:existence-squares}, such a path $\alpha,\overline\beta$ cancels with $\beta,\overline\alpha$ such that either $e(\beta)=\rho_1$ and $e(\overline\alpha)=I$, so $\beta,\overline\alpha$ contributes to $d_{\APS}(m_2((r,s)\otimes \rho_1))$, or $e(\beta)=I$ and $e(\overline\alpha)=\rho_1$, so $\beta,\overline\alpha$ contributes to $m_2(d_{\APS}(r,s)\otimes\rho_1)$, or there is a pair of choices of decorations for $\alpha,\overline\beta$ that similarly cancel with each other.

Now let $\rho_1=\overleftarrow{e}_C$ for some cleaved circle $C$ with $s(C)=+$. We need to check that
$$m_2(d_{\APS}(r,s)\otimes\overleftarrow{e}_C)+m_2((r,s)\otimes d_{\BB\Gamma_n}(\overleftarrow{e}_C))+d_{\APS}(m_2((r,s)\otimes\overleftarrow{e}_C))=0.$$
First, $m_2(d_{\APS}(r,s)\otimes\overleftarrow{e}_C)$ is a sum over all $\alpha,\overline\beta$ such that $e(\alpha)=I$ and $e(\overline\beta)=\overleftarrow{e}_C$ while $d_{\APS}(m_2((r,s)\otimes\overleftarrow{e}_C))$ is a sum over all $\alpha,\overline\beta$ such that $e(\alpha)=\overleftarrow{e}_C$ and $e(\overline\beta)=I$. The middle term $m_2((r,s)\otimes d_{\BB\Gamma_n}(\overleftarrow{e}_C))$ is a sum over all there-and-back paths $\alpha,\alpha^\dagger$ with active circle $C$ and all improper pairs $\alpha,\overline\beta$ with active circle $C$. Again, the proof of Theorem~\ref{thm:existence-squares} shows that these all cancel.

\end{proof}

\subsection{Invariance of the Type A Structure}

\begin{theorem}
Let $\overleftarrow{\TT}$ be a left tangle with diagram $\overrightarrow{T}$. The homotopy class of the type A structure $(\llangle T\rrbracket, \{m_i\}_{i=1}^\infty)$ for the maps $m_i$ defined above is an invariant of the tangle $\overleftarrow{\TT}$.
\end{theorem}

The invariance of the type A structure is proven in much the same way as the invariance of the type D structure, so we will omit it.

\section{Pairing and Isomorphism}

Let $\TT_1$ be a tangle in $M_1\#\dots\#M_p\setminus \mathring D^3$ and $\TT_2$ be a tangle in $M_{p+1}\#\dots\#M_r\setminus\mathring D^3$ both with $2n$ endpoints on $\partial D^3$ and form their boundary connected sum $\LL=\TT_1\natural \TT_2$, which is a link in $M=M_1\#\dots\#M_r$.

Let $T_1$ be a tangle diagram for $\TT_1$ and $T_2$ be a tangle diagram for $\TT_2$. In the above sections, we showed how to define a type A structure $\llangle T_1\rrbracket$ for $T_1$ and the type D structure $\llbracket T_2\rrangle$ for $T_2$. We define the box tensor product $\llangle T_1\rrbracket\boxtimes\llbracket T_2\rrangle$ of $\llangle T_1\rrbracket$ and $\llbracket T_2\rrangle$ to be the bigraded module
$$\llangle T_1\rrbracket\otimes_{\II_n}\llbracket T_2\rrangle$$
equipped with the map
$$\partial^\boxtimes(x\otimes y)=m_{1,T_1}(x)\otimes y + (m_{2,T_1}\otimes\I)(x\otimes\overrightarrow{\delta}_{T_2}(y)).$$

It follows from standard results about $A_\infty$ modules that $\partial^\boxtimes$ is a $(1,0)$-differential map on $\llangle T_1\rrbracket\boxtimes\llbracket T_2\rrangle$.

\begin{theorem}\label{thm:pairing-invarinace}
The chain homotopy type of $\llangle T_1\rrbracket\boxtimes\llbracket T_2\rrangle$ is invariant under isotopy in $\Sigma$, Reidemeister moves, and finger, mirror, and handleslide moves across $\partial D_i$ for $i\ne p$.
\end{theorem}
\begin{proof}
We have proven that Reidemeister moves in $\overleftarrow{\U}$ or $\overrightarrow{\U}$ do not change the homotopy equivalence classes of $\llangle T_1\rrbracket$ or $\llbracket T_2\rrangle$, respectively, thus they do not change the chain homotopy type of their box tensor product. Additionally, finger, mirror, and handleslide moves across $\partial D_i$ for $i<p$ preserve the homotopy type of $\llangle T_1\rrbracket$, and finger, mirror, and handleslide moves across $\partial D_i$ for $i>p$ preserve the homotopy type of $\llbracket T_2\rrangle$, so the homotopy type of $\llangle T_1\rrbracket\boxtimes\llbracket T_2\rrangle$ is also invariant under these moves.
\end{proof}

\begin{theorem}\label{thm:pairing-direct-iso}
Let $\LL$ be a link in $M$ with link diagram $L$ such that $L=T_1\natural T_2$ for tangle diagrams $T_1,T_2$ of tangles $\TT_1$ in $M_1\#\dots\#M_p\setminus \mathring D^3$ and $\TT_2$ in $M_{p+1}\#\dots\#M_r\setminus\mathring D^3$. Then $\CKh(T_1\natural T_2)$ is isomorphic to $\llangle T_1\rrbracket\boxtimes\llbracket T_2\rrangle$ as bigraded vector spaces. Furthermore, under this isomorphism, the Khovanov differential $d$ on $\CKh(T_1\natural T_2)$ is identified with $\partial^\boxtimes$ on $\llangle T_1\rrbracket\boxtimes\llbracket T_2\rrangle$.
\end{theorem}

\begin{proof}
The proof is the same as the proof of \cite[Proposition 7.3]{roberts-type-a}.
\end{proof}

As a consequence, it does not matter which connect sum sphere to divide $M$ along.

\begin{corollary}
Let $0\le p,p'\le r$, $q=r-p$, $q'=r-p'$, and let $\LL$ be a link in $M$ with link diagram $L$ such that $L=T_1^p\natural T_2^q=T_1^{p'}\natural T_2^{q'}$ for tangle diagrams $T_1^p,T_1^{p'},T_2^q,T_2^{q'}$ in $M_1\#\dots\#M_p\setminus \mathring D^3$, $M_1\#\dots\#M_{p'}\setminus \mathring D^3$, $M_{p+1}\#\dots\#M_r\setminus \mathring D^3$, and $M_{p'+1}\#\dots\#M_r\setminus \mathring D^3$, respectively. Then $(\llangle T_1^p\rrbracket\boxtimes\llbracket T_2^q\rrangle,\partial^\boxtimes)$ is isomorphic to $(\llangle T_1^{p'}\rrbracket\boxtimes\llbracket T_2^{q'}\rrangle,\partial^\boxtimes)$ as bigraded chain complexes.
\end{corollary}
\begin{proof}
By Theorem~\ref{thm:pairing-direct-iso}, $(\llangle T_1^p\rrbracket\boxtimes\llbracket T_2^q\rrangle,\partial^\boxtimes)$ and $(\llangle T_1^{p'}\rrbracket\boxtimes\llbracket T_2^{q'}\rrangle,\partial^\boxtimes)$ are both isomorphic to $(\CKh(L),d)$.
\end{proof}

Finally, we use the tangle constructions to provide another proof of the invariance of the Khovanov chain complex of a link.

\begin{proof}[Proof of Theorem~\ref{thm:ckh-invariance} using type A/D structures]
Choose one of the connect sum spheres $D_p$ to cut $L$ into $L=T_1^p\natural T_2^q$ as above. Invariance under isotopy in $\Sigma$ and invariance under Reidemeister moves follows from Theorem~\ref{thm:pairing-invarinace} and Theorem~\ref{thm:pairing-direct-iso}. The invariance under finger, mirror, and handleslide moves through $\partial D_i$ for $i\ne p$ of $(\CKh(L),d)$ is implied by the invariance of $(\llangle T_1^p\rrbracket,\{m_j\}_{j=1}^\infty)$ for $p>i$ or by the invariance of $(\llbracket T_2^q\rrangle,\overrightarrow{\delta}_{T_2^q})$ for $q<i$ under the same moves and Theorem~\ref{thm:pairing-direct-iso}.

Finally, the invariance under finger, mirror, and handleslide moves through $\partial D_p$ is implied by choosing a different separating sphere $D_{p'}$ to cut along.
\end{proof}

It is also possible to prove that the homotopy type of $\llangle L\rrangle$ is a link invariant without using the isomorphism with $(\CKh(L),d)$, but this proof is much lengthier.

\bibliographystyle{alpha} %\bibliographystyle{alphahack}
\bibliography{biblio}

@article{roberts-type-d,
title = {{A type $D$ structure in Khovanov homology}},
journal = {Advances in Mathematics},
volume = {293},
pages = {81-145},
year = {2016},
issn = {0001-8708},
doi = {https://doi.org/10.1016/j.aim.2016.02.007},
url = {https://www.sciencedirect.com/science/article/pii/S0001870816000621},
author = {Lawrence P. Roberts}
}

@article{roberts-type-a,
author = {Lawrence P. Roberts},
title = {{A type $A$ structure in Khovanov homology}},
volume = {16},
journal = {Algebraic \& Geometric Topology},
number = {6},
publisher = {Mathematical Sciences Publishers},
pages = {3653 -- 3719},
year = {2016},
doi = {10.2140/agt.2016.16.3653},
URL = {https://doi.org/10.2140/agt.2016.16.3653}
}

@article{aps-i-bundles,
   title={{Categorification of the Kauffman bracket skein module of $I$–bundles over surfaces}},
   volume={4},
   ISSN={1472-2747},
   url={http://dx.doi.org/10.2140/agt.2004.4.1177},
   DOI={10.2140/agt.2004.4.1177},
   number={2},
   journal={Algebraic \& Geometric Topology},
   publisher={Mathematical Sciences Publishers},
   author={Asaeda, Marta M and Przytycki, Jozef H and Sikora, Adam S},
   year={2004},
   month=dec, pages={1177–1210} 
}

@article{khovanov-categorification,
author = {Mikhail Khovanov},
title = {{A categorification of the Jones polynomial}},
volume = {101},
journal = {Duke Mathematical Journal},
number = {3},
publisher = {Duke University Press},
pages = {359 -- 426},
year = {2000},
doi = {10.1215/S0012-7094-00-10131-7},
URL = {https://doi.org/10.1215/S0012-7094-00-10131-7}
}

@article{gabrovsek-rp3, title={{The categorification of the Kauffman bracket skein module of $ \mathbb{R} {\mathrm{P} }^{3} $}}, volume={88}, DOI={10.1017/S0004972713000105}, number={3}, journal={Bulletin of the Australian Mathematical Society}, author={Gabrovšek, Boštjan}, year={2013}, pages={407–422}}

@article{willis-kh,
author = {Michael Willis},
title = {{Khovanov Homology for Links in ${\mathrm{\# }^{r}}({S^{2}}\times {S^{1}})$}},
volume = {70},
journal = {Michigan Mathematical Journal},
number = {4},
publisher = {University of Michigan, Department of Mathematics},
pages = {675 -- 748},
year = {2021},
doi = {10.1307/mmj/1594281620},
URL = {https://doi.org/10.1307/mmj/1594281620}
}

@article{rozansky-categorification,
      title={{A categorification of the stable $SU(2)$ Witten-Reshetikhin-Turaev invariant of links in $S^2 x S^1$}}, 
      author={Lev Rozansky},
      journal={arXiv preprint arXiv:1011.1958},
      year={2010},
      eprint={1011.1958},
      archivePrefix={arXiv},
      primaryClass={math.GT},
      url={https://arxiv.org/abs/1011.1958}, 
}

@article{chen-rp3,
      title={{Khovanov-type homologies of null homologous links in $\mathbb{RP}^3$}}, 
      author={Daren Chen},
      year={2021},
      journal={arXiv preprint arXiv:2104.04779},
      eprint={2104.04779},
      archivePrefix={arXiv},
      primaryClass={math.GT},
      url={https://arxiv.org/abs/2104.04779}, 
}

@article{rasmussen-s-inv,
    author = {Rasmussen, Jacob},
    title = {{Khovanov homology and the slice genus}},
    journal = {Inventiones mathematicae},
    year = {2010},
    volume = {182},
    number = {2},
    pages = {419-447},
    doi = {10.1007/s00222-010-0275-6}
}

@article{bar-natan,
author = {Dror Bar-Natan},
title = {{Khovanov's homology for tangles and cobordisms}},
volume = {9},
journal = {Geometry \& Topology},
number = {3},
publisher = {MSP},
pages = {1443 -- 1499},
year = {2005},
doi = {10.2140/gt.2005.9.1443},
URL = {https://doi.org/10.2140/gt.2005.9.1443}
}

@article{bw-categorification, title={{Categorification of the Colored Jones Polynomial and Rasmussen Invariant of Links}}, volume={60}, DOI={10.4153/CJM-2008-053-1}, number={6}, journal={Canadian Journal of Mathematics}, author={Beliakova, Anna and Wehrli, Stephan}, year={2008}, pages={1240–1266}}

\end{document}